\documentclass[a4paper]{article}
\usepackage[top=30mm,bottom=30mm,left=25mm,right=25mm]{geometry}
\usepackage{amsmath,amssymb,amsthm,ascmac,enumerate,mathrsfs}
\usepackage[dvipdfmx]{graphicx}

\theoremstyle{plain}
\newtheorem{pretheo}{Theorem}[section]
\newtheorem{preassu}[pretheo]{Assumption}
\newtheorem{precoro}[pretheo]{Corollary}
\newtheorem{predefi}[pretheo]{Definition}
\newtheorem{preexam}[pretheo]{Example}
\newtheorem{prelemm}[pretheo]{Lemma}
\newtheorem{preprop}[pretheo]{Proposition}
\newtheorem{prerema}[pretheo]{Remark}

\newenvironment{theo}{\begin{pretheo}}{\end{pretheo}}

\newenvironment{coro}{\begin{precoro}}{\end{precoro}}

\newenvironment{lemm}{\begin{prelemm}}{\end{prelemm}}
\newenvironment{prop}{\begin{preprop}}{\end{preprop}}
\newenvironment{rema}{\begin{prerema}\rm}{\end{prerema}}

\DeclareMathOperator{\di}{div}

\DeclareMathOperator{\Di}{Div}

\newcommand{\pa}{\partial}
\newcommand{\pr}{\prime}

\newcommand{\wh}[1]{\widehat{#1}}
\newcommand{\wit}[1]{\widetilde{#1}}

\newcommand{\BC}{\mathbf{C}}

\newcommand{\BN}{\mathbf{N}}

\newcommand{\BR}{\mathbf{R}}

\newcommand{\CE}{\mathcal{E}}
\newcommand{\CF}{\mathcal{F}}

\newcommand{\CL}{\mathcal{L}}
\newcommand{\CM}{\mathcal{M}}

\newcommand{\CP}{\mathcal{P}}

\newcommand{\CS}{\mathcal{S}}

\newcommand{\CV}{\mathcal{V}}

\newcommand{\CX}{\mathcal{X}}
\newcommand{\CY}{\mathcal{Y}}

\newcommand{\bdry}{\mathbf{R}_{0}^{N}}

\newcommand{\lhs}{\mathbf{R}_{-}^{N}}
\newcommand{\uhs}{\mathbf{R}_{+}^{N}}
\newcommand{\tws}{\mathbf{R}^{N-1}}
\newcommand{\ws}{\mathbf{R}^{N}}

\newcommand{\al}{\alpha}
\newcommand{\ga}{\gamma}
\newcommand{\de}{\delta}
\newcommand{\ep}{\varepsilon}
\newcommand{\te}{\theta}
\newcommand{\ka}{\kappa}
\newcommand{\la}{\lambda}
\newcommand{\si}{\sigma}
\newcommand{\ph}{\varphi}

\newcommand{\Ga}{\Gamma}
\newcommand{\De}{\Delta}
\newcommand{\Te}{\Theta}
\newcommand{\La}{\Lambda}
\newcommand{\Si}{\Sigma}
\newcommand{\Om}{\Omega}

\newcommand{\gra}{c_{g}}
\newcommand{\sur}{c_{\sigma}}
\newcommand{\wt}{\tilde{t}}

\begin{document}

\title
	{On decay properties of solutions to the Stokes equations with surface tension and gravity in the half space}






\author
{
Hirokazu Saito
\footnote{
Partially supported by JSPS Japanese-Germann Graduate externship at Waseda University.} \\
\small{Department of Pure and Applied Mathematics, Graduate School of Fundamental Science and Engineering,} \\
\small{Waseda University, Okubo 3-4-1, Shinjuku-ku, Tokyo 169-8555, Japan} \\
\small{E-mail: hsaito@aoni.waseda.jp} \\ \\
Yoshihiro Shibata
\footnote{Partially supported by 
JSPS Grant-in-aid for Scientific Research (S) $\#$ 24224004
and JSPS Japanese-Germann Graduate externship at Waseda University.} \\
\small{Department of Mathematics and Research Institute of 
Science and Engineering, Waseda University,} \\
\small{Okubo 3-4-1, Shinjuku-ku, Tokyo, 169-8555, Japan} \\
\small{E-mail: yshibata@waseda.jp}
}
\date{}

\maketitle

\begin{abstract}
In this paper, we proved decay properties of solutions to the Stokes equations
with surface tension and gravity in the half space
$\uhs=\{(x^\pr,x_N)\mid x^\pr\in\tws,\enskip x_N>0\}$ $(N\geq 2)$.
In order to prove the decay properties,
we first show that the zero points $\la_\pm$ of Lopatinskii determinant
for some resolvent problem associated with the Stokes equations
have the asymptotics: $\la_\pm=\pm i\gra^{1/2}|\xi^\pr|^{1/2} -2|\xi^\pr|^2+O(|\xi^\pr|^{5/2})$
as $|\xi^\pr|\to0$, where $\gra>0$ is the gravitational acceleration
and $\xi^\pr\in\tws$ is the tangential variable in the Fourier space.
We next shift the integral path in the representation formula
of the Stokes semi-group to the complex left half-plane by Cauchy's integral theorem,
and then it is decomposed into closed curves enclosing $\la_\pm$ and the remainder part.
We finally see, by the residue theorem, that
the low frequency part of the solution to the Stokes equations behaves like
the convolution of the $(N-1)$-dimensional heat kernel
and $\CF_{\xi^\pr}^{-1}[e^{\pm i\gra^{1/2}|\xi^\pr|^{1/2}t}](x^\pr)$ formally,
where $\CF_{\xi^\pr}^{-1}$ is the inverse Fourier transform with respect to $\xi^\pr$. 
However, main task in our approach is to show that
the remainder part in the above decomposition decay faster than the residue part.
\end{abstract}

\renewcommand{\thefootnote}{\fnsymbol{footnote}}


\numberwithin{equation}{section}

\section{Introduction and main results}\label{sec:intro}
Let $\uhs$ and $\bdry$ $(N\geq 2)$ be the half space and its boundary, that is,
\begin{equation*}
\uhs=\{(x^\pr,x_N)\mid x^\pr\in\tws,\enskip x_N>0\},\quad
\bdry=\{(x^\pr,x_N)\mid x^\pr\in\tws,\enskip x_N=0\}.
\end{equation*}
In this paper, we consider the following Stokes equations with the surface tension
and gravity in the half space $\uhs$:
\begin{equation}\label{SP}
\left\{\begin{aligned}
\pa_t U-\Di S(U,\Theta)&=0,&\di U&=0&&\text{in $\uhs$, $t>0$},\\
\pa_t H+U_N&=0&& &&\text{on $\bdry$, $t>0$},\\
S(U,\Te)\nu+(\gra-\sur\De^\pr)H\nu&=0&& &&\text{on $\bdry$, $t>0$},\\
U|_{t=0}=f\quad&\text{in $\uhs$,}&H|_{t=0}&=d&&\text{on $\tws$.}
\end{aligned}\right.
\end{equation}

Here the unknowns $U=(U_1(x,t),\dots,U_N(x,t))^T\footnote[3]{$M^T$ describes the transposed $M$.}$
and $\Theta=\Theta(x,t)$ are the velocity field
and the pressure at $(x,t)\in\uhs\times(0,\infty)$, respectively,
and also $H=H(x^\pr,t)$ is  the height function at $(x^\pr,t)\in\tws\times(0,\infty)$.
The operators $\di$ and $\De^\pr$ are defined by
\begin{equation*}
\di U=\sum_{j=1}^N D_j U_j,\quad
\De^\pr H=\sum_{j=1}^{N-1}D_j^2 H\quad(D_j=\frac{\pa}{\pa x_j})
\end{equation*}
for any $N$-component vector function $U$ and scalar function $H$.
$S(U,\Te)=-\Te I+D(U)$ is the stress tensor,
where $I$ is the $N\times N$ identity matrix and $D(U)$ is the doubled 
strain tensor whose $(i,j)$ component is $D_{ij}(U)=D_i U_j+D_j U_i$.
Moreover, $\Di S(U,\Te)$ is the $N$-component vector function with the $i\text{th}$ component:
\begin{equation*}
\sum_{j=1}^N D_j(D_j U_i+D_i U_j-\delta_{ij}\Te)=\De U_i+D_i \di U-D_{i}\Te.
\end{equation*}
Let $\nu=(0,\dots,0,-1)^{T}$ be the unit outer normal to
$\bdry$, and then 
%
%
\begin{equation*}
i^{\rm th}\text{ component of }S(U,\Te)\nu=
\left\{\begin{aligned}
&-(D_N U_i+D_i U_N)&&(i=1,\dots,N-1), \\
&-2D_N U_N+\Te&&(i=N).
\end{aligned}\right.
\end{equation*}
The parameters $\gra>0$ and $\sur>0$ describe the gravitational acceleration
and the surface tension coefficient, respectively,
and the functions $f=(f_1(x),\dots,f_N(x))^T$ and $d=d(x^\pr)$ are given initial data.

The equations \eqref{SP} arise in the study of a free boundary problem
for the incompressible Navier-Stokes equations.
The free boundary problem is mathematically to find
a $N$-component vector function $u=(u_1(x,t),\dots,u_N(x,t))^{T}$,
a scalar function $\te=\te(x,t)$, and a free boundary
$\Ga(t)=\{(x^\pr,x_N)\mid x^\pr\in\tws,\ x_N=h(x^\pr,t)\}$
satisfying the following Navier-Stokes equations:
\begin{equation}\label{NS}
\left\{\begin{aligned}
\rho(\pa_t u+u\cdot\nabla u)-\Di S(u,\te)&=-\rho\gra\nabla x_{N},
&&\di u=0&&\text{in $\Om(t)$, $t>0$},\\
\pa_t h+u^\pr\cdot\nabla^\pr h-u_N&=0&& &&\text{on $\Ga(t)$, $t>0$},\\
S(u,\te)\nu_t&=\sur\ka\nu_{t}&& &&\text{on $\Ga(t)$, $t>0$},\\
u|_{t=0}&=u_0&& &&\text{in $\Om(0)$}, \\
h|_{t=0}&=h_0 && && \text{on $\tws$}.
\end{aligned}\right.
\end{equation}

Here $\Om(t)=\{(x^{\prime},x_{N})\mid x^{\pr}\in\tws,\ x_{N}<h(x^{\pr},t)\}$,
and $\Om(0)$ is a given initial domain;
$\rho$ is a positive constant describing the density of the fluid;
$\ka=\ka(x,t)$ is the mean curvature of $\Ga(t)$,
and $\nu_t$ is the unit outer normal to $\Ga(t)$;
$u\cdot\nabla u=\sum_{j=1}^N u_j D_j u$,
and $u^\pr\cdot\nabla^\pr h=\sum_{j=1}^{N-1}u_j D_j h$.

A problem is called the finite depth one if the equations (\ref{NS}) is considered in 
$\Om(t)=\{(x^\pr,x_N)\mid x^\pr\in\tws,\ -b<x_N<h(x^\pr,t)\}$ for some constant $b>0$ 
with  Dirichlet boundary condition on the lower boundary: 
$\Ga_b=\{(x^\pr,x_N)\mid x^\pr\in\tws,\ x_N=-b\}$.
There are several results for the finite depth problem.
In fact, Beale \cite{Beale} proved the local well-posedness in  the case of $\sur=0$ and $\gra>0$, 
and also \cite{Beale2} proved the global well-posedness for small initial data when $\sur>0$ and $\gra>0$.
Beale and Nishida \cite{B-N} proved decay properties of the solution obtained in \cite{Beale2},
but the paper is just survey.
We can find the detailed proof in Hataya \cite{Hataya}.
Tani and Tanaka \cite{T-T} also treated both case of $\sur=0$ and $\sur>0$ under the condition $\gra>0$.
Along with these results, we refer to Allain \cite{Allain}, Hataya and Kawashima \cite{H-K},
and Bae \cite{Bae}.
Note that they treated the problem in the $L_2\text{-}L_2$ framework, that is,
their classes of solutions are contained in the space-time $L_2$ space,
and their methods are based on the Hilbert space structure. 
Thus, their methods do not work in  general Banach spaces. 
From this viewpoint, we need completely different techniques since our aim is to treat 
(\ref{NS}) in the  $L_p\text{-}L_q$ framework.

The study of free boundary problems with surface tension and gravity
in the $L_p$-$L_q$ maximal regularity class were started by Shibata and Shimizu \cite{S-S:Rep}.
We especially note that Abels \cite{Abels} proved the local well-posedness
of the finite depth problem with $p=q>N$, $\sur=0$, and $\gra>0$.
In the case of the $L_p\text{-}L_q$ framework,
Shibata \cite{S1} proved the local well-posedness of 
free boundary problems
for the Navier-Stokes equations with $\sur=\gra=0$ 
in general unbounded domains
containing the finite depth problem,
where $p$ and $q$ are exponents satisfying the conditions: 
$1<p,q<\infty$ and $2/p+N/q<1$.

Concerning \eqref{NS}, under some smallness condition of initial data,
Pr\"{u}ss and Simonett \cite{P-S1} showed the local well-posedness of
the two-phase problem containing \eqref{NS} with $\sur>0$ and $\gra=0$,
and also \cite{P-S2} and \cite{P-S3} proved the local well-posedness of 
the case where $\gra>0$ and $\sur>0$.
Recently, there are two papers due to
Shibata and Shimizu \cite{S-S:half,S-S:model},
which treat the linearized problem of (\ref{NS}) and some resolvent problem.
But all the papers 
do not have any results about decay properties of solutions
for the linearized problem of (\ref{NS}).
In the present paper, we show decay properties of solutions to (\ref{SP}) 
as the first step to prove the global well-posedness of (\ref{NS}).

Now we shall state our main results. 
For this purpose, we introduce some symbols and function spaces.
For any domain $\Om$ in $\ws$, positive integer $m$, and $1\leq q\leq\infty$, 
$L_{q}(\Om)$ and $W_{q}^{m}(\Om)$ denote the usual Lebesgue and Sobolev spaces 
with $\|\cdot\|_{L_{q}(\Om)}$ and $\|\cdot\|_{W_{q}^{m}(\Om)}$, respectively,
and we set $W_{q}^{0}(\Om)=L_{q}(\Om)$.
Let $\BN$ be the set of all natural numbers and $\BN_0=\BN\cup\{0\}$, 
and let $\BC$ be the set of all complex numbers.
For differentiations, we use the symbols $D_x^\al$ and $D_{\xi^\pr}^{\beta^\pr}$ defined by 
\begin{align*}
D_{x}^{\alpha}f(x_{1},\dots,x_{N})
&=\frac{\pa^{|\al|}}{\pa x_1^{\al_{1}}\dots\pa x_N^{\al_N}}f(x_1,\dots,x_N)
=D_1^{\al_1}\dots D_N^{\al_N}f(x_1,\dots,x_N), \\
D_{\xi^\pr}^{\beta^\pr}g(\xi_1,\dots,\xi_{N-1})
&=\frac{\pa^{|\beta^\pr|}}{\pa\xi_1^{\beta_1}\dots \pa\xi_{N-1}^{\beta^{N-1}}}g(\xi_{1},\dots,\xi_{N-1})
=D_{1}^{\beta_1}\dots D_{N-1}^{\beta_{N-1}}g(\xi_1,\dots,\xi_{N-1}),
\end{align*}
where $\al=(\al_1,\dots,\al_N)\in\BN_0^N$ and $\beta^\pr=(\beta_1,\dots,\beta_{N-1})\in\BN_0^{N-1}$.
In addition, for any vector functions $u(x)=(u_1(x),\dots,u_N(x))^T$, 
$D_x^\al u(x)$ is given by $D_x^\al u(x)=(D_x^\al u_1(x),\dots,D_x^\al u_N(x))^T$,
and also 
\begin{equation*}
\nabla u=\{D_i u_j\mid i,j=1,\dots,N\},\quad
\nabla^2 u=\{D_i D_j u_k\mid i,j,k=1,\dots,N\}.
\end{equation*}
%
%
%
%
Let $X$ and $Y$ be Banach spaces with $\|\cdot\|_X$ and $\|\cdot\|_Y$, respectively,
and then $\CL(X,Y)$ denotes the set of all bounded linear operators
from $X$ to $Y$, and set $\CL(X)=\CL(X,X)$.
For $m\in \BN_{0}$ and an interval $I$ in $\BR$,
$C^{m}(I,X)$ is the set of all $X$-valued $C^m$-functions defined on $I$.
Let $X^m$ be the $m$-product space of $X$ with $m\in\BN$, 
while we use the symbol $\|\cdot\|_{X}$ to denote its norm for short, that is,
\begin{equation*}
\begin{aligned}
&\|u\|_X=\sum_{j=1}^m\|u_j\|_X&&\text{for $u=(u_1,\dots,u_m)\in X^m$}.
\end{aligned}
\end{equation*}
For $1<q<\infty$, non-integer $s>0$, and $m\in\BN$,
$W_q^s(\BR^m)$ denotes the Slobodeckii spaces defined by 
\begin{align*}
W_q^s(\BR^m)&=\{u\in W_q^{[s]}(\BR^m)\mid\|u\|_{W_q^s(\BR^m)}<\infty\}, \\
\|u\|_{W_q^s(\BR^m)}&=\|u\|_{W_q^{[s]}(\BR^m)}+\sum_{|\al|=[s]}\Big(\int_{\BR^m}\int_{\BR^m}
\frac{|D_x^\al u(x)-D_y^\al u(y)|^{q}}{|x-y|^{m+(s-[s])q}}\,dxdy\Big)^{1/q},
\end{align*}
where $[s]$ is the largest integer lower than $s$.
For any vector function $u=(u_1,\dots,u_N)^T$ and 
$v=(v_1,\dots,v_N)^T$ defined on $\uhs$, we set  
\begin{equation*}
(u,v)_{\uhs}=\int_{\uhs}u(x)\cdot v(x)\,dx
=\sum_{j=1}^{N}\int_{\uhs}u_{j}(x)v_{j}(x)\,dx.
\end{equation*}
The letter $C$ denotes a generic constant and 
$C(a,b,c,\dots)$ a generic constant depending on the quantities $a,b,c,\dots$.
The value of $C$ and $C(a,b,c,\dots)$ may change from line to line.

%
%
Let $\wh{W}_{q}^{1}(\uhs)$ be the homogeneous spaces of order $1$ defined by 
$\wh{W}_{q}^{1}(\uhs)=\{\te\in L_{q,{\rm loc}}(\uhs)\mid\nabla\te\in L_{q}(\uhs)^N\}$.
In addition, we set 
$\wh{W}_{q,0}^1(\uhs)=\{\te\in\wh{W}_q^1(\uhs)\mid\te|_{\bdry}=0\}$ and
$W_{q,0}^{1}(\uhs)=\{\te\in W_q^1(\uhs)\mid\te|_{\bdry}=0\}$.
As was seen in \cite[Theorem A.3]{Shibata1}, 
$W_{q,0}^1(\uhs)$ is dense in $\wh{W}_{q,0}^1(\uhs)$ with the gradient norm $\|\nabla\cdot\|_{L_q(\uhs)}$.
Then the second solenoidal space $J_{q}(\uhs)$ is defined by
\begin{equation*}
J_{q}(\uhs)=\{f\in L_q(\uhs)^N\mid(f,\nabla\ph)_{\uhs}=0
\enskip\text{for any $\ph\in\wh{W}_{q^\pr,0}^1(\uhs)$}\},
\end{equation*}
where $1/q+1/q^\prime=1$. For simplicity, we set
\begin{align}\label{X_q}
&X_q=J_q(\uhs)\times W_q^{2-1/q}(\tws),\quad X_q^0=L_q(\uhs)\times L_q(\tws), \notag \\
&X_q^i=L_q(\uhs)\times W_q^{i-1/q}(\tws)\quad(i=1,2),	
\end{align}
and let $\CE H$ be the harmonic extension of $H$, that is, 
\begin{equation}\label{ext:H}
	\left\{\begin{aligned}
		\De\CE H&=0 && \text{in $\uhs$}, \\
		\CE H&=H && \text{on $\bdry$}.
	\end{aligned}\right.
\end{equation}
The main results of this paper then is stated as follows:
\begin{theo}\label{theo:main}
Let $1<p<\infty$, $c_g>0$, and $c_\si>0$.
\begin{enumerate}[$(1)$]
\item
For every $t>0$ there exists operators
\begin{equation*}
S(t)\in\CL(X_p^2,W_p^{2}(\uhs)^{N}),
\enskip\Pi(t)\in\CL(X_p^2,\widehat{W}_{p}^{1}(\uhs)),
\enskip T(t)\in\CL(X_p^2,{W}_{p}^{3-1/p}(\tws))
\end{equation*}
such that for $F=(f,d)\in X_p$
\begin{align*}
S(\cdot)F&\in C^1((0,\infty),J_p(\uhs))\cap C^0((0,\infty),W_p^2(\uhs)^N), \\
\Pi(\cdot)F&\in C^0((0,\infty),\widehat{W}_p^1(\uhs)), \\
T(\cdot)F&\in C^1((0,\infty),W_p^{2-1/p}(\tws))\cap C^0((0,\infty),W_p^{3-1/p}(\tws)),
\end{align*}
and that $(U,\Te,H)=(S(t)F,\Pi(t)F,T(t)F)$ solves uniquely $(\ref{SP})$ with
\begin{equation*}
\lim_{t\to 0}\|(U(t),H(t))-(f,d)\|_{X_p}=0.
\end{equation*}
\item
Let $1\leq r\leq 2\leq q\leq\infty$ and $F=(f,d)\in X_r^0\cap X_p^2$.
The operators, obtained in $(1)$, then are decomposed into
\begin{align}\label{dfn:ops}
&S(t)F=S_0(t)F+S_\infty(t)F+R(t)f,\notag \\
&\Pi(t)F=\Pi_0(t)F+\Pi_\infty(t)F+P(t)f, \notag\\
&T(t)F=T_0(t)F+T_\infty(t)F,
\end{align}
which satisfy the estimates as follows:
For $k=1,2$, $\ell=0,1,2$, and $t\geq 1$
\begin{align}\label{140917_1}
\|(S_0(t)F,\pa_t\CE(T_0(t)F))\|_{L_q(\uhs)}&\leq C(t+1)^{-m(q,r)}\|F\|_{X_r^0}\quad\text{if $(q,r)\neq(2,2)$}, \notag \\
\|\nabla^k  S_0(t)F\|_{L_q(\uhs)}&\leq C(t+1)^{-n(q,r)-k/8}\|F\|_{X_r^0}, \notag \\
\|(\pa_t S_0(t)F,\nabla\Pi_0(t)F)\|_{L_q(\uhs)}&\leq C(t+1)^{-m(q,r)-1/4}\|F\|_{X_r^0}, \notag \\
\|\nabla^k\pa_t\CE(T_0(t)F)\|_{L_q(\uhs)}&\leq C(t+1)^{-m(q,r)-k/2}\|F\|_{X_r^0}, \notag \\
\|\nabla^{1+\ell}\CE(T_0(t)F)\|_{L_q(\uhs)}&\leq C(t+1)^{-m(q,r)-1/4-\ell/2}\|F\|_{X_r^0}
\end{align}
with some positive constant $C$, where we have set
\begin{align*}
m(q,r)&=\frac{N-1}{2}\left(\frac{1}{r}-\frac{1}{q}\right)
+\frac{1}{2}\left(\frac{1}{2}-\frac{1}{q}\right), \\
n(q,r)&=\frac{N-1}{2}\left(\frac{1}{r}-\frac{1}{q}\right)
+\min\left\{\frac{1}{2}\left(\frac{1}{r}-\frac{1}{q}\right),\frac{1}{8}\left(2-\frac{1}{q}\right)\right\}.
\end{align*}
In addition, there exist positive constants $\de$ and $C$ such that for $t\geq 1$
\begin{multline}\label{140917_2}
\|(\pa_t S_\infty(t)F,\nabla\Pi_\infty(t)F)\|_{L_p(\uhs)}\\
+\|(S_\infty(t)F,\pa_t\CE(T_\infty(t)F)),\nabla \CE(T_\infty(t)F)))\|_{W_p^2(\uhs)}
\leq Ce^{-\de t}\|F\|_{X_p^2}.
\end{multline}
Finally, for $t\geq 1$ and $\ell=0,1,2$,
\begin{align}\label{140917_3}
&\|\nabla^\ell R(t)f\|_{L_p(\uhs)}\leq C(t+1)^{-\ell/2}\|f\|_{L_p(\uhs)}, \notag \\
&\|(\pa_t R(t)f,\nabla P(t)f)\|_{L_p(\uhs)}\leq C(t+1)^{-1}\|f\|_{L_p(\uhs)}.
\end{align}
\end{enumerate}
\end{theo}
%
%
%
%
This paper consist of five sections.
In the next section, we introduce some symbols and lemmas,
and also consider some resolvent problem associated with (\ref{SP}) with $\gra=\sur=0$.
In Section 3, we construct the operators $S(t),\Pi(t)$, and $T(t)$, and also give the decompositions \eqref{dfn:ops}.
Finally, Theorem \ref{theo:main} (2) is proved in Section 4 and Section 5.


\section{Preliminaries}\label{sec:pre}
We first give some symbols used throughout this paper. Set
\begin{equation*}
\Si_\ep=\{\la\in\BC\mid|\arg{\la}|<\pi-\ep,\ \la\neq0\},\quad
\Si_{\ep,\la_0}=\{\la\in\Si_\ep\mid|\la|\geq\la_{0}\}
\end{equation*}
for any $0<\ep<\pi/2$ and $\la_0>0$.
We then define
%
%
\begin{align}\label{0921_1}
&A=|\xi^\pr|,\quad B=\sqrt{\la+|\xi^\pr|^2}\quad
({\rm Re}B\geq0),\quad \CM(a)=\frac{e^{-Ba}-e^{-Aa}}{B-A}, \notag \\
&D(A,B)=B^{3}+AB^{2}+3A^{2}B-A^{3}, \notag \\
&L(A,B)=(B-A)D(A,B)+A(\gra+\sur A^{2})
\end{align}
for $\xi^\pr=(\xi_{1},\dots,\xi_{N-1})\in\tws$,
$\la\in\Si_\ep$, and $a>0$. Especially, we have, for $\ell=1,2$, 
\begin{align}\label{deriv:M}
&\frac{\pa^\ell}{\pa a^{\ell}}\CM(a)
=(-1)^{\ell}\left((B+A)^{\ell-1}e^{-Ba}+A^\ell\CM(a)\right), \notag \\
&\CM(a)=-a\int_0^1 e^{-(B\te+A(1-\te))a}\,d\te.
\end{align}
The following lemma was proved 
in \cite[Lemma 5.2, Lemma 5.3, Lemma 7.2]{S-S:model}.
\begin{lemm}\label{lem:symbol}
Let $0<\ep<\pi/2$, $s\in\BR$, $a>0$, and $\al^\pr\in\BN_{0}^{N-1}$.
\begin{enumerate}[$(1)$]
\item
There holds the estimate
\begin{equation*}
b_\ep(|\la|^{\frac{1}{2}}+A)\leq {\rm Re}B\leq |B|\leq (|\la|^{\frac{1}{2}}+A)
\end{equation*}
for any $(\xi^\pr,\la)\in\tws\times\Si_\ep$ with $b_\ep=(1/\sqrt{2})\{\sin(\ep/2)\}^{3/2}$.
\item
There exist a positive constant $C=C(\ep,s,\al^\pr)$ such that
for any $(\xi^\pr,\la)\in(\tws\setminus\{0\})\times\Si_\ep$
\begin{align*}
&|D_{\xi^\pr}^{\al^\pr}A^{s}|\leq CA^{s-|\al^{\pr}|},
\quad |D_{\xi^{\prime}}^{\alpha^{\prime}} e^{-Aa}|\leq CA^{-|\alpha^{\prime}|}e^{-(A/2)a},
\quad|D_{\xi^{\prime}}^{\alpha^{\prime}}B^{s}|\leq C(|\lambda|^{\frac{1}{2}}+A)^{s-|\alpha^{\prime}|}, \\
&|D_{\xi^{\prime}}^{\alpha^{\prime}}e^{-Ba}|\leq C(|\lambda|+A)^{-|\alpha^{\prime}|}e^{-(b_{\varepsilon}/8)(|\lambda|^{1/2}+A)a},
\quad|D_{\xi^{\prime}}^{\alpha^{\prime}}D(A,B)^{s}|\leq C(|\lambda|^{\frac{1}{2}}+A)^{3s}A^{-|\alpha^{\prime}|} \\
&|D_{\xi^{\prime}}^{\alpha^{\prime}}\CM(a)|\leq CA^{-1-|\alpha^{\prime}|}e^{-(b_{\varepsilon}/8)Aa},
\quad|D_{\xi^{\prime}}^{\alpha^{\prime}}\CM(a)|\leq C|\lambda|^{-\frac{1}{2}}A^{-|\alpha^{\prime}|}e^{-(b_{\varepsilon}/8)Aa}.
\end{align*}
\item
There exist positive constants $\lambda_{0}=\lambda_{0}(\varepsilon)\geq1$
and $C=C(\varepsilon,\lambda_0,\alpha^\prime)$ such that
for any $(\xi^{\prime},\lambda)\in(\mathbf{R}^{N-1}\setminus\{0\})\times\Sigma_{\varepsilon,\lambda_{0}}$
\begin{equation*}
|D_{\xi^{\prime}}^{\alpha^{\prime}}L(A,B)^{-1}|\leq 
C\{|\lambda|(|\lambda|^{\frac{1}{2}}+A)^{2}+A(\gra+\sur A^{2})\}^{-1}A^{-|\alpha^{\prime}|}.
\end{equation*}
\end{enumerate}
\end{lemm}
Let $f(x)$ and $g(\xi)$ be functions defined on $\ws$, 
and then the Fourier transform of $f(x)$ and the inverse Fourier 
transform of $g(\xi)$ are defined by
\begin{equation*}
\CF[f](\xi)=\int_{\mathbf{R}^{N}}e^{-ix\cdot\xi}f(x)\,dx,\quad
\CF_{\xi}^{-1}[g](x)=\frac{1}{(2\pi)^{N}}\int_{\mathbf{R}^{N}}e^{ix\cdot\xi}g(\xi)\,d\xi.
\end{equation*}
We also define the partial Fourier transform of $f(x)$ 
and the inverse partial Fourier transform of $g(\xi)$ with 
respect to tangential variables $x' = (x_1, \ldots, x_{N-1})$ and
its dual variable $\xi' = (\xi_1, \ldots, \xi_{N-1})$ by 
\begin{align*}
&\widehat{f}(\xi^{\prime},x_{N})	=\int_{\tws}e^{-ix^{\prime}\cdot\xi^{\prime}}f(x^{\prime},x_{N})\,dx^{\prime}, \\
&\CF_{\xi^{\prime}}^{-1}[g](x^{\prime},\xi_N)=\frac{1}{(2\pi)^{N-1}}
\int_{\tws}e^{ix^{\prime}\cdot\xi^{\prime}}g(\xi^{\prime},\xi_{N})\,d\xi^{\prime}.
\end{align*}

Next we consider the following resolvent problem:
\begin{equation}\label{eq:1}
\left\{
\begin{aligned}
\lambda w-{\rm Div}S(w,p)&=f,
&&{\rm div}w=0
&&\text{in $\uhs$}, \\
S(w,p)\nu&=0
&&
&&\text{on $\bdry$}.
\end{aligned}
\right.
\end{equation}
\begin{lemm}\label{lem:1}
Let $0<\varepsilon<\pi/2$, $1<q<\infty$, $\lambda\in\Sigma_{\varepsilon}$,
and $f\in L_{q}(\uhs)^{N}$.
Then the equations $(\ref{eq:1})$ admits a unique solution
$(w,p)\in W_{q}^{2}(\mathbf{R}_{+}^{N})^{N} 
\times\widehat{W}_{q}^{1}(\mathbf{R}_{+}^{N})$
possessing the estimate:
\begin{equation*}
\|(\lambda w,\lambda^{1/2}\nabla w,\nabla^{2}w,\nabla p)\|_{L_{q}(\mathbf{R}_{+}^{N})}
\leq C\|f\|_{L_{q}(\mathbf{R}_{+}^{N})}
\end{equation*}
with some positive constant $C=C(\ep,q,N)$.
In addition, $\widehat{w}_N(\xi', 0, \lambda)$ is given by 
\begin{align}
\widehat{w}_{N}(\xi^{\prime},0, \lambda)
=&\sum_{k=1}^{N-1}\frac{i\xi_{k}(B-A)}{D(A,B)}
\int_{0}^{\infty}e^{-By_{N}}\widehat{f}_{k}(\xi^{\prime},y_{N})\,dy_{N} \notag \\
&+\frac{A(B+A)}{D(A,B)}\int_{0}^{\infty}
e^{-By_{N}}\widehat{f}_{N}(\xi^{\prime},y_{N})\,dy_{N} \notag \\ 
&-\sum_{k=1}^{N-1}\frac{i\xi_{k}(B^{2}+A^{2})}{D(A,B)}
\int_{0}^{\infty}\mathcal{M}(y_{N})\widehat{f}_{k}(\xi^{\prime},y_{N})\,dy_{N} \notag \\
&-\frac{A(B^{2}+A^{2})}{D(A,B)}\int_{0}^{\infty}
\mathcal{M}(y_{N})\widehat{f}_{N}(\xi^{\prime},y_{N})\,dy_{N} 
\label{w:traceB} \displaybreak[0] \\
=&\sum_{k=1}^{N-1}\frac{i\xi_{k}(B-A)}{D(A,B)}\int_{0}^{\infty} 
e^{-Ay_{N}}\widehat{f}_{k}(\xi^{\prime},y_{N})\,dy_{N} \notag \\ 
&+\frac{A(B+A)}{D(A,B)}\int_{0}^{\infty}e^{-Ay_{N}}
\widehat{f}_{N}(\xi^{\prime},y_{N})\,dy_{N} \notag \\
&
-\sum_{k=1}^{N-1}\frac{2i\xi_{k}AB}{D(A,B)}
\int_{0}^{\infty}\mathcal{M}(y_{N})
\widehat{f}_{k}(\xi^{\prime},y_{N})\,dy_{N} \notag \\
&-\frac{2A^{3}}{D(A,B)}\int_{0}^{\infty}
\mathcal{M}(y_{N})\widehat{f}_{N}(\xi^{\prime},y_{N})\,dy_{N}. \label{w:traceA}
\end{align}
\end{lemm}
\begin{proof}
The lemma was proved by Shibata and Shimizu \cite[Theorem 4.1]{S-S:Neumann}
except for \eqref{w:traceB} and \eqref{w:traceA}, so that 
we prove \eqref{w:traceB} and \eqref{w:traceA} here. 

Given functions $g(x)$ defined on $\uhs$, 
we set their even extensions $g^e(x)$ and odd extensions $g^o(x)$ as 
\begin{equation}\label{evenodd}
\begin{aligned}
&g^{e}(x)=
\left\{\begin{aligned}
&g(x^{\prime},x_{N})&&\text{in $\uhs$}, \\
&g(x^{\prime},-x_{N})&&\text{in $\lhs$},
\end{aligned}\right.
&&g^{o}(x)=
\left\{\begin{aligned}
&g(x^{\prime},x_{N})&&\text{in $\uhs$}, \\
&-g(x^{\prime},-x_{N})
&&\text{in $\lhs$},
\end{aligned}\right.
\end{aligned}
\end{equation}
where $\lhs=\{(x^\pr,x_N)\mid x^\pr\in\tws,x_N<0\}$.
In addition, given the right member $f = (f_1, \ldots, f_N)^T$ of \eqref{eq:1},
we set $Ef=(f_{1}^{o},\dots,f_{N-1}^{o},f_{N}^{e})^T$. 
Let $(w^{1},p^{1})$ be the solution to the following 
resolvent problem: 
\begin{equation*}
\lambda w^{1}-{\rm Div}S(w^{1},p^{1})=Ef,\quad
{\rm div}w^{1}=0\quad
\text{in $\ws$}.
\end{equation*}
We then have the following solution formulas (cf. \cite[Section 3]{S-S:model}):
\begin{align}
w_{j}^{1}(x,\lambda)=&\CF_{\xi}^{-1}\left[\frac{\widehat{(Ef)}_{j}(\xi)}{\lambda+|\xi|^{2}}\right](x)  \notag \\
&-\sum_{k=1}^{N}\CF_{\xi}^{-1}\left[\frac{\xi_{j}\xi_{k}}{|\xi|^{2}(\lambda+|\xi|^{2})}\widehat{(Ef)}_{k}(\xi)\right](x)
\quad(j=1,\dots,N), \notag \\
p^{1}(x,\lambda)=&-\CF_{\xi}^{-1}\left[\frac{i\xi}{|\xi|^{2}}\cdot\widehat{Ef}(\xi)\right](x). \label{sol:whole}
\end{align}
As was seen in \cite[Section 4]{S-S:Neumann}, we have,
by the definition of the extension $E$, 
\begin{equation}\label{0401_1}
D_{N}w_{N}^{1}(x^{\prime},0,\lambda)=0,\quad p^{1}(x^{\prime},0,\lambda)=0.
\end{equation}

Next we give the exact formulas of $\widehat{w}_{N}^{1}(\xi^{\prime},0,\lambda)$
and $\widehat{D_{N}w_{j}^{1}}(\xi^{\prime},0,\lambda)$ for $j=1,\dots,N-1$.
To this end, we use the following lemma which is proved by the residue theorem.
\begin{lemm}\label{lem:residue}
Let $a\in\mathbf{R}\setminus\{0\}$, and let $\xi=(\xi_{1},\dots,\xi_{N})\in\mathbf{R}^{N}$.
Then
\begin{align*}
&\frac{1}{2\pi}\int_{-\infty}^{\infty}
\frac{e^{ia\xi_{N}}}{|\xi|^{2}}\,d\xi_{N}=\frac{e^{-A|a|}}{2A},
\quad\frac{1}{2\pi}\int_{-\infty}^{\infty}
\frac{i\xi_{N}e^{ia\xi_{N}}}{|\xi|^{2}}\,d\xi_{N}
=-{\rm sign}(a)\frac{e^{-A|a|}}{2}, \\
&\frac{1}{2\pi}\int_{-\infty}^{\infty}\frac{e^{ia\xi_{N}}}{\lambda+|\xi|^{2}}\,d\xi_{N}=\frac{e^{-B|a|}}{2B},
\quad\frac{1}{2\pi}
\int_{-\infty}^{\infty}\frac{\xi_{N}e^{ia\xi_{N}}}{\lambda+|\xi|^{2}}d\xi_{N}
={\rm sign}(a)\frac{i}{2}e^{-B|a|}, \\
&\frac{1}{2\pi}\int_{-\infty}^{\infty}\frac{\xi_{N}
e^{ia\xi_{N}}}{|\xi|^{2}(\lambda+|\xi|^{2})}\,d\xi_{N}
={\rm sign}(a)\frac{i}{2\lambda}(e^{-A|a|}-e^{-B|a|}), \\
&\frac{1}{2\pi}\int_{-\infty}^{\infty}
\frac{\xi_{N}^{2}e^{ia\xi_{N}}}{|\xi|^{2}(\lambda+|\xi|^{2})}\,d\xi_{N}
=-\frac{1}{2\lambda}(Ae^{-A|a|}-Be^{-B|a|}),
\end{align*}
where ${\rm sign}(a)$ defined by the formula: ${\rm sign}(a)=1$ when $a>0$ and ${\rm sign}(a)=-1$ when $a<0$.
\end{lemm}
In order to obtain
\begin{align}
\widehat{w}_{N}^{1}(\xi^{\prime},0,\lambda)=&\sum_{k=1}^{N-1}\frac{i\xi_{k}}{\lambda}\int_{0}^{\infty}\left(e^{-Ay_{N}}-e^{-By_{N}}\right)
\widehat{f}_{k}(\xi^{\prime},y_{N})\,dy_{N} \notag \displaybreak[0] \\
&+\int_{0}^{\infty}\frac{e^{-By_{N}}}{B}
\widehat{f}_{N}(\xi^{\prime},y_{N})\,dy_{N} \notag \displaybreak[0] \\
&+\frac{1}{\lambda}\int_{0}^{\infty}
\left(Ae^{-Ay_{N}}-Be^{-By_{N}}\right)
\widehat{f}_{N}(\xi^{\prime},y_{N})\,dy_{N},\notag \displaybreak[0] \\
\widehat{D_{N}w_{j}^{1}}(\xi^{\prime},0,\lambda)
=&-\sum_{k=1}^{N-1}\frac{\xi_{j}\xi_{k}}{\lambda}
\int_{0}^{\infty}\left(e^{-Ay_{N}}-e^{-By_{N}}\right)
\widehat{f}_{k}(\xi^{\prime},y_{N})\,dy_{N} \notag \displaybreak[0] \\
&+\int_{0}^{\infty}e^{-By_{N}}
\widehat{f}_{j}(\xi^{\prime},y_{N})\,dy_{N} \notag \displaybreak[0] \\
&+\frac{i\xi_{j}}{\lambda}
\int_{0}^{\infty}\left(Ae^{-Ay_{N}}-Be^{-By_{N}}\right)
\widehat{f}_{N}(\xi^{\prime},y_{N})\,dy_{N}, \label{0521_1}
\end{align}
we apply the partial Fourier transform with respect to $x'=(x_1,\dots,x_{N-1})$ to (\ref{sol:whole}),
insert the identities in Lemma \ref{lem:residue} into the resultant formula,
and use the formulas: 
\begin{align*}
\mathcal{F}[f_{j}^{o}](\xi)&
=\int_{0}^{\infty}\left(e^{-iy_{N}\xi_{N}}-e^{iy_{N}\xi_{N}}\right)
\widehat{f}_{j}(\xi^{\prime},y_{N})\,dy_{N}\quad(j=1,\dots,N-1), \\
\mathcal{F}[f_{N}^{e}](\xi)&
=\int_{0}^{\infty}\left(e^{-iy_{N}\xi_{N}}+e^{iy_{N}\xi_{N}}\right)
\widehat{f}_{N}(\xi^{\prime},y_{N})\,dy_{N}.
\end{align*}
Here and in the following, $j$ runs from $1$ through $N-1$.
By (\ref{0521_1}) and the fact that $\lambda = B^{2}-A^{2}$ and $e^{-By_{N}}-e^{-Ay_{N}}=(B-A)\CM(y_N)$,
we have
\begin{align}
\widehat{w}_{N}^{1}(\xi^{\prime},0,\lambda)
=&\frac{A}{B(B+A)}\int_{0}^{\infty}
e^{-Bx_{N}}\widehat{f}_{N}(\xi^{\prime},y_{N})\,dy_{N} \notag \displaybreak[0] \\
&-\sum_{k=1}^{N-1}\frac{i\xi_{k}}{B+A}
\int_{0}^{\infty}\mathcal{M}(y_{N})
\widehat{f}_{k}(\xi^{\prime},y_{N})\,dy_{N} \notag \displaybreak[0] \\
&-\frac{A}{B+A}\int_{0}^{\infty}\mathcal{M}(y_{N})
\widehat{f}_{N}(\xi^{\prime},y_{N})\,dy_{N},	\notag \displaybreak[0] \\
\widehat{D_{N}w_{j}^{1}}(\xi^{\prime},0,\lambda)
=&\int_{0}^{\infty}e^{-By_{N}}\widehat{f_{j}}(\xi^{\prime},y_{N})\,dy_{N}
-\frac{i\xi_{j}}{B+A}\int_{0}^{\infty}e^{-By_{N}}
\widehat{f}_{N}(\xi^{\prime},y_{N})\,dy_{N} \notag \displaybreak[0] \\
&+\sum_{k=1}^{N-1}\frac{\xi_{j}\xi_{k}}{B+A}
\int_{0}^{\infty}\mathcal{M}(y_{N})
\widehat{f}_{k}(\xi^{\prime},y_{N})\,dy_{N} \notag \displaybreak[0] \\  
&-\frac{i\xi_{j}A}{B+A}\int_{0}^{\infty}\mathcal{M}(y_{N})
\widehat{f}_{N}(\xi^{\prime},y_{N})\,dy_{N}. \label{w1:trace}
\end{align}

Next we give the exact formula of $\widehat{w}_{N}^{2}(\xi^{\prime},0,\lambda)$.
Setting $w=w^{1}+w^{2}$ and $p=p^{1}+p^{2}$ in (\ref{eq:1}) 
and noting (\ref{0401_1}), we achieve the equations:
\begin{equation*}
\left\{\begin{aligned}
\lambda w^{2}-{\rm Div}S(w^{2},p^{2})&=0,
&&{\rm div}w^{2}=0&&\text{in $\uhs$}, \\
D_{j}w_{N}^{2}+D_{N}w_{j}^{2}&=-h_{j} && 
&&\text{on $\bdry$}, \\
-p^{2}+2D_{N}w_{N}^{2}&=0 && 
&&\text{on $\bdry$}
\end{aligned}\right.
\end{equation*}
with $h_{j}=D_{j}w_{N}^{1}+D_{N}w_{j}^{1}$.
We then obtain the formulas (cf. \cite[Section 4]{S-S:model}):
\begin{align}\label{0918_6}
&w_{N}^{2}(x^\pr,x_N,\lambda)=\mathcal{F}_{\xi^{\prime}}^{-1}[\widehat{w}_{N}^{2}(\xi^{\prime},x_{N},\lambda)](x^{\prime}), \notag\\
&\widehat{w}_{N}^{2}(\xi^{\prime},x_{N},\lambda)=\left(\frac{B-A}{D(A,B)}e^{-Bx_{N}}+\frac{2AB}{D(A,B)}\mathcal{M}(x_{N})\right)
\sum_{j=1}^{N-1}i\xi_{j}\widehat{h}_{j}(\xi^{\prime},0,\lambda).
\end{align}
By (\ref{w1:trace}) and (\ref{0918_6}),
\begin{align*}
\widehat{w}_{N}^{2}(\xi^{\prime},0,\lambda)
=&\frac{B-A}{D(A,B)}\sum_{j=1}^{N-1}i\xi_{j}
\widehat{h}_{j}(\xi^{\prime},0,\lambda) \notag \displaybreak[0] \\ 
=&\sum_{j=1}^{N-1}\frac{i\xi_{j}(B-A)}{D(A,B)}
\left(i\xi_j\widehat{w}_{N}^{1}(\xi^{\prime},0,\lambda)+
\widehat{D_{N}w_{j}^{1}}(\xi^{\prime},0,\lambda)\right) \displaybreak[0] \\
=&\sum_{k=1}^{N-1}\frac{i\xi_{k}(B-A)}{D(A,B)}
\int_{0}^{\infty}e^{-By_{N}}\widehat{f}_{k}(\xi^{\prime},y_{N})\,dy_{N} \notag \displaybreak[0] \\
&+\frac{A^{2}(B-A)^{2}}{B(B+A)D(A,B)}\int_{0}^{\infty}e^{-By_{N}}
\widehat{f}_{N}(\xi^{\prime},y_{N})\,dy_{N} \displaybreak[0] \\
&+\sum_{k=1}^{N-1}\frac{2i\xi_{k}A^{2}(B-A)}{(B+A)D(A,B)}
\int_{0}^{\infty}\CM(y_{N})\widehat{f}_{k}(\xi^{\prime},y_{N})\,dy_{N} \notag \displaybreak[0] \\
&+\frac{2A^{3}(B-A)}{(B+A)D(A,B)}
\int_{0}^{\infty} \mathcal{M}(y_{N})\widehat{f}_{N}(\xi^{\prime},y_{N})\,dy_{N},
\end{align*}
which combined with (\ref{w1:trace}) furnishes (\ref{w:traceB}), 
because $\widehat{w}_{N}(\xi^{\prime},0,\lambda)
=\widehat{w}_{N}^{1}(\xi^{\prime},0,\lambda)
+\widehat{w}_{N}^{2}(\xi^{\prime},0,\lambda)$.
 
Finally, using the relation: $e^{-By_{N}}=e^{-Ay_{N}}+(B-A)\CM(y_{N})$ 
in (\ref{w:traceB}), we have (\ref{w:traceA}).
This completes the proof of the lemma.  
\end{proof}
\section{Decompositions of operators}\label{sec:decomp}
In this section, we construct the operators $S(t),\Pi(t)$, and $T(t)$ in Theorem \ref{theo:main},
and also show the decompositions \eqref{dfn:ops}.
For this purpose, we first give the exact formulas of the solution $(u,\te,h)$ to
\begin{equation}\label{RP}
\left\{\begin{aligned}
\la u-\Di S(u,\te)&=f && \di u=0 && \text{in $\uhs$}, \\
\la h+u_N&=d && &&\text{on $\bdry$}, \\
S(u,\te)\nu+(\gra-\sur\De^\pr)h\nu&=0 && && \text{on $\bdry$}.
\end{aligned}\right.
\end{equation}
Let $(w,p)$ be the solution to (\ref{eq:1})
and $(v,\pi,h)$ the solution to the equations:
\begin{equation}\label{eq:2}
\left\{\begin{aligned}
\lambda v-\Delta v+\nabla\pi&=0,\quad{\rm div}\,v=0 &&\text{in $\uhs$}, \\
\lambda h+ v_{N}&=-w_{N}+d &&\text{on $\bdry$},\\
S(v,\pi)\nu+(c_{g}-c_{\sigma}\Delta^{\prime})h\nu&=0 &&\text{on $\bdry$}.
\end{aligned}\right.
\end{equation}
Then, $u=v+w$, $\theta=\pi+p$, and $h$ solve \eqref{RP}.
Let $j$ and $k$ run from $1$ through $N-1$ and $J$ from $1$ through $N$, respectively, in the present section.
The exact formulas of $(v,\pi,h)$ are given by
\begin{align*}
v_{J}(x,\lambda)&=\mathcal{F}_{\xi^{\prime}}^{-1}[\widehat{v}_{J}(\xi^{\prime},x_{N},\lambda)](x^{\prime}),
\quad \pi(x,\lambda)=\mathcal{F}_{\xi^{\prime}}^{-1}[\widehat{\pi}(\xi^{\prime},x_{N},\lambda)](x^{\prime}) \notag \\
h(x^{\prime},\lambda)&=\mathcal{F}_{\xi^{\prime}}^{-1}\left[\frac{D(A,B)}{(B+A)L(A,B)}
\left(-\widehat{w}_{N}(\xi^{\prime},0,\lambda)+\widehat{d}(\xi^{\prime})\right)\right](x^{\prime})
\end{align*}
(cf. \cite[Section 7]{S-S:model}), where
\begin{align*}
\widehat{v}_{j}(\xi^{\prime},x_{N},\lambda)&=
\left(-\frac{i\xi_{j}(B-A)}{D(A,B)}e^{-Bx_{N}}+\frac{i\xi_{j}(B^{2}+A^{2})}{D(A,B)}\mathcal{M}(x_{N})\right)
(c_{g}+c_{\sigma}A^{2})\widehat{h}(\xi^{\prime},\lambda), \notag \\
\widehat{v}_{N}(\xi^{\prime},x_{N},\lambda)&=
\left(\frac{A(B+A)}{D(A,B)}e^{-Bx_{N}}-\frac{A(B^{2}+A^{2})}{D(A,B)}\mathcal{M}(x_{N})\right)
(c_{g}+c_{\sigma}A^{2})\widehat{h}(\xi^{\prime},\lambda), \notag \\
\widehat{\pi}(\xi^{\prime},x_{N},\lambda)&=
\frac{(B+A)(B^{2}+A^{2})}{D(A,B)}e^{-Ax_{N}}(c_{g}+c_{\sigma}A^{2})\widehat{h}(\xi^{\prime},\lambda).
\end{align*}
Inserting (\ref{w:traceB}) into $h(x^{\prime},\lambda)$,
we have the decompositions:
\begin{equation*}
\widehat{v}_{J}(\xi^{\prime},x_{N},\lambda)
=\widehat{v}_{J}^{f}(\xi^{\prime},x_{N},\lambda)
+\widehat{v}_{J}^{d}(\xi^{\prime},x_{N},\lambda),
\quad\widehat{\pi}(\xi^{\prime},x_{N},\lambda)
=\widehat{\pi}^{f}(\xi^{\prime},x_{N},\lambda)
+\widehat{\pi}^{d}(\xi^{\prime},x_{N},\lambda),
\end{equation*}
where each term on the right-hand sides is given by
\begin{align}\label{Fsol_eq:2}
\widehat{v}_J^f(\xi^\prime,x_N,\lambda)=
&\sum_{K=1}^N\frac{\CV_{JK}^{BB}(\xi^\prime,\lambda)(\gra+\sur A^2)}{L(A,B)}\int_0^\infty
e^{-B(x_N+y_N)}\widehat{f}_K(\xi^\prime,y_N)\,dy_N  \notag \\
&+\sum_{K=1}^N\frac{\CV_{JK}^{B\CM}(\xi^\prime,\lambda)(\gra+\sur A^2)}{L(A,B)}\int_0^\infty
e^{-Bx_N}\CM(y_N)\widehat{f}_K(\xi^\prime,y_N)\,dy_N  \notag \\
&+\sum_{K=1}^N\frac{\CV_{JK}^{\CM B}(\xi^\prime,\lambda)(\gra+\sur A^2)}{L(A,B)}
\int_0^\infty\CM(x_N)e^{-B y_N}\widehat{f}_K(\xi^\prime,y_N)\,dy_N  \notag \\
&+\sum_{K=1}^N\frac{\CV_{JK}^{\CM\CM}(\xi^\prime,\lambda)(\gra+\sur A^2)}{L(A,B)}
\int_0^\infty\CM(x_N)\CM(y_N)\widehat{f}_K(\xi^\prime,y_N)\,dy_N  \displaybreak[0]\notag \\
\widehat{v}_{j}^{d}(\xi^{\prime},x_{N},\lambda)=
&\frac{i\xi_{j}(c_{g}+c_{\sigma}A^{2})}{(B+A)L(A,B)}\left(-(B-A)e^{-Bx_{N}}+(B^{2}+A^{2})
\CM(x_{N})\right)\widehat{d}(\xi^{\prime}), \displaybreak[0] \notag \\
\widehat{v}_{N}^{d}(\xi^{\prime},x_{N},\lambda)=&
\frac{A(c_{g}+c_{\sigma}A^{2})}{(B+A)L(A,B)}\left((B+A)e^{-Bx_{N}}-(B^{2}+A^{2})
\CM(x_{N})\right)\widehat{d}(\xi^{\prime}), \notag \displaybreak[0] \\
\widehat{\pi}^f(\xi^\prime,x_N,\lambda)=&
\sum_{K=1}^N\frac{\CP_K^{AA}(\xi^\prime,\lambda)(\gra+\sur A^2)}{L(A,B)}
\int_0^\infty e^{-A(x_N+y_N)}\widehat{f}_K(\xi^\prime,y_N)\,dy_N \notag \\
&+\sum_{K=1}^{N}\frac{\CP_{K}^{A\CM}(\xi^\prime,\lambda)(\gra+\sur A^2)}{L(A,B)}
\int_0^\infty e^{-Ax_N}\CM(y_N)\widehat{f}_K(\xi^\prime,y_N)\,dy_N, \notag \displaybreak[0] \\
\widehat{\pi}^{d}(\xi^{\prime},x_{N},\lambda)=&
\frac{(B^{2}+A^{2})(c_{g}+c_{\sigma}A^{2})}{L(A,B)}e^{-Ax_{N}}\widehat{d}(\xi^{\prime}),
\end{align}
where we have set
\begin{align}\label{symbol}
\CV_{jk}^{BB}(\xi^\prime,\lambda)&=-\frac{\xi_j\xi_k(B-A)^2}{(B+A)D(A,B)},
&\CV_{jN}^{BB}(\xi^\prime,\lambda)&=\frac{i\xi_j A(B-A)}{D(A,B)}, \notag \displaybreak[0] \\
\CV_{Nk}^{BB}(\xi^\prime,\lambda)&=-\frac{i\xi_k A(B-A)}{D(A,B)},
&\CV_{NN}^{BB}(\xi^\prime,\lambda)&=-\frac{A^2(B+A)}{D(A,B)}, \notag \displaybreak[0] \\
\CV_{jk}^{B\CM}(\xi^\prime,\lambda)&=\frac{\xi_j\xi_k(B-A)(B^2+A^2)}{(B+A)D(A,B)},
&\CV_{jN}^{B\CM}(\xi^\prime,\lambda)&=-\frac{i\xi_j A(B-A)(B^2+A^2)}{(B+A)D(A,B)}, \notag \displaybreak[0] \\
\CV_{Nk}^{B\CM}(\xi^\prime,\lambda)&=\frac{i\xi_k A(B^2+A^2)}{D(A,B)},
&\CV_{NN}^{B\CM}(\xi^\prime,\lambda)&=\frac{A^2(B^2+A^2)}{D(A,B)}, \notag \displaybreak[0] \\
\CV_{jk}^{\CM B}(\xi^\prime,\lambda)&=\frac{\xi_j\xi_k(B-A)(B^2+A^2)}{(B+A)D(A,B)},
&\CV_{jN}^{\CM B}(\xi^\prime,\lambda)&=-\frac{i\xi_j A(B^2+A^2)}{D(A,B)}, \notag \displaybreak[0] \\
\CV_{Nk}^{\CM B}(\xi^\prime,\lambda)&=\frac{i\xi_k A(B-A)(B^2+A^2)}{(B+A)D(A,B)},
&\CV_{NN}^{\CM B}(\xi^\prime,\lambda)&=\frac{A^2(B^2+A^2)}{D(A,B)}, \notag \displaybreak[0] \\
\CV_{jk}^{\CM\CM}(\xi^\prime,\lambda)&=-\frac{\xi_j\xi_k(B^2+A^2)^2}{(B+A)D(A,B)},
&\CV_{jN}^{\CM\CM}(\xi^\prime,\lambda)&=\frac{i\xi_j A(B^2+A^2)^2}{(B+A)D(A,B)}, \notag \displaybreak[0] \\
\CV_{Nk}^{\CM\CM}(\xi^\prime,\lambda)&=-\frac{i\xi_k A(B^2+A^2)^2}{(B+A)D(A,B)},
&\CV_{NN}^{\CM\CM}(\xi^\prime,\lambda)&=-\frac{A^2(B^2+A^2)^2}{(B+A)D(A,B)}, \notag \displaybreak[0] \\
\CP_k^{AA}(\xi^\prime,\lambda)&=-\frac{i\xi_k (B-A)(B^2+A^2)}{D(A,B)},
&\CP_N^{AA}(\xi^\prime,\lambda)&=-\frac{A(B+A)(B^2+A^2)}{D(A,B)}, \notag \displaybreak[0] \\
\CP_k^{A\CM}(\xi^\prime,\lambda)&=\frac{2i\xi_k AB(B^2+A^2)}{D(A,B)},
&\CP_{N}^{A\CM}(\xi^\prime,\lambda)&=\frac{2A^3(B^2+A^2)}{D(A,B)}.
\end{align}
%
%
%
%
%
%
%
%
%
%
In addition, we see, by inserting \eqref{w:traceA} into $\wh{h}(\xi^\pr,\la)$, that
$\wh{h}(\xi^\pr,\la)=\wh{h}^f(\xi^\pr,\la)+\wh{h}^d(\xi^\pr,\la)$  with
\begin{align}\label{hFsol_eq:2}
\widehat{h}^{f}(\xi^{\prime},\la)=&-\sum_{k=1}^{N-1}\frac{i\xi_{k}(B-A)}{(B+A)L(A,B)}
\int_{0}^{\infty}e^{-Ay_N}\widehat{f}_{k}(\xi^{\prime},y_{N})\,dy_{N} \notag \displaybreak[0] \\
&-\frac{A}{L(A,B)}\int_{0}^{\infty}e^{-Ay_N}\widehat{f}_{N}(\xi^{\prime},y_{N})\,dy_{N} \displaybreak[0] \notag \\
&+\sum_{k=1}^{N-1}\frac{2i\xi_{k}AB}{(B+A)L(A,B)}	\int_{0}^{\infty}\CM(y_{N})
\widehat{f}_{k}(\xi^{\prime},y_{N})\,dy_{N} \displaybreak[0] \notag \\
&+\frac{2A^{3}}{(B+A)L(A,B)}\int_{0}^{\infty}\CM(y_{N})
\widehat{f}_{N}(\xi^{\prime},y_{N})\,dy_{N}, \notag \displaybreak[0] \\
\widehat{h}^{d}(\xi^{\prime},\lambda)=&\frac{D(A,B)}{(B+A)L(A,B)}\wh{d}(\xi^{\prime}).
\end{align}

Next we shall construct cut-off functions.
Let $\varphi\in C_{0}^{\infty}(\mathbf{R}^{N-1})$
be a function such that $0\leq\varphi(\xi^{\prime})\leq 1$,
$\varphi(\xi^{\prime})=1$ for $|\xi^{\prime}|\leq1/3$,
and $\varphi(\xi^{\prime})=0$ for $|\xi^{\prime}|\geq 2/3$.
Let $A_0$  be a number
in $(0, 1)$, which is determined in Section 4 below.
We then define $\varphi_0$ and $\varphi_\infty$ by
\begin{equation}\label{A0}
\varphi_{0}(\xi^\prime)=\varphi(\xi^{\prime}/A_{0}),\quad
\varphi_{\infty}(\xi^{\prime})=1-\varphi(\xi^{\prime}/A_{0}),
\end{equation}
and also set, for $a\in\{0,\infty\}$, $g\in\{f,d\}$, and $F=(f,d)$,
\begin{align}\label{sol:decomp}
S_{a}^{g}(t;A_0)F&=\frac{1}{2\pi i}\int_{\Gamma(\varepsilon)}e^{\lambda t}\mathcal{F}_{\xi^{\prime}}^{-1}[\varphi_{a}(\xi^{\prime})
\widehat{v}^{g}(\xi^{\prime},x_{N},\lambda)](x^{\prime})\,d\lambda, \notag \displaybreak[0] \\
\Pi_{a}^{g}(t;A_0)F&=\frac{1}{2\pi i}\int_{\Gamma(\varepsilon)}e^{\lambda t}\mathcal{F}_{\xi^{\prime}}^{-1}[\varphi_{a}(\xi^{\prime})
\widehat{\pi}^{g}(\xi^{\prime},x_{N},\lambda)](x^{\prime})\,d\lambda, \notag \displaybreak[0] \\
T_{a}^{g}(t;A_0)F&=\frac{1}{2\pi i}\int_{\Gamma(\varepsilon)}e^{\lambda t}\mathcal{F}_{\xi^{\prime}}^{-1}[\varphi_{a}(\xi^{\prime})
\widehat{h}^{g}(\xi^{\prime},\lambda)](x^{\prime})\,d\lambda, \notag \displaybreak[0] \\
R(t)f&=\frac{1}{2\pi i}\int_{\Gamma(\varepsilon)}e^{\lambda t}
\mathcal{F}_{\xi^{\prime}}^{-1}[\widehat{w}(\xi^{\prime},x_{N},\lambda)](x^{\prime})\,d\lambda, \notag \displaybreak[0] \\
P(t)f&=\frac{1}{2\pi i}\int_{\Gamma(\varepsilon)}e^{\lambda t}
\mathcal{F}_{\xi^{\prime}}^{-1}[\widehat{p}(\xi^{\prime},x_{N},\lambda)](x^{\prime})\,d\lambda\quad(t>0)
\end{align}
with $\wh{v}^g(\xi',x_N,\la)=(\wh{v}_1^g(\xi',x_N,\la),\dots,\wh{v}_N^g(\xi',x_N,\la))^T$.
Here we have taken the integral path $\Gamma(\varepsilon)$ as follows:
\begin{equation}\label{Gamma}
	\Gamma(\varepsilon)=\Gamma^{+}(\varepsilon)\cup\Gamma^{-}(\varepsilon),
		\quad\Gamma^{\pm}(\varepsilon)=\{\lambda\in\BC\ |\ \lambda=\widetilde{\lambda}_0(\varepsilon)+se^{\pm i(\pi-\varepsilon)},\ s\in(0,\infty)\}
\end{equation}
for $\widetilde{\lambda}_0(\varepsilon)=2\lambda_0(\varepsilon)/\sin{\ep}$ with $\ep\in(0,\pi/2)$,
where $\lambda_0(\varepsilon)$ is the same number as in  Lemma \ref{lem:symbol} (3).
%
%
%
%
\begin{figure}[h]
	\begin{center}
		\includegraphics[width=8cm,clip]{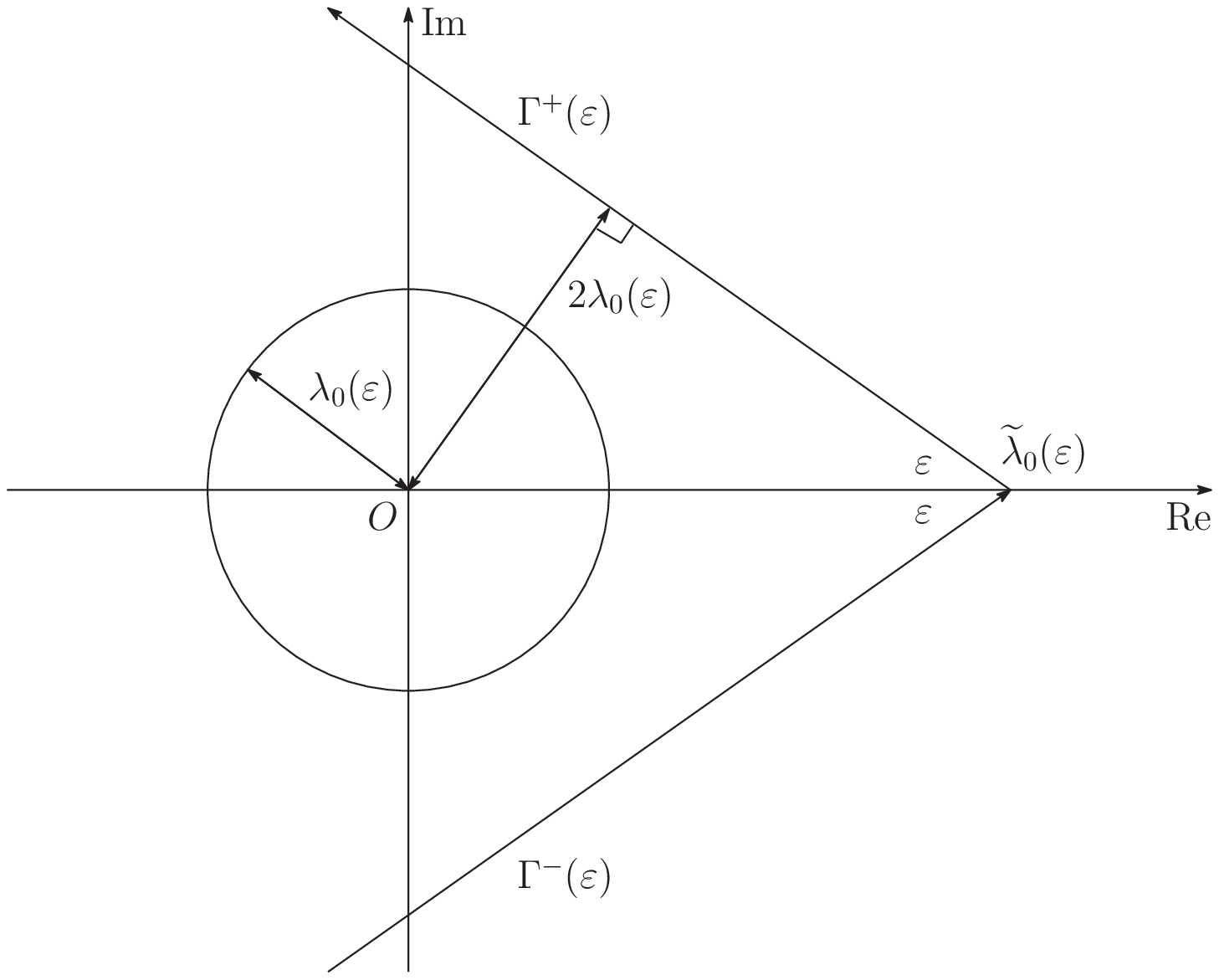}
	\end{center}
	\caption{
 			$\Gamma(\varepsilon)=\Gamma^{+}(\varepsilon)\cup\Gamma^-(\varepsilon)$
		}
\end{figure}
\footnote[0]
{Figure 1 is reprinted from Ph.D. thesis of the first author.}
%
%
%
%
%
%
%
%
\begin{rema}
\begin{enumerate}[(1)]
\item
If we set
\begin{align*}
S(t)F&=\sum_{a\in\{0,\infty\}}\sum_{g\in\{f,d\}}S_{a}^g(t;A_0)F+R(t)f, \notag \displaybreak[0] \\
\Pi(t)F&=\sum_{a\in\{0,\infty\}}\sum_{g\in\{f,d\}}\Pi_{a}^g(t;A_0)F+P(t)f,\notag \displaybreak[0] \\
T(t)F&=\sum_{a\in\{0,\infty\}}\sum_{g\in\{f,d\}}T_{a}^g(t;A_0)F,
\end{align*}
then $S(t)F$, $\Pi(t)F$, and $T(t)F$ are the requirements in Theorem \ref{theo:main} (1).
Especially, let $\CS(t):F\mapsto(S(t)F,T(t)F)$ and $1<p<\infty$,
and then $\{\CS(t)\}_{t\geq 0}$ is an analytic semi-group on $X_p$, defined in \eqref{X_q},
as was seen in \cite{Shibata1}.
On the other hand, by Lemma \ref{lem:1}, $\{R(t)\}_{t\geq 0}$ is an analytic semi-group on $J_p(\uhs)$,
and also $R(t)$ and $P(t)$ satisfy
\begin{align*}
&\|\nabla^\ell R(t)f\|_{L_p(\uhs)}\leq Ct^{-\ell/2}\|f\|_{L_p(\uhs)}, \\
&\|(\pa_t R(t)f,\nabla P(t)f)\|_{L_p(\uhs)}\leq Ct^{-1}\|f\|_{L_p(\uhs)}
\end{align*}
for $f\in L_p(\uhs)^N$, $\ell=0,1,2$, and $t>0$. These estimates imply that \eqref{140917_3} holds.
\item
For $a\in\{0,\infty\}$ and $g\in\{f,d\}$, the extension $\CE(T_a^g(t;A_0)F)$ defined as \eqref{ext:H} is decomposed into
\begin{equation}\label{140917_4}
\CE(T_a^g(t;A_0)F)
=\frac{1}{2\pi i}\int_{\Ga(\ep)}e^{\la t}\CF_{\xi^\pr}^{-1}
[\ph_a(\xi^\pr)e^{-Ax_N}\wh{h}^g(\xi^\pr,\la)](x^\pr)\,d\la.
\end{equation}
\item
In the following sections, we show, for $a\in\{0,\infty\}$, that
\begin{align*}
&S_a(t)F=\sum_{g\in\{f,d\}}S_a^g(t;A_0)F,\quad
\Pi_a(t)F=\sum_{g\in\{f,d\}}\Pi_a^g(t;A_0)F, \\
&T_a(t)F=\sum_{g\in\{f,d\}}T_a^g(t;A_0)F
\end{align*}
satisfy the estimates \eqref{140917_1} and \eqref{140917_2}, respectively. 
%
%
%
%
%
%
\end{enumerate}
\end{rema}

We devote the last part of this section to the proof of the following lemma.
\begin{lemm}\label{lem:anal}
Let $\xi^{\prime}\in\mathbf{R}^{N-1}\setminus\{0\}$ and
$\lambda\in\{z\in\mathbf{C}\mid{\rm Re}z\geq0\}$. Then $L(A,B)\neq 0$.
\end{lemm}
\begin{proof}
Applying the partial Fourier transform with respect to tangential variable $x^\prime$ to the equations (\ref{RP}) with $f=0$ and $d=0$ yields that
\begin{align}\label{eq:F}
&\lambda\widehat{u}_{j}(x_{N})-\sum_{k=1}^{N-1}i\xi_{k}(i\xi_{j}\widehat{u}_{k}(x_{N})+i\xi_{k}\widehat{u}_{j}(x_{N})) \notag \\
&\qquad\qquad\qquad\qquad\qquad\qquad\,\,\enskip
-D_{N}(D_{N}\widehat{u}_{j}(x_{N})+i\xi_{j}\widehat{u}_{N}(x_{N}))+i\xi_{j}\widehat{\theta}(x_{N})=0, \displaybreak[0] \notag \\
&\lambda\widehat{u}_{N}(x_{N})-\sum_{k=1}^{N-1}i\xi_{k}(D_{N}\widehat{u}_{k}(x_{N})+i\xi_{k}\widehat{u}_{N}(x_{N}))
-2D_{N}^{2}\widehat{u}_{N}(x_{N})+D_{N}\widehat{\theta}(x_{N})=0, \displaybreak[0] \notag \\
&\sum_{k=1}^{N-1}i\xi_{k}\widehat{u}_{k}(x_{N})+D_{N}\widehat{u}_{N}(x_{N})=0,\quad 
\lambda\widehat{h}+\widehat{u}_{N}(0)=0, \displaybreak[0] \notag \\
&D_{N}\widehat{u}_{j}(0)+i\xi_{j}\widehat{u}_{N}(0)=0,\quad-\widehat{\theta}(0)+2D_{N}\widehat{u}_{N}(0)+(c_{g}+c_{\sigma}A^{2})\widehat{h}=0
\end{align}
for $x_{N}>0$, where we have used the symbols:
\begin{equation*}
\widehat{u}_{J}(x_{N})=\widehat{u}_{J}(\xi^{\prime},x_{N}),\quad
\widehat{\theta}(x_{N})=\widehat{\theta}(\xi^{\prime},x_{N}),			
\quad\widehat{h}=\widehat{h}(\xi^{\prime}).
\end{equation*}
We here set
\begin{equation*}
\widehat{u}(x_{N})=(\widehat{u}_{1}(x_{N}),\dots,\widehat{u}_{N}(x_{N}))^{T},
\quad\|f\|^2=\int_0^\infty f(x_{N})\overline{f(x_{N})}\,dx_{N},
\end{equation*}
and show that $L(A,B)\neq 0$ by contradiction.
Suppose that $L(A,B)=0$. We know that
(\ref{eq:F}) admits a solution
$(\widehat{u}(x_N),\widehat{\theta}(x_N),\widehat{h})\neq 0$ that decays exponentially
when $x_N\to\infty$ (see e.g. \cite[Section 4]{S-S:model}). 
To obtain
\begin{multline}\label{0920_1}
0=\lambda\|\widehat{u}\|^{2}+2\|D_{N}\widehat{u}_{N}\|^{2}
+\sum_{j,k=1}^{N-1}\|i\xi_{k}\widehat{u}_{j}\|^{2} \\
+\|\sum_{j=1}^{N-1}i\xi_j\widehat{u}_j\|^2
+\sum_{j=1}^{N-1}\|D_{N}\widehat{u}_{j}+i\xi_{j}\widehat{u}_{N}\|^{2}
+\overline{\lambda}(c_{g}+c_{\sigma}A^{2})|\widehat{h}|^2,
\end{multline}
we multiply the first equation of \eqref{eq:F}
by $\overline{\widehat{u}_j(x_N)}$
and the second equation 
by $\overline{\widehat{u}_N(x_N)}$,
and integrate the resultant formulas with respect to $x_N\in(0,\infty)$,
and furthermore, after integration  by parts, we use
the third to sixth equations of (\ref{eq:F}).
Taking the real part of \eqref{0920_1}, we have
\begin{equation*}
D_N\widehat{u}_N(x_N)=0,\quad
D_{N}\widehat{u}_j(x_N)+i\xi_j\widehat{u}_N=0\quad
\text{for ${\rm Re}\,\lambda\geq0$}.
\end{equation*}
In particular, $\widehat{u}_N$ is a constant,
but $\widehat{u}_N=0$ since $\lim_{x_N\to\infty}\widehat{u}_N=0$.
We thus have $D_N\widehat{u}_j=0$, which implies that $\widehat{u}_j=0$ since $\lim_{x_N\to\infty}\widehat{u}_j=0$.
Combining $\widehat{u}_j=0$ and the first equation of (\ref{eq:F}) yields that
$i\xi_j\widehat{\theta}=0$.
This implies that $\widehat{\theta}=0$ because $\xi^\pr\neq0$.
In addition, by the sixth equation of (\ref{eq:F}),
we have $(\gra+\sur A^2)\wh{h}=0$.
Since $\gra+\sur A^2\neq0$, we see that $\wh{h}=0$.
We thus have $\wh{u}=0$, $\wh{\te}=0$, and $\wh{h}=0$,
which leads to a contradiciton. 
This completes the proof of Lemma \ref{lem:anal}.
\end{proof}

\section{Analysis of low frequency parts}\label{sec:low}
In this section, we show the estimates \eqref{140917_1} in Theorem \ref{theo:main} (2).
If we consider the Lopatinskii determinant $L(A,B)$ defined in (\ref{0921_1}) as a polynomial
with respect to $B$, then it has four roots $B_j^\pm$ ($j=1,2$), which have the following asymptotics:
\begin{equation}\label{expan:low}
B_j^\pm=e^{\pm i(2j-1)(\pi/4)}\gra^{1/4}A^{1/4}-\frac{A^{7/4}}{2e^{\pm i(2j-1)(\pi/4)}\gra^{1/4}}
-\frac{\sur A^{9/4}}{e^{\pm i(2j-1)(3\pi/4)}\gra^{3/4}}+O(A^{10/4})
\end{equation}
as $A\to0$. Set $\la_\pm=(B_1^\pm)^2-A^2$, and then
\begin{equation}\label{expan:low2}
\la_\pm=\pm i\gra^{1/2}A^{1/2}-2A^2\mp\frac{2\sur}{i\gra^{1/2}}A^{10/4}+O(A^{11/4})
\quad\text{as $A\to 0$.}
\end{equation}
\begin{rema}For $\lambda \in \Sigma_\epsilon$,
we choose a brunch such that
${\rm Re}B={\rm Re}\sqrt{\lambda+A^{2}}>0$. 
Note that $\lambda_\pm \in \Sigma_\epsilon$ and ${\rm Re}\, (\lambda_\pm+ A^2) < 0$.
\end{rema}
We define a positive number $\ep_0$ by $\ep_0=\tan^{-1}\{(A^2/8)/A^2\}=\tan^{-1}(1/8)$, and furthermore,
we set
\begin{align*}
\Ga_0^\pm&=\{\la\in\BC\mid\la=\la_\pm+(\gra^{1/2}/4)A^{1/2}e^{\pm iu},\ u:0\to2\pi\},\\
\Ga_1^\pm&=\{\la\in\BC\mid\la=-A^2+(A^2/4)e^{\pm iu},\ u:0\to\pi/2\},\\
\Ga_2^\pm&=\{\la\in\BC\mid\la=-(A^2(1-u)+\ga_0u)\pm i((A^2/4)(1-u)+\wit{\ga}_0 u),\ u:0\to1\},\\
\Ga_3^\pm&=\{\la\in\BC\mid\la=-(\ga_0\pm i\wit{\ga}_0)+ue^{\pm i(\pi-\ep_0)},\ u:0\to\infty\}
\end{align*}
with $\ga_0=\la_0(\ep_0)$ given by Lemma \ref{lem:symbol} (3) and
\begin{equation}\label{gam0:tilde}
\wit{\ga}_0=\frac{1}{8}\left(\la_0(\ep_0)+\wit{\la}_0(\ep_0)\right)
=\frac{1}{8}\left(1+\frac{2}{\sin{\ep_0}}\right)\la_0(\ep_0)=\frac{(1+2\sqrt{65})\ga_0}{8},
\end{equation}
where $\wit{\la}_0(\ep_0)$ is the same constant as in (\ref{Gamma}) with $\ep=\ep_0$.
%
%
%
%
\begin{figure}[h]
\begin{center}
\includegraphics[width=10cm,clip]{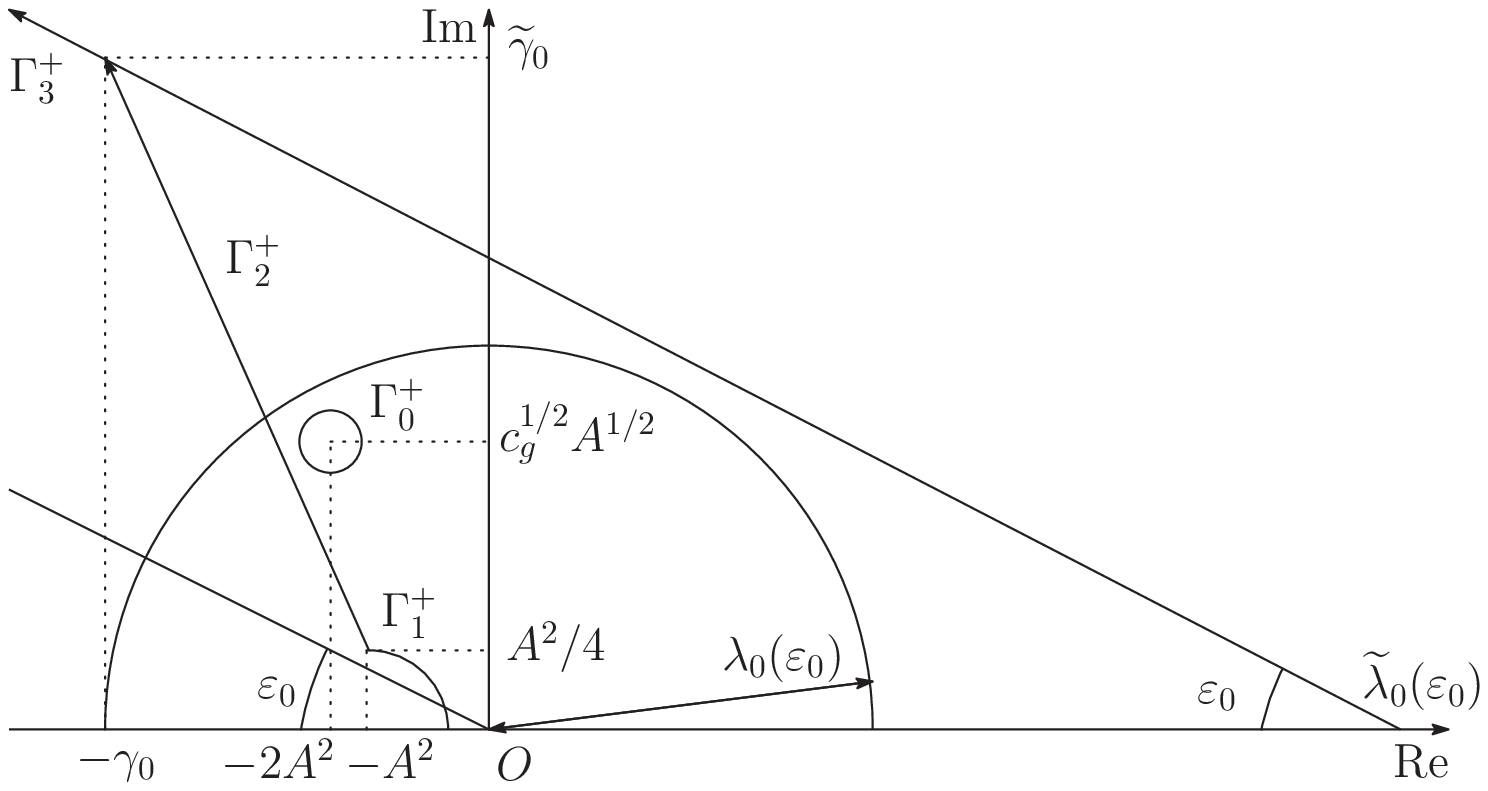}
\end{center}\caption{$\Gamma_\sigma^+$ $(\sigma=0,1,2,3)$}
\end{figure}
\footnote[0]
{Figure 2 is reprinted from Ph.D. thesis of the first author.}

Then, by Cauchy's integral theorem, we decompose $S_0^{g}(t;A_0)F$, $\Pi_0^g(t;A_0)F$, and $\CE({T}_0^g(t;A_0)F)$
given by (\ref{sol:decomp}) and \eqref{140917_4} as follows: For $g\in\{f,d\}$
\begin{align}\label{140923_10}
&S_0^g(t;A_0)F=\sum_{\si=0}^{3}S_0^{g,\si}(t;A_0)F, \quad
\Pi_0^g(t;A_0)F=\sum_{\si=0}^3\Pi_0^{g,\si}(t;A_0)F, \notag \\
&\CE(T_0^g(t;A_0)F)=\sum_{\si=0}^3\CE(T_0^{g,\si}(t;A_0)F)
\end{align}
with
\begin{align}\label{low:decomp}
S_0^{g,\si}(t;A_0)F&=\CF_{\xi^\pr}^{-1}\left[\frac{1}{2\pi i}\int_{\Ga_\si^+\cup\Ga_\si^-}
e^{\la t}\ph_0(\xi^\pr)\wh{v}^g(\xi^\pr,x_N,\la)\,d\la\right](x^\pr), \notag \\
\Pi_0^{g,\si}(t;A_0)F&=\CF_{\xi^\pr}^{-1}\left[\frac{1}{2\pi i}\int_{\Ga_\si^+\cup\Ga_\si^-}
e^{\la t}\ph_0(\xi^\pr)\wh{\pi}^g(\xi^\pr,x_N,\la)\,d\la\right](x^\pr), \notag \\
\CE(T_0^{g,\si}(t;A_0)F)&=\CF_{\xi^\pr}^{-1}\left[\frac{1}{2\pi i}\int_{\Ga_\si^+\cup\Ga_\si^-}
e^{\la t}\ph_{0}(\xi^\pr)e^{-Ax_N}\wh{h}^g(\xi^\pr,\la)\,d\la\right](x^\pr),
\end{align}
where $\varphi_0(\xi^\prime)$ is the cut-off function given in (\ref{A0}).
In order to estimate each term in (\ref{low:decomp}),
we here introduce operators $K_{n}^{\pm,\si}(t;A_0)$ and $L_{n}^{\pm,\si}(t;A_0)$ defined by
\begin{align}\label{op:low}
[K_{n}^{\pm,\si}(t;A_0)f](x)&=\int_0^{\infty}\CF_{\xi^\pr}^{-1}\left[\int_{\Ga_\si^\pm}
e^{\la t}\ph_0(\xi^\pr)k_n(\xi^\pr,\la)\CX_n(x_N,y_N)\,d\la\,\wh{f}(\xi^\pr,y_N)\right](x^\pr)\,dy_N, \notag \\
[L_{n}^{\pm,\si}(t;A_0)d](x)&=\CF_{\xi^\pr}^{-1}\left[\int_{\Ga_\si^\pm}e^{\la t}\ph_{0}(\xi^\pr)
\ell_n(\xi^\pr,\la)\CY_n(x_N)\,d\la\,\wh{d}(\xi^\pr)\right](x^\pr)\quad(\si=0,1,2,3)
\end{align}
with some multipliers $k_n(\xi^\pr,\la)$ and $\ell_n(\xi^\pr,\la)$, 
where $\CX_n(x_N,y_N)$ and $\CY_n(x_N)$ are given by
\begin{equation*}
\CX_n(x_N,y_N)=
\left\{\begin{aligned}
&e^{-A(x_N+y_N)}&(n=1),\\
&e^{-Ax_N}\CM(y_N)&(n=2), \\
&e^{-B(x_N+y_N)}&(n=3), \\
&e^{-Bx_N}\CM(y_N)&(n=4), \\
&\CM(x_N)e^{-By_N}&(n=5), \\
&\CM(x_N)\CM(y_N)&(n=6), \\
\end{aligned}\right.
\quad\CY_n(x_N)=
\left\{\begin{aligned}
&e^{-Ax_N}&(n=1), \\
&e^{-Bx_N}&(n=2), \\
&\CM(x_N)&(n=3).
\end{aligned}\right.
\end{equation*}

\subsection{Analysis on $\Ga_0^\pm$}
Our aim here is to show the following theorem for the operators
given in (\ref{low:decomp}) with $\si=0$.
\begin{theo}\label{thm:Gam0}
Let $1\leq r\leq  2\leq q\leq\infty$ and $F=(f,d)\in L_r(\uhs)^N\times L_r(\tws)$. 
Then there exists an $A_0 \in (0, 1)$ such that the following assertions hold:
\begin{enumerate}[$(1)$]
\item
Let $k=0,1$, $\ell=0,1,2$, and $\al^\pr\in\BN_0^{N-1}$.
Then there exist a positive constant $C=C(\al^\pr)$ such that for any $t>0$
\begin{align*}
&\|\pa_t^k D_{x^\pr}^{\al^\pr}D_N^{\ell}S_0^{f,0}(t;A_0)F\|_{L_{q}(\uhs)}
\leq C(t+1)^{-\frac{N}{2}\left(\frac{1}{r}-\frac{1}{q}\right)
-\frac{k}{4}-\frac{|\al^\pr|}{2}-\frac{\ell}{8}}\|f\|_{L_{r}(\uhs)}, \\
&\|\pa_t^k D_{x^\pr}^{\al^\pr}D_N^{\ell}S_0^{d,0}(t;A_0)F\|_{L_{q}(\uhs)} \\
&\quad\leq C
\left\{\begin{aligned}
&(t+1)^{-\frac{N-1}{2}\left(\frac{1}{r}-\frac{1}{q}\right)-\frac{1}{2}
\left(\frac{1}{2}-\frac{1}{q}\right)-\frac{k}{4}-\frac{|\al^\pr|}{2}}\|d\|_{L_{r}(\tws)}&&(\ell=0), \\
&(t+1)^{-\frac{N-1}{2}\left(\frac{1}{r}-\frac{1}{q}\right)-\frac{1}{8}\left(2-\frac{1}{q}\right)
-\frac{k}{4}-\frac{|\al^\pr|}{2}-\frac{\ell}{8}}\|d\|_{L_{r}(\tws)}&&(\ell=1,2).
\end{aligned}\right.
\end{align*}
\item
There exists a positive constant $C$ such that for any $t>0$
\begin{align*}
\|\nabla\Pi_0^{f,0}(t;A_0)F\|_{L_{q}(\uhs)}
&\leq C(t+1)^{-\frac{N}{2}\left(\frac{1}{r}-\frac{1}{q}\right)-\frac{1}{4}}\|f\|_{L_{r}(\uhs)}, \\
\|\nabla\Pi_{0}^{d,0}(t;A_0)F\|_{L_{q}(\uhs)} &\leq C(t+1)^{-\frac{N-1}{2}\left(\frac{1}{r}-\frac{1}{q}\right)
-\frac{1}{2}\left(\frac{1}{2}-\frac{1}{q}\right)-\frac{1}{4}}\|d\|_{L_{r}(\tws)}.
\end{align*}
\item
Let $\al\in\BN_0^N$. Then there exists a positive constant $C=C(\alpha)$ such that for any $t>0$
\begin{align*}
\|D_x^\al \nabla\CE({T}_0^{f,0}(t;A_0)F)\|_{L_{q}(\uhs)}&\leq C(t+1)^{-\frac{N}{2}\left(\frac{1}{r}-\frac{1}{q}\right)
-\frac{1}{4}-\frac{|\al|}{2}}\|f\|_{L_{r}(\uhs)}, \\
\|D_x^{\al}\pa_t\CE(T_0^{f,0}(t;A_0)F)\|_{L_{q}(\uhs)}&\leq 
C(t+1)^{-\frac{N}{2}\left(\frac{1}{r}-\frac{1}{q}\right)-\frac{|\al|}{2}}\|f\|_{L_{r}(\uhs)}, \\
\|D_x^{\al}\nabla\CE(T_{0}^{d,0}(t;A_0)F)\|_{L_{q}(\uhs)}&\leq C(t+1)^{-\frac{N-1}{2}\left(\frac{1}{r}-\frac{1}{q}\right)
-\frac{1}{2}\left(\frac{1}{2}-\frac{1}{q}\right)-\frac{1}{4}-\frac{|\al|}{2}}\|d\|_{L_{r}(\tws)}.
\end{align*}
\end{enumerate}
\end{theo}
We here introduce some fundamental lemmas to show Theorem \ref{thm:Gam0}.
\begin{lemm}\label{lemm:fund1}
Let $s_i\geq 0$ $(i=0,1,2,3)$.
Then there exists a positive constant $C=C(s_0,s_1,s_2,s_3)$ such that for any $\tau>0$, $a\geq 0$, and $Z\geq0$
\begin{equation*}
e^{-s_0 (Z^2)\tau}Z^{s_1}e^{-s_2 (Z^{s_3})a}\leq C(\tau^{s_1/2}+a^{s_1/s_3})^{-1}.
\end{equation*}
\end{lemm}
\begin{lemm}\label{lemm:fund2}
Let $1\leq q,r\leq\infty$, $a>0$, $b_1>0$, and $b_2>0$.
\begin{enumerate}[$(1)$]
\item
Set $g(x_N,\tau)=(\tau^a+(x_N)^{b_1})^{-1}$for $x_N>0$ and $\tau>0$.
Then there exists a positive constant $C$ such that for any $\tau>0$
\begin{equation*}
\|g(\tau)\|_{L_q((0,\infty))}\leq C\tau^{-a\left(1-\frac{1}{b_1 q}\right)},
\end{equation*}
provided that $b_1 q>1$.
\item
Let $f\in L_r((0,\infty))$, and set, for $x_N>0$ and $\tau>0$,
\begin{equation*}
g(x_N,\tau)=\int_0^\infty\frac{f(y_N)}{\tau^a+(x_N)^{b_1}+(y_N)^{b_2}}\,dy_N.
\end{equation*}
Then there exists a positive constant $C$ such that for any $\tau>0$
\begin{equation*}
\|g(\tau)\|_{L_q((0,\infty))}
\leq C\tau^{-a\left(1-\frac{1}{b_1q}-\frac{1}{b_2}+\frac{1}{b_2 r}\right)}\|f\|_{L_r((0,\infty))},
\end{equation*}
provided that for $r^\pr =r/(r-1)$
\begin{equation*}
b_1 q>1,\quad b_2\left(1-\frac{1}{b_1 q}\right)r^{\pr}>1.
\end{equation*}
\end{enumerate}
\end{lemm}
%
%
%
%
%
%
%
%
%
%
%
%
%
%
By using Lemma \ref{lemm:fund1} and Lemma \ref{lemm:fund2}, we obtain the following lemma.
\begin{lemm}\label{lem:Gam0}
Let $1\leq r\leq 2 \leq q\leq\infty$, and let $f\in L_{r}(\uhs)^{N}$ and $d\in L_{r}(\tws)$.
For multipliers $\kappa_n(\xi', \lambda)$ and $m_n(\xi', \lambda)$ given below, we set, in $(\ref{op:low})$,
\begin{equation*}
k_n(\xi^\pr,\la)=\frac{\ka_n(\xi^\pr,\la)}{L(A,B)},\quad
\ell_n(\xi^\pr,\la)=\frac{m_n(\xi^\pr,\la)}{L(A,B)}.
\end{equation*}
\begin{enumerate}[$(1)$]
\item
Let $s\geq0$ and suppose that there exist constants $A_1\in(0,1)$ and  $C=C(s)>0$
such that for any $A\in(0,A_1)$
\begin{equation*}
\begin{aligned}
|\ka_1(\xi^\pr,\la_\pm)|&\leq CA^{\frac{6}{4}+s},
&|\ka_2(\xi^\pr,\la_\pm)|&\leq CA^{\frac{7}{4}+s},
&|\ka_3(\xi^\pr,\la_\pm)|&\leq CA^{\frac{6}{4}+s}, \\
|\ka_4(\xi^\pr,\la_\pm)|&\leq CA^{\frac{7}{4}+s},
&|\ka_5(\xi^\pr,\la_\pm)|&\leq CA^{\frac{7}{4}+s},
&|\ka_6(\xi^\pr,\la_\pm)|&\leq CA^{\frac{8}{4}+s}.
\end{aligned}
\end{equation*}
Then there exist constants $A_0 \in (0, A_1)$ and  $C=C(s)>0$ such that for any $t>0$
\begin{align*}
\|K_{n}^{\pm,0}(t;A_0)f\|_{L_{q}(\uhs)}
&\leq C(t+1)^{-\frac{N}{2}\left(\frac{1}{r}-\frac{1}{q}\right)-\frac{s}{2}}
\|f\|_{L_{r}(\uhs)}\quad(n=1,2,6), \\
\|K_{3}^{\pm,0}(t;A_0)f\|_{L_{q}(\uhs)}
&\leq C(t+1)^{-\left(\frac{N-1}{2}+\frac{1}{8}\right)\left(\frac{1}{r}-\frac{1}{q}\right)-\frac{3}{8}-\frac{s}{2}}
\|f\|_{L_{r}(\uhs)}, \\
\|K_{4}^{\pm,0}(t;A_0)f\|_{L_{q}(\uhs)}
&\leq C(t+1)^{-\left(\frac{N-1}{2}+\frac{1}{8}\right)\left(\frac{1}{r}-\frac{1}{q}\right)
-\frac{3}{8r}-\frac{s}{2}}\|f\|_{L_{r}(\uhs)}, \\
\|K_{5}^{\pm,0}(t;A_0)f\|_{L_{q}(\uhs)}
&\leq C(t+1)^{-\left(\frac{N-1}{2}+\frac{1}{8}\right)
\left(\frac{1}{r}-\frac{1}{q}\right)-\frac{3}{8}\left(1-\frac{3}{q}\right)-\frac{s}{2}}
\|f\|_{L_{r}(\uhs)}.
\end{align*}
\item
Let $s\geq0$ and suppose that there exist constants $A_1\in(0,1)$
and $C=C(s)>0$ such that for any $A\in(0,A_1)$
\begin{equation*}
|m_1(\xi^\pr,\la_\pm)|\leq CA^{1+s},\enskip
|m_2(\xi^\pr,\la_\pm)|\leq CA^{1+s},\enskip
|m_3(\xi^\pr,\la_\pm)|\leq CA^{\frac{5}{4}+s}.
\end{equation*}
Then there exist constants $A_0 \in (0, A_1)$ and
$C=C(s)>0$ such that for any $t>0$
\begin{align*}
\|L_{n}^{\pm,0}(t;A_0)d\|_{L_{q}(\uhs)}
&\leq C(t+1)^{-\frac{N-1}{2}\left(\frac{1}{r}-\frac{1}{q}\right)
-\frac{1}{2}\left(\frac{1}{2}-\frac{1}{q}\right)-\frac{s}{2}}\|d\|_{L_{r}(\tws)}\quad(n=1,3), \\
\|L_{2}^{\pm,0}(t;A_0)d\|_{L_{q}(\uhs)}&\leq C(t+1)^{-\frac{N-1}{2}\left(\frac{1}{r}-\frac{1}{q}\right)
-\frac{1}{8}\left(2-\frac{1}{q}\right)-\frac{s}{2}}\|d\|_{L_{r}(\tws)}.
\end{align*}
\end{enumerate}
\end{lemm}
\begin{proof}
We use the abbreviations:
$\|\cdot\|_{2}=\|\cdot\|_{L_{2}(\tws)}$,
$\wh{f}(y_N)=\wh{f}(\xi^\pr,y_{N})$, and $\wt=t+1$ for $t>0$ in this proof,
and consider only the estimates on $\Ga_0^+$ since the estimates on $\Ga_0^-$ can be shown similarly. \\
%
%
%
%
(1) We first show the inequality for $K_{1}^{+,0}(t;A_0)$.
Noting that $B^2-(B_1^+)^2=\la-\la_+$,
by the residue theorem, we have
\begin{align}\label{0810_1}
&[K_{1}^{+,0}(t;A_0)f](x)= \notag \\
&\int_0^\infty\CF_{\xi^\pr}^{-1}\left[\int_{\Ga_0^+}e^{\la t}
\frac{\ph_0(\xi^\pr)\ka_1(\xi^\pr,\la)(B+B_1^+)}
{(\la-\la_+)(B-B_1^-)(B-B_2^+)(B-B_2^-)}
e^{-A(x_{N}+y_{N})}\,d\la \wh{f}(y_{N})\right](x^\pr)\,dy_{N} \notag \\
&=4\pi i\int_{0}^{\infty}\CF_{\xi^\pr}^{-1}\left[\frac{e^{\la_+ t}\ph_0(\xi^\pr)\ka_1(\xi^\pr,\la_+)B_{1}^{+}}
{(B_{1}^{+}-B_{1}^{-})(B_{1}^{+}-B_{2}^{+})(B_{1}^{+}-B_{2}^{-})}e^{-A(x_{N}+y_{N})}
\wh{f}(y_N)\right](x^\pr)\,dy_{N}.
\end{align}
In view of (\ref{expan:low}) and (\ref{expan:low2}), we can choose $A_0 \in (0, A_1)$ in such a way that 
\begin{equation}\label{0810_6}
|e^{\la_+ t}|\leq C e^{-A^2\wt},\enskip|B_1^+-B_1^-|\geq CA^{\frac{1}{4}},
\enskip|B_1^+-B_2^+|\geq CA^{\frac{1}{4}},\quad|B_1^+-B_2^-|\geq CA^{\frac{1}{4}}
\end{equation}
for any $A\in (0,A_0)$ and $t>0$ with some constant $C$. 
Thus, by $L_q\text{-}L_r$ estimates of 
the $(N-1)$-dimensional heat kernel and Parseval's theorem, we have
\begin{align}\label{0810_2}
&\|[K_{1}^{+,0}(t;A_0)f](\cdot\,,x_N)\|_{L_{q}(\tws)} \notag \displaybreak[0] \\
&\quad\leq C\wt^{-\frac{N-1}{2}\left(\frac{1}{2}-\frac{1}{q}\right)}
\int_{0}^{\infty}\left\|\frac{e^{-(A^2/2)\wt}\ph_0(\xi^\pr)A^{\frac{6}{4}+s}A^{\frac{1}{4}}}{A^{\frac{3}{4}}}
e^{-A(x_N+y_N)}\wh{f}(y_N)\right\|_{2}dy_{N} \notag \displaybreak[0] \\
&\quad\leq C\wt^{-\frac{N-1}{2}\left(\frac{1}{2}-\frac{1}{q}\right)-\frac{s}{2}}\int_{0}^{\infty}
\left\|e^{-(A^2/4)\wt}Ae^{-A(x_N+y_N)}\wh{f}(y_N)\right\|_{2}dy_{N} \notag \displaybreak[0] \\
&\quad\leq C\wt^{-\frac{N-1}{2}\left(\frac{1}{2}-\frac{1}{q}\right)-\frac{s}{2}}
\int_{0}^{\infty}\frac{\|e^{-(A^2/8)\wt}
\,\wh{f}(y_N)\|_{L_{2}(\tws)}}{\wt^{\,1/2}+x_{N}+y_{N}}\,dy_{N} \notag \displaybreak[0] \\
&\quad\leq C\wt^{-\frac{N-1}{2}\left(\frac{1}{r}-\frac{1}{q}\right)-\frac{s}{2}}\int_{0}^{\infty}
\frac{\|f(\cdot\,,y_{N})\|_{L_{r}(\tws)}}{\wt^{\,1/2}+x_{N}+y_{N}}\,dy_{N},
\end{align}
where we have used Lemma \ref{lemm:fund1} with $s_0=1/8$, $s_i=1$ $(i=1,2,3)$, $a=x_N+y_N$, and $Z=A$.
%
%
%
%
%
%
%
%
%
%
%
%
%
%
%
%
%
%
%
%
%
%
If $q>2$, then applying Lemma \ref{lemm:fund2} (2) with $a=1/2$ and $b_1=b_2=1$ to (\ref{0810_2}) furnishes that
\begin{equation*}
\|K_{1}^{+,0}(t;A_0)f\|_{L_q(\uhs)}\leq C
\wt^{-\frac{N}{2}\left(\frac{1}{r}-\frac{1}{q}\right)-\frac{s}{2}}
\|f\|_{L_r(\uhs)}.
\end{equation*}
In the case of $(q,r)=(2,2)$, by \eqref{0810_2}
\begin{equation*}
\|[K_1^{+,0}(t;A_0)f](\cdot,x_N)\|_2
\leq C \wt^{-\frac{s}{2}}\int_0^\infty\left\|\CF_{\xi^\pr}^{-1}\left[Ae^{-A(x_N+y_N)}\wh{f}(y_N)\right]\right\|_2dy_N,
\end{equation*}
and then it follows from \cite[Lemma 5.4]{S-S:model} that
\begin{equation*}
\|K_{1}^{+,0}(t;A_0)f\|_{L_2(\uhs)}\leq C
\wt^{-\frac{s}{2}}\|f\|_{L_2(\uhs)}.
\end{equation*}
On the other hand, in the case of $1\leq r<2$ and $q=2$, by the second inequality of \eqref{0810_2}, Lemma \ref{lemm:fund1}, and H\"older's inequality
\begin{align*}
\|K_1^{+,0}(t;A_0)f\|_{L_2(\uhs)}
&\leq C \wt^{-\frac{s}{2}}\int_0^\infty\|e^{-(A^2/2)\wt}A^{1/2}e^{-Ay_N}\wh{f}(y_N)\|_2\,dy_N \displaybreak[0] \notag \\
&\leq C\wt^{-\frac{N-1}{2}\left(\frac{1}{r}-\frac{1}{2}\right)-\frac{s}{2}}
\int_0^\infty\frac{\|f(\cdot,y_N)\|_{L_r(\tws)}}{\wt^{1/4}+(y_N)^{1/2}}\,dy_N \displaybreak[0] \notag \\
&\leq C\wt^{-\frac{N}{2}\left(\frac{1}{r}-\frac{1}{2}\right)-\frac{s}{2}}\|f\|_{L_r(\uhs)},
\end{align*}
which implies that the required inequality for $K_1^{+,0}(t;A_0)$ holds.
Summing up the arguments above, we see that the following lemma holds.
\begin{lemm}\label{lemm:fund3}
Let $1\leq r\leq 2\leq q\leq\infty$, $\tau >0$, and $s_i>0$ $(i=1,2)$. For $x_N>0$ and $f\in L_r(\uhs)$, we set
\begin{equation*}
F(x_N,\tau)=\int_0^\infty\left\|e^{-s_1 A^2 \tau}A e^{-s_2 A(x_N+y_N)}\wh{f}(\xi^\pr,y_N)\right\|_{L_2(\tws)}\,dy_N.
\end{equation*}
Then there exits a positive constant $C$ such that for any $\tau>0$
\begin{equation*}
\|F(\tau)\|_{L_q((0,\infty))}\leq C \tau^{-\frac{N-1}{2}\left(\frac{1}{r}-\frac{1}{2}\right)-\frac{1}{2}\left(\frac{1}{r}-\frac{1}{q}\right)}
\|f\|_{L_r(\uhs)}.
\end{equation*}
\end{lemm}

Secondly we show the inequality for $K_{2}^{+,0}(t;A_0)$.
We here set
\begin{equation*}
\CM_\pm(a)=\frac{e^{-B_1^\pm a}-e^{-Aa}}{B_1^\pm-A}\quad\text{for $a>0$}.
\end{equation*}
In view of \eqref{expan:low} and \eqref{expan:low2}, we can choose $A_0 \in (0, A_1)$ in such a way that
for any $A\in (0,A_0)$ and $a>0$
\begin{equation}\label{est:M}
|\CM_\pm(a)|=\frac{|e^{-B_1^\pm a}-e^{-Aa}|}{|B_1^\pm-A|}\leq CA^{-1/4}e^{-Aa}
\end{equation}
with some constant $C$.
Thus, by the same calculations as in (\ref{0810_1}) and (\ref{0810_2}),
we obtain
%
%
%
%
\begin{multline*}
\|[K_{2}^{+,0}(t;A_0)f](\cdot\,,x_N)\|_{L_{q}(\tws)} \\
\leq C\wt^{-\frac{N-1}{2}\left(\frac{1}{2}-\frac{1}{q}\right)-\frac{s}{2}}
\int_{0}^{\infty}\left\|e^{-(A^2/2)\wt}Ae^{-A(x_N+y_N)}\wh{f}(y_N)\right\|_2 dy_N,
\end{multline*}
which furnishes the required inequality of $K_2^{+,0}(t;A_0)$ by Lemma \ref{lemm:fund3}.

Thirdly we show the inequality for $K_{3}^{+,0}(t;A_0)$.
In view of \eqref{expan:low} and \eqref{expan:low2}, we can
choose $A_0 \in (0, A_1)$ such that
\begin{equation*}
|e^{-B_{1}^+(x_{N}+y_{N})}|\leq e^{-CA^{1/4}(x_{N}+y_{N})}
\quad\text{for any $A\in(0,A_0)$}
\end{equation*}
with some constant $C$, so that we easily see that by Lemma \ref{lemm:fund1}
\begin{align*}
&\|[K_3^{+,0}(t;A_0)f](\cdot\,,x_N)\|_{L_q(\tws)} \\
&\quad\leq C\wt^{-\frac{N-1}{2}\left(\frac{1}{2}-\frac{1}{q}\right)-\frac{s}{2}}
\int_{0}^{\infty}\left\|e^{-(A^2/2)\wt}A e^{-CA^{1/4}(x_N+y_N)}
\wh{f}(y_N)\right\|_2\,dy_{N} \notag \\
&\quad\leq C\wt^{-\frac{N-1}{2}\left(\frac{1}{r}-\frac{1}{q}\right)-\frac{s}{2}}
\int_{0}^{\infty}\frac{\|f(\cdot\,,y_{N})\|_{L_{r}(\tws)}}{\wt^{\,1/2}+(x_{N})^4+(y_{N})^4}\,dy_{N}.
\end{align*}
Combining the inequality above with Lemma \ref{lemm:fund2} (2) with $a=1/2$ and $b_1=b_2=4$,
we obtain the required inequality of $K_3^{+,0}(t;A_0)$.

Finally we show the inequalities for $K_n^{+,0}(t;A_0)$ $(n=4,5,6)$.
Using similar argumentations to the above cases, we have for $n=4,5$
\begin{align*}
\|[K_{4}^{+,0}(t;A_0)f](\cdot\,,x_{N})\|_{L_{q}(\tws)}
&\leq C \wt^{-\frac{N-1}{2}\left(\frac{1}{r}-\frac{1}{q}\right)-\frac{s}{2}}
\int_{0}^{\infty}\frac{\|f(\cdot\,,y_{N})\|_{L_{r}(\tws)}}
{\wt^{\,1/2}+(x_N)^4+y_N}\,dy_{N}, \\
\|[K_{5}^{+,0}(t;A_0)f](\cdot\,,x_{N})\|_{L_{q}(\tws)}
&\leq C\wt^{-\frac{N-1}{2}\left(\frac{1}{r}-\frac{1}{q}\right)-\frac{s}{2}}
\int_{0}^{\infty}\frac{\|f(\cdot\,,y_{N})\|_{L_{r}(\tws)}}
{\wt^{\,1/2}+x_N+(y_N)^4}\,dy_{N},
\end{align*}
which, combined with Lemma \ref{lemm:fund2} (2), furnishes the required inequalities of $K_n^{+,0}(t;A_0)$ $(n=4,5)$.
In addition, for $n=6$, we have
\begin{multline*}
\|[K_{6}^{+,0}(t;A_0)f](\cdot\,,x_{N})\|_{L_{q}(\tws)} \\
\leq C\wt^{-\frac{N-1}{2}\left(\frac{1}{2}-\frac{1}{q}\right)-\frac{s}{2}}
\int_0^\infty\left\|e^{-(A^2/2)\wt}A e^{-CA(x_N+y_N)}\wh{f}(y_N)\right\|_2 dy_N
\end{multline*}
with a positive constant $C$,
which yields the required inequality of $K_{6}^{+,0}(t;A_0)$ by Lemma \ref{lemm:fund3}.
\\
(2) 
We consider the case of $n=1,3$.
Noting that $B^2-(B_1^\pm)^2=\la-\la_\pm$, by the residue theorem, we have
\begin{equation*}
[L_{n}^{+,0}(t;A_0)d](x)
=4\pi i\CF_{\xi^\pr}^{-1}\left[\frac{e^{\la_+ t}\ph_0(\xi^\pr)m_{n}(\xi^\pr,\la_+)B_1^+}
{(B_1^+-B_1^-)(B_1^+-B_2^-)(B_1^+-B_2^-)}
\CY_n(x_N)\wh{d}(\xi^\pr)\right](x^\pr).
\end{equation*}
Thus, by (\ref{0810_6}), \eqref{est:M}, Lemma \ref{lemm:fund1}, $L_q\text{-}L_r$ estimates of
the $(N-1)$-dimensional heat kernel, and Parseval's theorem, we have
\begin{align}\label{140918_4}
&\|[L_{n}^{+,0}(t;A_0)d](\cdot\,,x_{N})\|_{L_{q}(\tws)} \notag \displaybreak[0] \\
&\quad\leq C\wt^{-\frac{N-1}{2}\left(\frac{1}{2}-\frac{1}{q}\right)-\frac{s}{2}}
\|e^{-(A^2/2)\wt}A^{1/2}e^{-Ax_N}\wh{d}(\xi^\pr)\|_{2} \notag \displaybreak[0] \\
&\quad\leq C\wt^{-\frac{N-1}{2}\left(\frac{1}{2}-\frac{1}{q}\right)-\frac{s}{2}}
\|e^{-(A^2/4)\wt}\ \wh{d}(\xi^\pr)\|_{2}/(\wt^{\,1/4}+(x_{N})^{1/2}) \notag \displaybreak[0] \\
&\quad\leq C\wt^{-\frac{N-1}{2}\left(\frac{1}{r}-\frac{1}{q}\right)-\frac{s}{2}}
\|d\|_{L_{r}(\tws)}/(\wt^{\,1/4}+(x_{N})^{1/2}).
\end{align}
If $q>2$, then by Lemma \ref{lemm:fund2} (1) we obtain the required inequality of $L_{n}^{+,0}(t;A_0)$ $(n=1,3)$.
In the case of $q=2$, we see that by the first inequality of \eqref{140918_4}
\begin{align*}
\|L_{n}^{+,0}(t;A_0)d\|_{L_2(\uhs)}
\leq C\wt^{-\frac{s}{2}}\|e^{-(A^2/2)\wt}\wh{d}(\xi^\pr)\|_2
\leq C\wt^{-\frac{N-1}{2}\left(\frac{1}{r}-\frac{1}{2}\right)-\frac{s}{2}}\|d\|_{L_r(\tws)}.
\end{align*}
Analogously, we can obtain the required inequality of $L_{2}^{+,0}(t;A_0)$,
which complete the proof of the lemma.
\end{proof}
Noting that for some $A_2\in(0,1)$ and $C>0$ there holds
$|D(A,B_1^\pm)|\geq C A^{3/4}$ for $A\in(0,A_2)$,
we see that there exist positive numbers $A_1\in(0,A_2)$ and $C$ such that for any $A\in(0,A_1)$ and $j,k=1,\dots,N$
\begin{equation*}
\begin{aligned}
|\CV_{jk}^{BB}(\xi^\pr,\la_\pm)|&\leq CA^{\frac{6}{4}},
&|\CV_{jk}^{B\CM}(\xi^\pr,\la_\pm)|&\leq CA^{\frac{7}{4}},
&|\CV_{jk}^{\CM B}(\xi^\pr,\la_\pm)|&\leq CA^{\frac{7}{4}},\\
|\CV_{jk}^{\CM\CM}(\xi^\pr,\la_\pm)|&\leq C A^{\frac{8}{4}},
&|\CP_{j}^{AA}(\xi^\pr,\la_\pm)|&\leq CA,
&|\CP_{j}^{A\CM}(\xi^\pr,\la_\pm)|&\leq CA^{\frac{5}{4}}.
\end{aligned}
\end{equation*}
%
%
%
%
%
%
Therefore, recalling the formulas \eqref{Fsol_eq:2}, \eqref{symbol}, \eqref{hFsol_eq:2}, and \eqref{low:decomp} with $\si=0$
and using \eqref{deriv:M},
we obtain the required inequalities of Theorem \ref{thm:Gam0} by Lemma \ref{lem:Gam0}.

\subsection{Analysis on $\Gamma_{1}^{\pm}$}
Our aim here is to show the following theorem for the operators defined in (\ref{low:decomp}) with $\si=1$.
\begin{theo}\label{thm:Gam1}
Let $1\leq r\leq 2\leq q\leq\infty$ and $F=(f,d)\in L_r(\uhs)^N \times L_r(\tws)$.
Then, there exists an $A_0 \in (0, 1)$ such that we have the following assertions:
\begin{enumerate}[$(1)$]
\item
Let $k=0,1$, $\ell=0,1,2$, and $\alpha^{\prime}\in\mathbf{N}_{0}^{N-1}$.
Then there exists a positive constant $C=C(\alpha^{\prime})$ such that for any $t>0$
\begin{align*}
\|\pa_{t}^{k}D_{x^{\prime}}^{\alpha^{\prime}}D_{N}^{\ell}S_{0}^{f,1}(t;A_0)F\|_{L_{q}(\mathbf{R}_{+}^{N})}
&\leq C(t+1)^{-\frac{N}{2}\left(\frac{1}{r}-\frac{1}{q}\right)-\frac{2k+|\alpha^{\prime}|+\ell}{2}}
\|f\|_{L_{r}(\mathbf{R}_{+}^{N})}, \\
\|\pa_{t}^{k}D_{x^{\prime}}^{\alpha^{\prime}}D_{N}^{\ell}S_{0}^{d,1}(t;A_0)F\|_{L_{q}(\mathbf{R}_{+}^{N})}
&\leq C(t+1)^{-\frac{N-1}{2}\left(\frac{1}{r}-\frac{1}{q}\right)-\frac{1}{2}\left(\frac{1}{2}-\frac{1}{q}\right)
-\frac{3}{4}-\frac{2k+|\alpha^{\prime}|+\ell}{2}}\|d\|_{L_{r}(\tws)}.
\end{align*}
\item
There exists a positive constant $C$ such that for any $t>0$
\begin{align*}
\|\nabla\Pi_{0}^{f,1}(t;A_0)F\|_{L_{q}(\uhs)}&\leq C(t+1)^{-\frac{N}{2}\left(\frac{1}{r}-\frac{1}{q}\right)-1}
\|f\|_{L_{r}(\mathbf{R}_{+}^{N})}, \\
\|\nabla \Pi_{0}^{d,1}(t;A_0)F\|_{L_{q}(\uhs)}&\leq C(t+1)^{-\frac{N-1}{2}\left(\frac{1}{r}-\frac{1}{q}\right)
-\frac{1}{2}\left(\frac{1}{2}-\frac{1}{q}\right)-\frac{7}{4}}\|d\|_{L_{r}(\mathbf{R}^{N-1})}.
\end{align*}
\item
Let $\alpha\in\BN_{0}^{N}$. Then there exists a positive constant $C=C(\alpha)$
such that for any $t>0$
\begin{align*}
\|D_{x}^{\alpha}\nabla \CE(T_{0}^{f,1}(t;A_0)F)\|_{L_{q}(\mathbf{R}_{+}^{N})}&\leq 
C(t+1)^{-\frac{N}{2}\left(\frac{1}{r}-\frac{1}{q}\right)-1-\frac{|\alpha|}{2}}\|f\|_{L_{r}(\uhs)}, \\
\|D_{x}^{\alpha}\pa_{t}\CE(T_{0}^{f,1}(t;A_0)F)\|_{L_{q}(\mathbf{R}_{+}^{N})}&\leq
C(t+1)^{-\frac{N}{2}\left(\frac{1}{r}-\frac{1}{q}\right)-\frac{3}{2}-\frac{|\alpha|}{2}}
\|f\|_{L_{r}(\mathbf{R}_{+}^{N})}, \\
\|D_{x}^{\alpha}\nabla\CE(T_{0}^{d,1}(t;A_0)F)\|_{L_{q}(\mathbf{R}_{+}^{N})}&\leq
C(t+1)^{-\frac{N-1}{2}\left(\frac{1}{r}-\frac{1}{q}\right)-\frac{1}{2}\left(\frac{1}{2}-\frac{1}{q}\right)
-\frac{7}{4}-\frac{|\alpha|}{2}}\|d\|_{L_{r}(\tws)}.
\end{align*}
\end{enumerate}
\end{theo}
We start with the following lemmas in order to show Theorem \ref{thm:Gam1}.
\begin{lemm}\label{lem:poly}
Let $f(z)=z^{3}+2z^{2}+12z-8$.
Then $f(z)\neq 0$ for $z\in\{\omega\in\mathbf{C}\mid{\rm Re}\,\omega\geq0\}\setminus(0,1)$.
\end{lemm}
\begin{proof}
We note that $f(z)$ has only one real root $\al$ because $f(0)=-8$, $f(1)=7$ and $f^\pr(z)=3z^{2}+4z+12>0$ for $z\in\BR$,
and it is clear that $\al$ is in $(0,1)$.
Let $\beta$ and $\bar{\beta}$ be the other roots of $f(z)$.
Since $\alpha +\beta +\bar{\beta}=-2$, 
we have $2{\rm Re}\beta=-2-\alpha<0$.
This completes the proof.
\end{proof}
\begin{lemm}\label{lem1:Gam1}
Let $\lambda\in\Gamma_{1}^{\pm}$ and $\xi^\pr\in\tws$.
Then
\begin{equation*}
\frac{A}{4}\leq{\rm Re}B\leq|B|\leq\frac{A}{2},\quad
|D(A,B)|\geq CA^{3}
\end{equation*}
for some positive constant $C$ independent of $\xi^\pr$ and $\la$.
In addition, there exist positive constants $A_1\in(0,1)$ and $C$ such that $|L(A,B)|\geq CA$ for any $A\in(0,A_1)$.
\end{lemm}
\begin{proof}
We first show the inequalities for $B$ and $D(A,B)$.
Note that
\begin{equation}\label{0920_3}
B=\sqrt{\lambda+A^{2}}=(A/2)e^{\pm i(u/2)}
\end{equation}
since $\la=-A^{2}+(A^{2}/4)e^{\pm iu}$ for $u\in[0,\pi/2]$ on $\Ga_1^\pm$.
Therefore, it is clear that the required inequalities of $B$ hold.
We insert the identity \eqref{0920_3} into $D(A,B)$ to obtain
\begin{equation*}
D(A,B)=\frac{A^{3}}{8}\left((e^{\pm i(u/2)})^{3}+2(e^{\pm i(u/2)})^{2}+12(e^{\pm i(u/2)})-8\right),
\end{equation*}
which, combined with Lemma \ref{lem:poly}, furnishes that
$|D(A,B)|\geq CA^{3}$ for some positive constant $C$ independent of $\xi^\pr$ and $\la$.

We finally show the last inequality.
By (\ref{expan:low2})
\begin{equation*}
B^{2}-(B_{1}^{\pm})^{2}=\mp ic_{g}^{1/2}A^{1/2}+A^{2}\left(1+\frac{e^{\pm iu}}{4}\right)+O(A^{10/4})
\quad\text{as $A\to0$},
\end{equation*}
so that there exist positive constants $A_1\in(0,1)$ and $C$ such that
\begin{equation}\label{ineq:1}
|B^{2}-(B_{1}^{\pm})^{2}|\geq  CA^{1/2}\quad\text{for any $A\in(0,A_1)$}.
\end{equation}
On the other hand, we have $|B+B_1^\pm|\leq  CA^{1/4}$ on $\Ga_1^\pm$ when $A$ is sufficiently small,
which, combined with \eqref{ineq:1}, furnishes that
%
%
\begin{equation*}
|B-B_{1}^{\pm}|=\frac{|B^{2}-(B_{1}^{\pm})^{2}|}{|B+B_{1}^{\pm}|}\geq CA^{1/4} \quad\text{for any $A\in(0,A_1)$}.
\end{equation*}
Since $|B-B_{1}^{\pm}|\leq|B-B_{2}^{\pm}|$ as follows from ${\rm Re}B\geq0$ and (\ref{expan:low}), we thus obtain
\begin{equation*}
|L(A,B)|=|(B-B_1^+)(B-B_1^-)(B-B_2^+)(B-B_2^-)|\geq CA
\end{equation*}
for any $A\in(0,A_1)$, $\lambda\in\Gamma_1^\pm$, and a positive constant $C$ independent of $\xi^\prime$ and $\lambda$.
\end{proof}
Next, we show some multiplier theorem on $\Ga_1^\pm$.
\begin{lemm}\label{lem2:Gam1}
Let $1\leq r\leq 2\leq q\leq\infty$, and let $f\in L_{r}(\mathbf{R}_{+}^{N})^{N}$ and $d\in L_{r}(\mathbf{R}^{N-1})$.
We use the symbols defined in $(\ref{op:low})$.
\begin{enumerate}[$(1)$]
\item
Let $s\geq0$ and suppose that there exist
constants $A_1\in(0,1)$ and  $C=C(s)>0$ such that
for any $\lambda\in\Gamma_{1}^{\pm}$ and $A\in(0,A_1)$
\begin{align*}
&|k_{n}(\xi^\prime,\lambda)|\leq CA^{-1+s}\enskip (n=1,3),\quad
|k_{n}(\xi^\prime,\lambda)|\leq CA^{s}\enskip (n=2,4,5), \\
&|k_{6}(\xi^\prime,\lambda)|\leq CA^{1+s}.
\end{align*}
Then there exist constants $A_0 \in (0, A_1)$ and  $C=C(s)>0$ such that for any $t>0$ we have the estimates:
\begin{align*}
\|K_{n}^{\pm,1}(t;A_0)f\|_{L_{q}(\mathbf{R}_{+}^{N})}&\leq C(t+1)^{-\frac{N}{2}\left(\frac{1}{r}-\frac{1}{q}\right)
-\frac{s}{2}}\|f\|_{L_{r}(\mathbf{R}_{+}^{N})}\quad(n=1,2,3,4,5,6).
\end{align*}
\item
Let $s\geq0$ and suppose that there exist
constants $A_1\in(0,1)$ and  $C=C(s)>0$ such that
for any $\lambda\in\Gamma_{1}^{\pm}$ and $A\in(0,A_1)$
\begin{equation*}
|\ell_{n}(\xi^\prime,\lambda)|\leq CA^{s}\ (n=1,2),\quad
|\ell_{3}(\xi^\prime,\lambda)|\leq CA^{1+s}.
\end{equation*}
Then there exist constants $A_0 \in (0, A_1)$ and $C=C(s)>0$ such that for any $t>0$ we have the estimates:
\begin{align*}
\|L_{n}^{\pm,1}(t;A_0)d\|_{L_{q}(\mathbf{R}_{+}^{N})}&\leq C(t+1)^{-\frac{N-1}{2}\left(\frac{1}{r}-\frac{1}{q}\right)
-\frac{1}{2}\left(\frac{1}{2}-\frac{1}{q}\right)-\frac{3}{4}-\frac{s}{2}}\|d\|_{L_{r}(\mathbf{R}^{N-1})}\quad(n=1,2,3).
\end{align*}
\end{enumerate}
\end{lemm}
%
%
\begin{proof}
We use the abbreviations: $\|\cdot\|_{2}=\|\cdot\|_{L_{2}(\mathbf{R}^{N-1})}$,
$\widehat{f}(y_{N})=\widehat{f}(\xi^{\prime},y_{N})$,
and $\wt=t+1$ for $t>0$ in this proof,
and consider only the estimates on $\Gamma_1^{+}$
since the estimates on $\Gamma_1^-$ can be shown similarly. \\
(1)
Since $\lambda=-A^{2}+(A^{2}/4)e^{iu}$ for $u\in[0,\pi/2]$ on $\Gamma_{1}^{+}$, we have
\begin{align*}
&[K_{1}^{+,1}(t;A_0)f](x)= \\
&\int_{0}^{\infty}\CF_{\xi^{\prime}}^{-1}\left[\int_{0}^{\frac{\pi}{2}}e^{\left(-A^{2}+(A^{2}/4)e^{iu}\right)t}
\varphi_{0}(\xi^{\prime})k_{1}(\xi^\prime,\lambda)e^{-A(x_{N}+y_{N})}\frac{iA^{2}}{4}e^{iu}\,du
\widehat{f}(y_{N})\right](x^{\prime})\,dy_{N}.
\end{align*}
Noting that $|e^{\left(-A^{2}+(A^{2}/4)e^{iu}\right)t}|\leq C e^{-(3/4)A^{2}\wt}$
for some positive constant $C$ independent of $\xi^\pr$, $u$, and $t$,
%
%
%
%
we see that by Lemma \ref{lemm:fund1}, $L_q\text{-}L_r$ estimates of the $(N-1)$-dimensional heat kernel, and Parseval's theorem
\begin{multline*}
\|[K_{1}^{+,1}(t;A_0)f](\cdot\,,x_{N})\|_{L_{q}(\mathbf{R}^{N-1})} \\
\leq C\wt^{-\frac{N-1}{2}\left(\frac{1}{2}-\frac{1}{q}\right)-\frac{s}{2}}\int_{0}^{\infty}
\left\|e^{-(A^2/2)\wt}A e^{-A(x_{N}+y_{N})}\widehat{f}(y_{N})\right\|_{2}dy_{N},
\end{multline*}
and furthermore, for $n=2,3,4,5,6$
\begin{multline*}
\|[K_{n}^{+,1}(t;A_0)f](\cdot\,,x_{N})\|_{L_{q}(\mathbf{R}^{N-1})} \\
\leq C\wt^{-\frac{N-1}{2}\left(\frac{1}{2}-\frac{1}{q}\right)-\frac{s}{2}}\int_{0}^{\infty}
\left\|e^{-(A^2/2)\wt}A e^{-CA(x_{N}+y_{N})}\widehat{f}(y_{N})\right\|_{2}dy_{N}
\end{multline*}
with some positive constant $C$ analogously,
 where we have used the fact that for $a>0$ and $\la\in\Ga_1^\pm$
\begin{equation}\label{est:M:Gam1}
|\mathcal{M}(a)|\leq a\int_{0}^{1}|e^{-\left(B\theta+A(1-\theta)\right)y_{N}}|\,d\theta\leq ae^{-(A/4)a}
\leq 8A^{-1}e^{-(A/8)a}.
\end{equation}
We thus obtain the required inequality for $K_{n}^{+,1}(t;A_0)$ $(n=1,\dots,6)$ by Lemma \ref{lemm:fund3}. \\
%
%
%
%
%
%
%
(2)
Since $\lambda=-A^{2}+(A^{2}/4)e^{iu}$ for $u\in[0,\pi/2]$ on $\Gamma_{1}^{+}$, we have
\begin{align*}
&[L_{1}^{+,1}(t;A_0)d](x) \\
&\quad=\CF_{\xi^{\prime}}^{-1}\left[\int_{0}^{\frac{\pi}{2}}e^{\left(-A^{2}+(A^{2}/4)e^{iu}\right)t}
\varphi_{0}(\xi^{\prime})\ell_{1}(\xi^\prime,\lambda)e^{-Ax_{N}}\frac{i A^{2}}{4}e^{iu}\,du
\,\widehat{d}(\xi^{\prime})\right](x^{\prime}).
\end{align*}
By calculations similar to the case of $K_{1}^{+,1}(t)$ and Lemma \ref{lemm:fund1},
\begin{align*}
&\|[L_{1}^{+,1}(t;A_0)d](\cdot,x_{N})\|_{L_{q}(\tws)} \\
&\quad\leq C\wt^{-\frac{N-1}{2}\left(\frac{1}{2}-\frac{1}{q}\right)-\frac{1}{2}-\frac{s}{2}}
\int_{0}^{\frac{\pi}{2}}\left\|e^{-(A^2/2)\wt}Ae^{-Ax_{N}}\widehat{d}(\xi^{\prime})\right\|_{2}\,du \\
&\quad\leq C\wt^{-\frac{N-1}{2}\left(\frac{1}{r}-\frac{1}{q}\right)-\frac{1}{2}-\frac{s}{2}}
\|d\|_{L_{r}(\mathbf{R}^{N-1})}/(\wt^{\,1/2}+x_{N}),
\end{align*}
and also for $n=2,3$ we have by \eqref{est:M:Gam1}
\begin{equation*}
\|[L_{n}^{+,1}(t;A_0)d](\cdot,x_{N})\|_{L_{q}(\tws)}
\leq C\wt^{-\frac{N-1}{2}\left(\frac{1}{r}-\frac{1}{q}\right)-\frac{1}{2}-\frac{s}{2}}
\|d\|_{L_{r}(\mathbf{R}^{N-1})}/(\wt^{\,1/2}+x_{N}).
\end{equation*}
We thus obtain the required inequality for $L_{n}^{+,1}(t;A_0)$ $(n=1,2,3)$ by Lemma \ref{lemm:fund2} (1).
\end{proof}
By Lemma \ref{lem1:Gam1} we see that there exist
constants $A_1\in(0,1)$ and $C>0$ such that
for any $\lambda\in\Gamma_1^\pm$, $A\in (0,A_1)$, and $j,k=1,\dots,N$
\begin{equation*}
\begin{aligned}
&|\CV_{jk}^{BB}(\xi^\prime,\lambda)/L(A,B)|\leq CA^{-1},
&&|\CV_{jk}^{B\CM}(\xi^\prime,\lambda)/L(A,B)|\leq C, \\
&|\CV_{jk}^{\CM B}(\xi^\prime,\lambda)/L(A,B)|\leq C,
&&|\CV_{jk}^{\CM\CM}(\xi^\prime,\lambda)/L(A,B)|\leq CA, \\
&|\CP_j^{AA}(\xi^\prime,\lambda)/L(A,B)|\leq C,
&&|\CP_j^{A\CM}(\xi^\prime,\lambda)/L(A,B)|\leq CA.
\end{aligned}
\end{equation*}
Therefore, recalling (\ref{Fsol_eq:2})-(\ref{hFsol_eq:2}) 
and (\ref{low:decomp}) with $\sigma=1$ and using Lemma \ref{lem2:Gam1},
we have Theorem \ref{thm:Gam1}.

\subsection{Analysis on $\Gamma_{2}^{\pm}$}
Our aim here is to show the following theorem for the operators defined in (\ref{low:decomp}) with $\si=2$.
\begin{theo}\label{thm:Gam2}
Let $1\leq r\leq 2\leq q\leq\infty$ and $F=(f,d)\in L_r(\uhs)^N\times L_r(\tws)$.
Then there exists an $A_0\in(0,1)$ such that the following assertions hold:
\begin{enumerate}[$(1)$]
\item
Let $k=0,1,\ell=0,1,2$, and $\alpha^{\prime}\in\mathbf{N}_{0}^{N-1}$.
Then there exists a positive constant $C=C(\alpha^{\prime})$ 
such that for any $t>0$
\begin{align*}
\|\pa_t^k D_{x^{\prime}}^{\alpha^{\prime}}D_{N}^{\ell}
S_{0}^{f,2}(t;A_0)F\|_{L_{q}(\uhs)}&\leq
C(t+1)^{-\frac{N}{2}\left(\frac{1}{r}-\frac{1}{q}\right)-k-
\frac{|\alpha^{\prime}|+\ell}{2}}\|f\|_{L_r(\uhs)}, \\
\|\pa_t^{k}D_{x^{\prime}}^{\alpha^{\prime}}
D_{N}^{\ell}S_{0}^{d,2}(t;A_0)F\|_{L_{q}(\uhs)}&\leq
C(t+1)^{-\frac{N-1}{2}\left(\frac{1}{r}-\frac{1}{q}\right)
-\frac{1}{2}\left(\frac{1}{2}-\frac{1}{q}\right)
-k-\frac{|\alpha^{\prime}|+\ell}{2}}\|d\|_{L_r(\tws)},
\end{align*}
provided that $k+\ell+|\al^\pr|\neq0$.
In addition, if $(q,r)\neq(2,2)$, then
\begin{align*}
\|S_{0}^{f,2}(t;A_0)F\|_{L_q(\uhs)}
&\leq 
C(t+1)^{-\frac{N}{2}\left(\frac{1}{r}-\frac{1}{q}\right)}\|f\|_{L_r(\uhs)}, \\
\|S_{0}^{d,2}(t;A_0)F\|_{L_{q}(\uhs)}
&\leq
C(t+1)^{-\frac{N-1}{2}\left(\frac{1}{r}-\frac{1}{q}\right)
-\frac{1}{2}\left(\frac{1}{2}-\frac{1}{q}\right)}\|d\|_{L_r(\tws)}.
\end{align*}
\item
There exists a positive constant $C$ such that for any $t>0$
\begin{align*}
\|\nabla\Pi_{0}^{f,2}(t;A_0)F\|_{L_{q}(\uhs)}&
\leq C(t+1)^{-\frac{N}{2}\left(\frac{1}{r}-\frac{1}{q}\right)
-\frac{1}{4}}\|f\|_{L_{r}(\mathbf{R}_{+}^{N})}, \\
\|\nabla\Pi_{0}^{d,2}(t;A_0)F\|_{L_{q}(\mathbf{R}_{+}^{N})}&
\leq C(t+1)^{-\frac{N-1}{2}\left(\frac{1}{r}-\frac{1}{q}\right)
-\frac{1}{2}\left(\frac{1}{2}-\frac{1}{q}\right)
-1}\|d\|_{L_{r}(\tws)}.
\end{align*}
\item
Let $\alpha\in\mathbf{N}_{0}^{N}$. 
Then there exists a positive constant 
$C=C(\alpha)$ such that for any $t>0$
\begin{align*}
\|D_{x}^{\alpha}\nabla\CE({T}_{0}^{f,2}(t;A_0)F)\|_{L_{q}(\uhs)}
&\leq C(t+1)^{-\frac{N}{2}\left(\frac{1}{r}-\frac{1}{q}\right)
-\frac{1}{4}-\frac{|\alpha|}{2}}\|f\|_{L_{r}(\uhs)}, \\
\|D_{x}^{\alpha}\pa_t\CE({T}_{0}^{f,2}(t;A_0)F)\|_{L_{q}(\mathbf{R}_{+}^{N})}
&\leq C(t+1)^{-\frac{N}{2}\left(\frac{1}{r}-\frac{1}{q}\right)
-\frac{|\alpha|}{2}}\|f\|_{L_{r}(\mathbf{R}_{+}^{N})}\quad\text{if $|\al|\neq 0$}, \\
\|D_{x}^{\alpha}\nabla\CE({T}_{0}^{d,2}(t;A_0)F)\|_{L_{q}(\uhs)}
&\leq C(t+1)^{-\frac{N-1}{2}\left(\frac{1}{r}-\frac{1}{q}\right)
-\frac{1}{2}\left(\frac{1}{2}-\frac{1}{q}\right)
-1-\frac{|\alpha|}{2}}\|d\|_{L_{r}(\tws)}.
\end{align*}
In addition, if $(q,r)\neq(2,2)$, then
\begin{equation*}
\|\pa_t\CE({T}_{0}^{f,2}(t;A_0)F)\|_{L_{q}(\mathbf{R}_{+}^{N})}
\leq C(t+1)^{-\frac{N}{2}\left(\frac{1}{r}-\frac{1}{q}\right)}\|f\|_{L_{r}(\mathbf{R}_{+}^{N})}.
\end{equation*}
\end{enumerate}
\end{theo}
We start with the following lemma in order 
to show Theorem \ref{thm:Gam2}.
\begin{lemm}\label{lem1:Gam2}
There exist positive constants $A_1\in(0,1)$, $b_{0}\geq1$, and $C$ 
such that for any $\lambda\in\Gamma_{2}^{\pm}$
and $A\in(0,A_1)$
\begin{align*}
&b_{0}^{-1}(A\sqrt{1-u}+\sqrt{u}+A)\leq{\rm Re}B\leq|B|\leq b_{0}(A\sqrt{1-u}+\sqrt{u}+A), \\
&|D(A,B)|\geq C(A\sqrt{1-u}+\sqrt{u}+A)^{3}, \\
&|L(A,B)|\geq C(A\sqrt{1-u}+\sqrt{u}+A^{1/4})^{4}.
\end{align*}
\end{lemm}
\begin{proof}
We fist show the inequalities for $B$.
Set $\si=\la+A^2$ and $\te=\arg\si$. Noting that
%
%
%
%
$$
\lambda=-(A^{2}(1-u)+\gamma_{0}u)+\pm i
((A^{2}/4)(1-u)+\widetilde{\gamma}_{0}u)
$$
for $u\in[0,1]$ on $\Gamma_{2}^{\pm}$,
we have
\begin{align*}
&|\sigma|+A^{2}(1-u)+\gamma_{0}u-A^{2}\leq 2(A^{2}(1-u)+\gamma_{0}u+A^{2})+\frac{A^{2}}{4}(1-u)+\widetilde{\gamma}_{0}u \\
&\leq 3\max(\gamma_0,\widetilde{\gamma}_0)(A^{2}(1-u)+u+A^{2})\leq 3\max(\gamma_0,\widetilde{\gamma}_0)(A\sqrt{1-u}+\sqrt{u}+A)^{2},
\end{align*}
which is used to obtain
\begin{align*}
{\rm Re}B&=|\sigma|^{\frac{1}{2}}\cos{\frac{\theta}{2}}=\frac{|\sigma|^{\frac{1}{2}}}{\sqrt{2}}(1+\cos\theta)^{\frac{1}{2}}
=\frac{|\sigma|^{\frac{1}{2}}}{\sqrt{2}}\left(\frac{|\sigma|+{\rm Re}\,\sigma}{|\sigma|}\right)^{\frac{1}{2}}
=\frac{1}{\sqrt{2}}\left(\frac{|\sigma|^{2}-({\rm Re}\sigma)^{2}}{|\sigma|-{\rm Re}\sigma}\right)^{\frac{1}{2}} \displaybreak[0] \\
&=\frac{(A^{2}/4)(1-u)+\gamma_{0}u+(A^{2}/8)-(A^{2}/8)(1-u)-(A^{2}/8)u}{\sqrt{2}(|\sigma|+A^{2}(1-u)+\gamma_0u-A^{2})^{1/2}} \displaybreak[0] \\
&\geq\frac{(A^2/8)(1-u)+\gamma_0 u +(A^2/8)-(A_0^2/8)u}
{\sqrt{6}\max(\gamma_0^{1/2},\widetilde{\gamma}_0^{1/2})(A\sqrt{1-u}+\sqrt{u}+A)} \displaybreak[0] \\
&\geq\frac{(1/8)\{A^{2}(1-u)+\gamma_{0}u+A^{2}\}}{\sqrt{6}\max(\gamma_0^{1/2},\widetilde{\gamma}_0^{1/2})(A\sqrt{1-u}+\sqrt{u}+A)}
\geq\frac{A\sqrt{1-u}+u+A}{24\sqrt{6}\max(\gamma_0^{1/2},\widetilde{\gamma}_0^{1/2})}
\end{align*}
for any $A\in(0,A_1)$ provided that $A_1^2\leq 7 \gamma_0$.
It is clear that the other inequalities concerning $B$ hold.

Next we consider $D(A,B)$.
Noting that $\lambda\in\Gamma_{2}^{\pm}\subset \Sigma_{\ep_0}$ and using Lemma \ref{lem:symbol} (2),
we obtain
\begin{equation*}
	|D(A,B)|\geq C(\ep_0)(|\lambda|^{\frac{1}{2}}+A)^{3}\geq C(\ep_0)(A\sqrt{1-u}+\sqrt{u}+A)^{3}.
\end{equation*}

Finally, we show the inequality for $L(A,B)$.
By (\ref{expan:low2})
\begin{equation*}
B^{2}-(B_{1}^{\pm})^{2}=
-(A^{2}(1-u)+\gamma_{0}u)\pm i
\left(\frac{A^{2}}{4}(1-u)+\widetilde{\gamma}_{0}u 
-c_{g}^{1/2}A^{1/2}\right)
+2A^{2}+O(A^{10/4})
\end{equation*}
as $A\to0$, and also we have 
\begin{align*}
&\left|-(A^{2}(1-u)+\gamma_{0}u)\pm i\left(\frac{A^{2}}{4}(1-u)+\widetilde{\gamma}_{0}u-c_{g}^{1/2}A^{1/2}\right)\right|^{2} \displaybreak[0] \\
&=(A^{2}(1-u)+\gamma_{0}u)^{2}+\left(\frac{A^{2}}{4}(1-u)+\widetilde{\gamma}_{0}u\right)^{2}+c_{g}A
-2c_{g}^{1/2}A^{1/2}\left(\frac{A^{2}}{4}(1-u)+\widetilde{\gamma}_{0}u\right) \displaybreak[0] \\
&\geq\left(A^{2}(1-u)+\frac{\widetilde{\gamma}_0}{3}u\right)^{2}+\frac{1}{11}c_g A
-\frac{1}{10}\left(\frac{A^2}{4}(1-u)+\widetilde{\gamma}_0u\right)^2  \displaybreak[0] \\
&\geq\frac{1}{90}(A^{2}(1-u)+\widetilde{\gamma}_{0}u)^{2}+\frac{1}{11}c_{g}A
\geq C\left(A\sqrt{1-u}+\sqrt{u}+A^{1/4}\right)^{4}.
\end{align*}
We thus see that 
there exist positive constants $A_1$ and $C$ such that
for any $A\in(0,A_1)$ and $\lambda\in\Gamma_2^\pm$
\begin{multline*}
|B-B_{1}^{\pm}|=\frac{|B^{2}-(B_{1}^{\pm})^{2}|}{|B+B_{1}^{\pm}|} \\
\geq
\frac{C\left(A\sqrt{1-u}+\sqrt{u}+A^{1/4}\right)^{2}}
{b_{0}(A\sqrt{1-u}+\sqrt{u}+A)+c_{g}^{1/4}A^{1/4}}
\geq C(A\sqrt{1-u}+\sqrt{u}+A^{1/4}).
\end{multline*}
Since $|B-B^\pm_1| \leq |B-B^\pm_2|$ as follows from
${\rm Re}\,B \geq 0$ and \eqref{expan:low}, we have
the required inequality for $L(A,B)$, which completes the proof 
 of Lemma \ref{lem1:Gam2}.
\end{proof}
\begin{lemm}\label{lem2:Gam2}
Let $1\leq r\leq  2\leq q\leq\infty$, and let $f\in L_{r}(\mathbf{R}_{+}^{N})^{N}$ and $d\in L_{r}(\mathbf{R}^{N-1})$.
We use the symbols defined in $(\ref{op:low})$.
\begin{enumerate}[$(1)$]
\item
Let $s\geq0$ and suppose that there 
exist constants $A_1\in(0,1)$ and  $C=C(s)>0$ such that
for any $\lambda\in\Gamma_{2}^{\pm}$ and $A\in(0,A_1)$
\begin{equation*}
\begin{aligned}
|k_{n}(\xi^\prime,\lambda)|&\leq C(A\sqrt{1-u}
+\sqrt{u}+A)^{-2}A|B|^s &&(n=1,3), \\
|k_{2}(\xi^\prime,\lambda)|&\leq C(A\sqrt{1-u}
+\sqrt{u}+A)^{-2}A^2|B|^s, \\
|k_{n}(\xi^\prime,\lambda)|&\leq C(A\sqrt{1-u}
+\sqrt{u}+A)^{-1}A|B|^s && (n=4,5), \\
|k_{6}(\xi^\prime,\lambda)|&\leq CA|B|^{s}.
\end{aligned}
\end{equation*}
Then there exist constants $A_0 \in (0, A_1)$ and
$C=C(s)>0$ such that for any $t>0$ and $n=1,\dots,6$ we have 
the estimates:
\begin{equation*}
\|K_{n}^{\pm,2}(t;A_0)f\|_{L_{q}(\uhs)}\leq C(t+1)^{-\frac{N}{2}\left(\frac{1}{r}-\frac{1}{q}\right)-\frac{s}{2}}
\|f\|_{L_{r}(\uhs)},
\end{equation*}
provided that $s>0$.
In the case of $s=0$, we have
\begin{equation*}
\|K_{n}^{\pm,2}(t;A_0)f\|_{L_{q}(\uhs)}\leq C(t+1)^{-\frac{N}{2}\left(\frac{1}{r}-\frac{1}{q}\right)}
\|f\|_{L_{r}(\uhs)}\quad\text{if $(q,r)\neq(2,2)$}.
\end{equation*}
\item
Let $s\geq0$ and suppose that there exist constants $A_1\in(0,1)$ 
and  $C=C(s)>0$ such that
for any $\lambda\in\Gamma_{2}^{\pm}$ and $A\in(0,A_1)$ 
\begin{equation*}
\begin{aligned}
|\ell_{n}(\xi^\prime,\lambda)|&\leq C(A\sqrt{1-u}
+\sqrt{u}+A^{1/4})^{-4}A|B|^s &&(n=1,2), \\
|\ell_{3}(\xi^\prime,\lambda)|&\leq C(A\sqrt{1-u}
+\sqrt{u}+A^{1/4})^{-3}A|B|^{s}.
\end{aligned}
\end{equation*}
Then there exist constants $A_0 \in (0, A_1)$ and 
$C=C(s)>0$ such that for any $t>0$ we have the estimates:
\begin{align*}
\|L_{n}^{\pm,2}(t;A_0)d\|_{L_{q}(\uhs)}&\leq C(t+1)^{-\frac{N-1}{2}
\left(\frac{1}{r}-\frac{1}{q}\right)
-\frac{1}{2}\left(\frac{1}{2}-\frac{1}{q}\right)
-\frac{s}{2}}\|d\|_{L_{r}(\tws)} \quad(n=1,3,\ s>0), \\
\|L_{2}^{\pm,2}(t;A_0)d\|_{L_{q}(\uhs)}
&\leq C(t+1)^{-\frac{N-1}{2}\left(\frac{1}{r}-\frac{1}{q}\right)
-\frac{1}{2}\left(\frac{1}{2}-\frac{1}{q}\right)
-\frac{3}{4}-\frac{s}{2}}\|d\|_{L_{r}(\tws)}\quad(s\geq0).
\end{align*}
%
%
In the case of $s=0$, for $n=1,3$, we have
\begin{equation*}
\|L_{n}^{\pm,2}(t;A_0)d\|_{L_{q}(\uhs)}\leq C(t+1)^{-\frac{N-1}{2}
\left(\frac{1}{r}-\frac{1}{q}\right)
-\frac{1}{2}\left(\frac{1}{2}-\frac{1}{q}\right)}\|d\|_{L_{r}(\tws)}
\quad\text{if $(q,r)\neq(2,2)$}.
\end{equation*}
\end{enumerate}
\end{lemm}
\begin{proof}
We use the abbreviations: $\|\cdot\|_{2}=\|\cdot\|_{L_{2}(\mathbf{R}^{N-1})}$,
$\widehat{f}(y_{N})=\widehat{f}(\xi^{\prime},y_{N})$
and $\wt=t+1$ for $t>0$ in this proof,
and consider only the estimates on $\Gamma_2^{+}$
since the estimates on $\Gamma_2^-$ can be shown similarly. \\
(1) We first show the inequality for $K_{1}^{+,2}(t;A_0)$.
Recalling that
$\lambda=-(A^{2}(1-u)+\gamma_{0}u)+i((A^{2}/4)(1-u)+\widetilde{\gamma}_{0}u)$ for $u\in[0,1]$ on $\Gamma_{2}^{+}$,
we have, by \eqref{op:low},
\begin{align*}
&[K_{1}^{+,2}(t;A_0)f](x) \\
&=\int_{0}^{\infty}\CF_{\xi^{\prime}}^{-1}\left[\int_{0}^{1}
e^{\{-(A^{2}(1-u)+\gamma_{0}u)+i((A^{2}/4)(1-u)+\widetilde{\gamma}_{0}u)\}t}\varphi_{0}(\xi^{\prime})
k_{1}(\xi^\prime,\lambda)e^{-A(x_{N}+y_{N})}\right. \displaybreak[0] \\
&\quad\left.\times\left\{-(\gamma_{0}-A^{2})+ i\left(\widetilde{\gamma}_{0}-\frac{A^{2}}{4}\right)\right\}\,du
\,\widehat{f}(y_{N})\right](x^{\prime})\,dy_{N}.
\end{align*}
Since it follows from Lemma \ref{lem1:Gam2} that
%
%
%
%
\begin{align}\label{140922_2}
&|e^{\{-(A^{2}(1-u)+\gamma_{0}u)+ i((A^{2}/4)(1-u)+\widetilde{\gamma}_{0}u)\}t}| \notag \\ 
&\quad\leq e^{-\frac{3}{4}A^{2}t}e^{-\frac{1}{4}(A^{2}(1-u)+\gamma_{0}u)t}
\leq Ce^{-\frac{3}{4}A^{2}\wt}e^{-C|B|^{2}\wt}
\end{align}
with some positive constant $C$, independent of $\xi^\prime$, $\lambda$, and $t$,
for any $A\in(0,A_0)$ by choosing suitable $A_0\in(0,A_1)$,
we have, by $L_q\text{-}L_r$ estimates of the $(N-1)$-dimensional heat kernel and Parseval's theorem,
\begin{align}\label{140922_1}
&\|[K_{1}^{+,2}(t;A_0)f](\cdot,x_{N})\|_{L_{q}(\tws)} \notag \\
&\leq C\wt^{-\frac{N-1}{2}\left(\frac{1}{2}-\frac{1}{q}\right)}\int_{0}^{\infty}\left\|\int_{0}^{1}
\frac{e^{-\frac{A^{2}}{2}\wt}e^{-C|B|^{2}\wt}\varphi_{0}(\xi^{\prime})A|B|^s
e^{-A(x_{N}+y_{N})}}{(A\sqrt{1-u}+\sqrt{u}+A)^{2}}\,du\,\widehat{f}(y_{N})\right\|_{2}\,dy_{N}  \\
&\leq C\wt^{-\frac{N-1}{2}\left(\frac{1}{2}-\frac{1}{q}\right)}\int_{0}^{\infty}
\left\|\int_{0}^{1}\frac{e^{-C |B|^2\wt}|B|^{s-\de}\varphi_{0}(\xi^{\prime})}{(\sqrt{u})^{2-\de}}\,du
\,e^{-\frac{A^{2}}{2}\wt}A e^{-A(x_{N}+y_{N})}\widehat{f}(y_{N})\right\|_{2}\,dy_{N} \notag
\end{align}
for a sufficiently small $\de>0$. If $s>0$, then we have,
by Lemma \ref{lem1:Gam2} and Lemma \ref{lemm:fund1} with $Z=|B|$ and $a=0$,
\begin{equation}\label{140922_4}
\int_0^1\frac{e^{-C |B|^2\wt}|B|^{s-\de}\varphi_{0}(\xi^{\prime})}{(\sqrt{u})^{2-\de}}\,du
\leq C \widetilde{t}^{-\frac{s-\de}{2}}\int_0^1\frac{e^{-C u\wt}}{(\sqrt{u})^{2-\de}}\,du
\leq C \widetilde{t}^{-\frac{s}{2}}.
\end{equation}
We thus obtain
\begin{align*}
&\|[K_{1}^{+,2}(t;A_0)f](\cdot,x_{N})\|_{L_{q}(\tws)} \\
&\quad\leq C\wt^{-\frac{N-1}{2}\left(\frac{1}{2}-\frac{1}{q}\right)-\frac{s}{2}}
\int_0^\infty\left\|e^{-(A^2/2)\widetilde{t}}A e^{-A(x_N+y_N)}\wh{f}(y_N)\right\|_2\,dy_N,
\end{align*}
which furnishes the required inequality by Lemma \ref{lemm:fund3}.
In the case of $s=0$, by Lemma \ref{lemm:fund1} and \eqref{140922_1}
\begin{align*}
&\|K_1^{+,2}(t;A_0)f\|_{L_q(\tws)} \\
&\quad\leq C\wt^{-\frac{N-1}{2}\left(\frac{1}{2}-\frac{1}{q}\right)}\int_{0}^{\infty}\left\|\int_{0}^{1}
\frac{e^{-C|B|^{2}\wt}\varphi_{0}(\xi^{\prime})}{(\sqrt{u})^{2-\de}}\,du
\,e^{-\frac{A^{2}}{2}\wt} A^{1-\de} e^{-A(x_{N}+y_{N})}\,\widehat{f}(y_{N})\right\|_{2}\,dy_{N}  \\
&\quad\leq
C\wt^{-\frac{N-1}{2}\left(\frac{1}{r}-\frac{1}{q}\right)-\frac{\de}{2}}
\int_0^\infty\frac{\|f(\cdot,y_N)\|_{L_r(\tws)}}{\widetilde{t}^{(1-\de)/2}+(x_N)^{1-\de}+(y_N)^{1-\de}}\,dy_N,
\end{align*}
which implies that the required inequality holds by Lemma \ref{lemm:fund2} (2) if we choose a sufficiently small $\de>0$
and $(q,r)\neq(2,2)$.
Analogously, for $K_{2}^{+,2}(t;A_0)$, we see that the required inequality holds,
noting that there holds, by \eqref{deriv:M} and Lemma \ref{lem1:Gam2},
\begin{equation*}
|\CM(a)|
\leq a\int_{0}^{1}e^{-\{({\rm Re}B)\te+A(1-\te)\}a}\,d\te
\leq a e^{-b_{0}^{-1}Aa}\leq 2b_{0}A^{-1}e^{-(b_{0}^{-1}/2)Aa}
\end{equation*}
for $a>0$ and any $A\in(0,A_0)$ by choosing suitable $A_0\in(0,A_1)$.
%
%
%
%
%
%
%
%
%
%
%
%

Next we show the inequalities for $K_{n}^{+,2}(t;A_0)$ $(n=3,4,5)$.
Note that by Lemma \ref{lem1:Gam2} we have
\begin{equation}\label{140922_11}
|\CM(a)|=\frac{|e^{-Ba}-e^{-Aa}|}{|B-A|}\leq C|B|^{-1}e^{-CAa}
\leq C\frac{e^{-CAa}}{\sqrt{u}}
\end{equation}
%
%
%
%
with $a>0$ and some positive constant $C$
for any $A\in(0,A_0)$ by choosing suitable $A_0\in(0,A_1)$.
Then, by \eqref{140922_2} and Lemma \ref{lem1:Gam2}, there holds
\begin{align*}
&\|[K_n^{+,2}(t;A_0)f](\cdot,x_N)\|_{L_q(\tws)} \\
&\quad\leq
C\widetilde{t}^{-\frac{N-1}{2}\left(\frac{1}{2}-\frac{1}{q}\right)}\int_0^\infty
\left\|\int_0^1\frac{e^{-\frac{A^2}{2}\widetilde{t}}e^{-C|B|^2\widetilde{t}}A|B|^se^{-CA(x_N+y_N)}\ph_0(\xi^\pr)}
{(A\sqrt{1-u}+\sqrt{u}+A)\sqrt{u}}\,du\,\wh{f}(y_N)\right\|_2 dy_N,
\end{align*}
which furnishes that the required inequalities of $K_{n}^{+,2}(t;A_0)$ $(n=3,4,5)$ hold
in the same manner as we have obtained the inequality of $K_{1}^{+,2}(t;A_0)$ from \eqref{140922_1}.

Finally, we consider $K_{6}^{+,2}(t;A_0)$. 
%
%
%
%
%
%
By \eqref{140922_4} and \eqref{140922_11}, we have for $s>0$
\begin{align*}
&\|[K_6^{+,2}(t;A_0)f](\cdot,x_N)\|_{L_q(\tws)} \\
&\quad\leq
C\widetilde{t}^{-\frac{N-1}{2}\left(\frac{1}{2}-\frac{1}{q}\right)}
\int_0^\infty\left\|\int_0^1e^{-\frac{A^2}{2}\widetilde{t}}e^{-C|B|^2\widetilde{t}}\ph_0(\xi')
A|B|^s\CM(x_N)\CM(y_N)\,du\,\wh{f}(y_N)\right\|_2\,dy_N \\
&\quad\leq
C\widetilde{t}^{-\frac{N-1}{2}\left(\frac{1}{2}-\frac{1}{q}\right)}
\int_0^\infty\left\|\int_0^1e^{-\frac{A^2}{2}\widetilde{t}}e^{-C|B|^2\widetilde{t}}\ph_0(\xi')
A|B|^{s-2}e^{-CA(x_N+y_N)}\,du\,\wh{f}(y_N)\right\|_2\,dy_N \\
&\quad\leq
C\widetilde{t}^{-\frac{N-1}{2}\left(\frac{1}{2}-\frac{1}{q}\right)}
\int_0^\infty\left\|\int_0^1\frac{e^{-C|B|^2\widetilde{t}}|B|^{s-\de}\ph_0(\xi')}{(\sqrt{u})^{2-\de}}\,du
\,e^{-\frac{A^2}{2}\widetilde{t}}Ae^{-CA(x_N+y_N)}\,\wh{f}(y_N)\right\|_2\,dy_N \\
&\quad\leq
C\widetilde{t}^{-\frac{N-1}{2}\left(\frac{1}{2}-\frac{1}{q}\right)-\frac{s}{2}}
\int_0^\infty\left\|e^{-(A^2/2)\widetilde{t}}Ae^{-CA(x_N+y_N)}\,\wh{f}(y_N)\right\|_2dy_N
\end{align*}
by choosing sufficiently small $\de>0$.
We thus obtain the required inequality of $K_6^{+,2}(t;A_0)$ by Lemma \ref{lemm:fund3} if $s>0$.
In the case of $s=0$, since it follows from Lemma \ref{lem1:Gam2} that 
\begin{align*}
|\mathcal{M}(a)|&\leq a\int_{0}^{1}e^{-\{({\rm Re}B)\theta+A(1-\theta)\}a}\,d\theta
\leq a\int_{0}^{1}e^{-\{(b_{0}^{-1}(\sqrt{u}+A))\theta+A(1-\theta)\}x_{N}}\,d\theta \\
&\leq a e^{-b_{0}^{-1}Aa}\int_{0}^{1}e^{-b_{0}^{-1}\sqrt{u}\theta a}\,d\theta\quad(a>0,\ \lambda\in\Gamma_2^+)
\end{align*}
for any $A\in(0,A_0)$ by choosing some $A_0\in(0,A_1)$,
%
%
we easily obtain by Lemma \ref{lemm:fund1}
\begin{align*}
&\|[K_{6}^{+,2}(t;A_0)f](\cdot\,,x_{N})\|_{L_{q}(\tws)} \\
&\quad\leq C\wt^{-\frac{N-1}{2}\left(\frac{1}{2}-\frac{1}{q}\right)}
\int_{0}^{\infty}\left\|\int_0^1e^{-\frac{A^{2}}{2}\wt}e^{-C|B|^{2}\wt}\varphi_{0}(\xi^{\prime})
A\CM(x_{N})\CM(y_{N})\,du\,\widehat{f}(y_{N})\right\|_{2}dy_{N} \\
&\quad\leq C\wt^{-\frac{N-1}{2}\left(\frac{1}{2}-\frac{1}{q}\right)}\int_{0}^{\infty}x_{N}y_{N}
\left\|\,e^{-\frac{A^{2}}{2}\wt} A e^{-CA(x_{N}+y_{N})}\,\widehat{f}(y_{N})\right\|_{2} \\
&\qquad\ \times\iiint_{[0,1]^{3}}e^{-C u\wt}
e^{-C\sqrt{u}\varphi x_{N}}e^{-C\sqrt{u}\psi y_{N}}\,dud\varphi d\psi dy_{N} \\
&\quad\leq C\wt^{-\frac{N-1}{2}\left(\frac{1}{r}-\frac{1}{q}\right)}\int_{0}^{\infty}
\frac{x_{N}y_{N}\|f(\cdot\,,y_{N})\|_{L_{r}(\tws)}}{\wt^{\,1/2}+x_{N}+y_{N}}
\iint_{[0,1]^{2}}\frac{d\varphi d\psi}{\wt+(\varphi x_{N})^{2}+(\psi y_{N})^{2}}\,dy_{N}
\end{align*}
for some positive constant $C$. The change of variable: 
$\psi y_{N}=\{\wt+(\varphi x_{N})^{2}\}^{1/2}\ell$ yields that
\begin{multline*}
\int_{0}^{1}\frac{d\psi}{\wt+(\varphi x_{N})^{2}+(\psi y_{N})^{2}} \\
\leq\frac{1}{\wt+(\varphi x_{N})^{2}}\int_{0}^{\infty}
\frac{1}{1+\ell^{2}}\frac{\{\wt+(\varphi x_{N})^{2}\}^{1/2}}{y_{N}}\,d\ell
\leq \frac{C}{y_{N}(\wt^{\,1/2}+\varphi x_{N})}
\end{multline*}
for a positive constant $C$, so that
\begin{align*}
&\|[K_{6}^{+,2}(t;A_0)f](\cdot\,,x_{N})\|_{L_{q}(\mathbf{R}^{N-1})} \\
&\quad\leq C\wt^{-\frac{N-1}{2}\left(\frac{1}{r}-\frac{1}{q}\right)}\int_{0}^{\infty}
\frac{x_{N}\|f(\cdot\,,y_{N})\|_{L_{r}(\mathbf{R}^{N-1})}}{\wt^{1/2}+x_{N}+y_{N}}
\int_{0}^{1}\frac{d\varphi}{\wt^{1/2}+\varphi x_{N}}\,dy_{N} \\
&\quad= C\wt^{-\frac{N-1}{2}\left(\frac{1}{r}-\frac{1}{q}\right)}
\int_{0}^{\infty}\frac{x_{N}\|f(\cdot\,,y_{N})\|_{L_{r}(\mathbf{R}^{N-1})}}
{(\wt^{\,1/2}+x_{N}+y_{N})^{1-\delta}}\int_{0}^{1}\frac{d\varphi dy_{N}}
{(\wt^{\,1/2}+x_{N}+y_{N})^{\delta}(\wt^{\,\frac{1}{2}}+\varphi x_{N})}
\end{align*}
for any $0<\delta<1$.
By the change of variable: $\varphi x_{N}=\wt^{\,1/2}\ell$, we then have
\begin{align*}
\int_{0}^{1}\frac{d\varphi}{(\wt^{\frac{1}{2}}+x_{N}+y_{N})^{\delta}(\wt^{\,\frac{1}{2}}+\varphi x_{N})}
&\leq
\int_{0}^{1}\frac{d\varphi}{(\wt^{\,1/2}+\varphi x_{N})^{\delta}(\wt^{\,1/2}+\varphi x_{N})} \\
&\leq C\int_{0}^{1}\frac{d\varphi}{\wt^{\,(1+\delta)/2}+(\varphi x_{N})^{1+\delta}} \\
&\leq \frac{C}{\wt^{\,(1+\delta)/2}}\int_{0}^{\infty}\frac{1}{1+\ell^{1+\delta}}\frac{\wt^{1/2}}{x_{N}}\,d\ell
\leq \frac{C}{x_{N}\wt^{\,\delta/2}}
\end{align*}
with a positive constant $C$, which furnishes that
\begin{equation*}
\|[K_{6}^{+,2}(t;A_0)f](\cdot\,,x_{N})\|_{L_{q}(\tws)}\leq C\wt^{-\frac{N-1}{2}\left(\frac{1}{r}-\frac{1}{q}\right)
-\frac{\delta}{2}}\int_{0}^{\infty}\frac{\|f(\cdot\,,y_{N})\|_{L_{r}(\tws)}}
{\wt^{\,(1-\delta)/2}+x_{N}^{1-\delta}+y_{N}^{1-\delta}}\,dy_{N}.
\end{equation*}
We therefore obtain the required inequality by Lemma \ref{lemm:fund2} (2)
by choosing a sufficiently small $\de>0$ when $(q,r)\neq (2,2)$. \\
%
%
%
%
(2)	First, we show the inequality for $L_{1}^{+,2}(t;A_0)$.
Noting that $\lambda=-(A^{2}(1-u)+\gamma_{0}u) + i((A^{2}/4)(1-u)+\gamma_{0}u)$
for $u\in[0,1]$ on $\Gamma_{2}^{+}$,
we have, by (\ref{op:low}),
\begin{align*}
&[L_{1}^{+,2}(t;A_0)d](x) \\
&=\CF_{\xi^{\prime}}^{-1}\left[\int_{0}^{1}e^{\{-(A^{2}(1-u)+\gamma_{0}u)
+ i((A^{2}/4)(1-u)+\widetilde{\gamma}_{0}u)\}t}\varphi_{0}(\xi^{\prime})\ell_{1}(\xi^\prime,\lambda)e^{-A(x_{N}+y_{N})}\right. \\
&\quad\left.\times\left\{-(\gamma_{0}-A^{2})+ i\left(\widetilde{\gamma}_{0}-\frac{A^{2}}{4}\right)\right\}
\widehat{d}(\xi^{\prime})\right](x^{\prime}).
\end{align*}
In a similar way to the case of $K_{1}^{+,2}(t;A_0)$, we have by \eqref{140922_4} and Lemma \ref{lemm:fund1}
\begin{align}
&\|[L_{1}^{+,2}(t;A_0)d](\cdot\,,x_{N})\|_{L_{q}(\tws)} \notag \\
&\quad\leq
C\wt^{-\frac{N-1}{2}\left(\frac{1}{2}-\frac{1}{q}\right)}\left\|\int_{0}^{1}
\frac{e^{-(A^{2}/2)\wt} e^{-C|B|^{2}\wt}\varphi_{0}(\xi^{\prime})A^{1/2}|B|^s e^{-Ax_{N}}}
{(A\sqrt{1-u}+\sqrt{u}+A^{1/4})^{2}}\,du\,\widehat{d}(\xi^{\prime})\right\|_{2} \label{140922_13} \\
&\quad\leq C\wt^{-\frac{N-1}{2}\left(\frac{1}{2}-\frac{1}{q}\right)}\left\|\int_{0}^{1}
\frac{e^{-C |B|^2\wt}|B|^{s-\de}\varphi_{0}(\xi^{\prime})}{(\sqrt{u})^{2-\delta}}\,du\,e^{-(A^2/2)\wt}A^{1/2}e^{-Ax_{N}}
\widehat{d}(\xi^{\prime})\right\|_{2} \label{140922_5} \\
&\quad\leq
C\wt^{-\frac{N-1}{2}\left(\frac{1}{r}-\frac{1}{q}\right)-\frac{s}{2}}\|d\|_{L_r(\tws)}/(\widetilde{t}^{\,1/4}+(x_N)^{1/2}).
\notag
\end{align}
%
%
%
%
We thus obtain the required inequality by Lemma \ref{lemm:fund2} (1) if $s>0$ and $q>2$.
In the case of $s>0$ and $q=2$, by \eqref{140922_5} and using \eqref{140922_4} again, we have
\begin{align*}
\|L_1^{+,2}(t;A_0)d\|_{L_2(\uhs)}
&\leq
C
\left\|\int_{0}^{1}
\frac{e^{-C |B|^2\wt}|B|^{s-\de}}{(\sqrt{u})^{2-\delta}}\,du\,e^{-(A^2/2)\wt}
\widehat{d}(\xi^{\prime})\right\|_{2} \\
&\leq C\widetilde{t}^{-\frac{N-1}{2}\left(\frac{1}{r}-\frac{1}{2}\right)-\frac{s}{2}}\|d\|_{L_r(\tws)}.
\end{align*}
If $s=0$, then we have by Lemma \ref{lemm:fund1} and Lemma \ref{lem1:Gam2}
\begin{align}\label{140922_8}
&\|[L_1^{+,2}(t;A_0)d](\cdot,x_N)\|_{L_q(\tws)} \notag \\
&\quad\leq
C\widetilde{t}^{-\frac{N-1}{2}\left(\frac{1}{2}-\frac{1}{q}\right)}
\left\|
\int_0^1\frac{e^{-(A^{2}/2)\wt} e^{-C|B|^{2}\wt}\varphi_{0}(\xi^{\prime})A^{1/2} e^{-Ax_{N}}}
{(A\sqrt{1-u}+\sqrt{u}+A^{1/4})^{2}}\,du\,\widehat{d}(\xi^{\prime})
\right\|_2\notag \displaybreak[0] \\ 
&\quad\leq
C\widetilde{t}^{-\frac{N-1}{2}\left(\frac{1}{2}-\frac{1}{q}\right)}
\left\|
\int_0^1\frac{ e^{-C u\wt} }
{(\sqrt{u})^{2-\de}}\,du\,e^{-(A^{2}/2)\wt}A^{(1/2)-(\de/4)}e^{-Ax_{N}} \,\widehat{d}(\xi^{\prime})
\right\|_2 \displaybreak[0] \\
&\quad\leq
C\widetilde{t}^{-\frac{N-1}{2}\left(\frac{1}{r}-\frac{1}{q}\right)-\frac{\de}{2}}
\|d\|_{L_r(\tws)}/(\widetilde{t}^{(1/4)-(\de/8)}+(x_N)^{(1/2)-(\de/4)}),
\notag
\end{align}
which, combined with Lemma \ref{lemm:fund2} (1), furnishes that the required inequality holds for $q>2$
by choosing a sufficiently small $\de>0$.
In the case of $s=0$ and $q=2$, by \eqref{140922_8} and Young's inequality with $1+(1/2)=(1/p)+(1/r)$ for $1\leq r<2$,
we have
\begin{equation}\label{140922_10}
\|L_1^{+,2}(t;A_0)d\|_{L_2(\uhs)}
\leq C\widetilde{t}^{-\frac{\de}{2}}
\|\CF_{\xi^\pr}^{-1}[e^{-(A^{2}/2)\widetilde{t}}A^{-\de/4}]\|_{L_p(\tws)}\|d\|_{L_r(\tws)}.
\end{equation}
We use the following proposition proved by \cite[Theorem 2.3]{SS01} 
to calculate the right-hand side of \eqref{140922_10}.
\begin{prop}\label{prop:SS01}
Let $X$ be a Banach space and $\|\cdot\|_X$ its norm.
Suppose that $L$ and $n$ be a non-negative integer and positive integer, respectively.
Let $0<\si\leq 1$ and $s=L+\si-n$.
Let $f(\xi)$ be a $C^\infty$-function, defined on $\BR^n\setminus\{0\}$ with value X, which satisfies the following two conditions:
\begin{enumerate}[$(1)$]
\item
$D_\xi^\al f \in L_1(\BR^n,X)$ for any multi-index $\al\in\BN_0^n$ with $|\al|\leq L$.
\item
For any multi-index $\al\in\BN_0^n$, there exists a positive constant $C(\al)$ such that
\begin{equation*}
\|D_\xi^\al f(\xi)\|_X\leq C(\al)|\xi|^{s-|\al|}\quad(\xi\in\BR^n\setminus\{0\}).
\end{equation*}
\end{enumerate}
Then there exists a positive constant $C(n,s)$ such that
\begin{equation*}
\|\CF_{\xi}^{-1}[f](x)\|_X\leq C(n,s)\left(\max_{|\al|\leq L+2}C(\al)\right)|x|^{-(n+s)}
\quad(x\in\BR^n\setminus\{0\}).
\end{equation*}
\end{prop}
By Proposition \ref{prop:SS01} with $n=N-1$, $L=N-2$, and $\si=1-\de/4$,
we have
\begin{equation*}
|\CF_{\xi^\pr}^{-1}[e^{-(A^{2}/2)\widetilde{t}}A^{-\de/4}](x^\pr)|
\leq C|x'|^{-(N-1-\de/4)}
\end{equation*}
for a positive constant $C$, and furthermore, by direct calculations
\begin{equation*}
|\CF_{\xi^\pr}^{-1}[e^{-(A^{2}/2)\widetilde{t}}A^{-\de/4}](x^\pr)|\leq C
\widetilde{t}^{\,-(1/2)(N-1-\de/4)}.
\end{equation*}
We thus obtain
\begin{equation*}
|\CF_{\xi^\pr}^{-1}[e^{-(A^{2}/2)\widetilde{t}}A^{-\de/4}](x^\pr)|
\leq \frac{C}{\widetilde{t}^{\,(1/2)(N-1-\de/4)}+|x'|^{(N-1-\de/4)}}
\end{equation*}
for some positive constant $C$.
Therefore, by choosing a sufficiently small $\de>0$, we see that
\begin{equation*}
\|\CF_{\xi^\pr}^{-1}[e^{-(A^{2}/2)\widetilde{t}}A^{-\de/4}]\|_{L_p(\tws)}
\leq C\widetilde{t}^{-\frac{N-1}{2}\left(1-\frac{1}{p}\right)+\frac{\de}{8}}
=C\widetilde{t}^{-\frac{N-1}{2}\left(\frac{1}{r}-\frac{1}{2}\right)+\frac{\de}{8}}
\end{equation*}
since $p>1$ by $1\leq r<2$, which, combined with \eqref{140922_10}, furnishes that
the required inequality holds.
Summing up in the case of $s=0$, we have obtained
\begin{equation*}
\|L_1^{+,2}(t;A_0)d\|_{L_q(\uhs)}
\leq C
\widetilde{t}^{-\frac{N-1}{2}\left(\frac{1}{r}-\frac{1}{q}\right)-\frac{1}{2}\left(\frac{1}{2}-\frac{1}{q}\right)}
\|d\|_{L_r(\tws)}
\end{equation*}
for some positive constant $C$ and $1\leq r\leq 2\leq q\leq\infty$ when $(q,r)\neq(2,2)$.

Concerning $L_{2}^{+,2}(t;A_0)$, we see, by Lemma \ref{lemm:fund1}, that 
\begin{align*}
&\|[L_{2}^{+,2}(t;A_0)d](\cdot\,,x_{N})\|_{L_{q}(\mathbf{R}^{N-1})} \\
&\quad\leq C\wt^{-\frac{N-1}{2}\left(\frac{1}{2}-\frac{1}{q}\right)-\frac{s}{2}}
\left\|\int_{0}^{1}e^{-(A^2/2)\wt}e^{-C|B|^{2}\wt}\varphi_{0}(\xi^{\prime})
e^{-({\rm Re}B)x_{N}}\,du\,\widehat{d}(\xi^{\prime})\right\|_{2} \\
&\quad\leq C\wt^{-\frac{N-1}{2}\left(\frac{1}{2}-\frac{1}{q}\right)-\frac{s}{2}}\left\|\int_{0}^{1}
e^{-Cu\wt}e^{-C\sqrt{u}x_{N}}\,du\,e^{-(A^2/2)\wt}\widehat{d}(\xi^{\prime})\right\|_{2} \\
&\quad\leq C\wt^{-\frac{N-1}{2}\left(\frac{1}{2}-\frac{1}{q}\right)-\frac{s}{2}}\frac{\|\,e^{-(A^2/2)\wt}
\widehat{d}(\xi^{\prime})\|_{2}}{\wt+(x_{N})^{2}}
\leq C\wt^{-\frac{N-1}{2}\left(\frac{1}{r}-\frac{1}{q}\right)-\frac{s}{2}}
\frac{\|d\|_{L_{r}(\mathbf{R}^{N-1})}}{\wt+(x_{N})^{2}},
\end{align*}
which, combined with Lemma \ref{lemm:fund2} (1), furnishes the required inequality for $L_{2}^{+,2}(t;A_0)$.

Finally, we show the inequality for $L_{3}^{+,2}(t;A_0)$.
We easily have by \eqref{140922_11} and Lemma \ref{lem1:Gam2}
\begin{multline*}
\|[L_{3}^{+,2}(t)d](\cdot\,,x_{N})\|_{L_{q}(\mathbf{R}^{N-1})} \\
\leq C\wt^{-\frac{N-1}{2}\left(\frac{1}{2}-\frac{1}{q}\right)}\left\|\int_{0}^{1}
\frac{e^{-(A^2/2)\wt}e^{-C|B|^2\wt}\ph_0(\xi')A^{1/2}|B|^se^{-CAx_{N}}}
{(A\sqrt{1-u}+\sqrt{u}+A^{1/4})\sqrt{u}}\,du\,\widehat{d}(\xi^{\prime})\right\|_{2}
\end{multline*}
for a positive constant $C$.
We thus obtain the required inequality in the same manner as we have obtained the inequality of $L_1^{+,2}(t;A_0)$
from \eqref{140922_13}.
%
%
%
%
%
%
%
%
%
%
\end{proof}
\begin{coro}\label{cor:Gam2}
Let $1\leq r\leq 2\leq q\leq\infty$, and let $f\in L_{r}(\mathbf{R}_{+}^{N})^{N}$ and $d\in L_{r}(\mathbf{R}^{N-1})$.
We use the symbols defined in $(\ref{op:low})$.
\begin{enumerate}[$(1)$]
\item
Let $\alpha\in\mathbf{N}_{0}^{N}$ and we assume that there exist positive constants $A_1\in(0,1)$ and $C$
such that for any $\lambda\in\Gamma_{2}^{\pm}$ and $A\in(0,A_1)$
\begin{align*}
|k_{1}(\xi^\prime,\lambda)|&\leq C(A\sqrt{1-u}+\sqrt{u}+A^{1/4})^{-4}A, \\
|k_{2}(\xi^\prime,\lambda)|&\leq C(A\sqrt{1-u}+\sqrt{u}+A^{1/4})^{-4}A^2, \\
|\ell_{1}(\xi^\prime,\lambda)|&\leq C(A\sqrt{1-u}+\sqrt{u}+A^{1/4})^{-4}|B|^2.
\end{align*}
Then there exist positive constants $A_0\in(0,A_1)$ and $C=C(\alpha)$ such that for any $t>0$ and $n=1,2$
\begin{align*}
\|D_{x}^{\alpha}\nabla K_{n}^{\pm,2}(t;A_0)f\|_{L_{q}(\uhs)}&\leq C(t+1)^{-\frac{N}{2}\left(\frac{1}{r}-\frac{1}{q}\right)
-\frac{1}{4}-\frac{|\alpha|}{2}}\|f\|_{L_{r}(\uhs)}, \\
\|D_{x}^{\alpha}\pa_t K_{n}^{\pm,2}(t;A_0)f\|_{L_{q}(\uhs)}&\leq C(t+1)^{-\frac{N}{2}\left(\frac{1}{r}-\frac{1}{q}\right)
-\frac{|\alpha|}{2}}\|f\|_{L_{r}(\uhs)} \quad\text{if $|\al|\neq0$},\\
\|D_{x}^{\alpha}\nabla L_{1}^{\pm,2}(t;A_0)d\|_{L_{q}(\uhs)}&\leq C(t+1)^{-\frac{N-1}{2}\left(\frac{1}{r}-\frac{1}{q}\right)
-\frac{1}{2}\left(\frac{1}{2}-\frac{1}{q}\right)-1-\frac{|\alpha|}{2}}\|d\|_{L_{r}(\tws)}.
\end{align*}
In addition, if $(q,r)\neq(2,2)$, then we have for any $t>0$ and $n=1,2$
\begin{equation*}
\|\pa_t K_{n}^{\pm,2}(t;A_0)f\|_{L_{q}(\uhs)}
\leq C(t+1)^{-\frac{N}{2}\left(\frac{1}{r}-\frac{1}{q}\right)}\|f\|_{L_{r}(\uhs)}.
\end{equation*}

\item
Let $k=0,1$, $\ell=0,1,2$, and  $\alpha^{\prime}\in\mathbf{N}_{0}^{N-1}$.
We assume that there exist positive constants $A_1\in(0,1)$ and $C$
such that for any $\lambda\in\Gamma_{2}^{\pm}$ and $A\in(0,A_1)$
\begin{align*}
|k_{3}(\xi^\prime,\lambda)|&\leq C(A\sqrt{1-u}+\sqrt{u}+A)^{-2}A, \\
|k_{n}(\xi^\prime,\lambda)|&\leq C(A\sqrt{1-u}+\sqrt{u}+A)^{-2}A|B| \quad(n=4,5), \\
|k_{6}(\xi^\prime,\lambda)|&\leq C(A\sqrt{1-u}+\sqrt{u}+A)^{-2}A|B|^{2} \\
|\ell_{2}(\xi^\prime,\lambda)|&\leq C(A\sqrt{1-u}+\sqrt{u}+A^{1/4})^{-4}A, \\
|\ell_{3}(A,B)|&\leq C(A\sqrt{1-u}+\sqrt{u}+A^{1/4})^{-3}A.
\end{align*}
Then there exist positive constants $A_0\in(0,A_1)$ and  $C=C(\alpha^{\prime})$ such that for any $t>0$
\begin{align*}
&\|\pa_{t}^{k}D_{x^{\prime}}^{\alpha^{\prime}}D_{N}^{\ell}K_{n}^{\pm,2}(t;A_0)f\|_{L_{q}(\mathbf{R}_{+}^{N})} \\
&\quad\leq C(t+1)^{-\frac{N}{2}\left(\frac{1}{r}-\frac{1}{q}\right)-k-\frac{|\alpha^{\prime}|+\ell}{2}}
\|f\|_{L_{r}(\mathbf{R}_{+}^{N})}\quad(n=3,4,5,6), \\
&\|\pa_{t}^{k}D_{x^{\prime}}^{\alpha^{\prime}}D_{N}^{\ell}L_{n}^{\pm,2}(t;A_0)d\|_{L_{q}(\mathbf{R}_{+}^{N})} \\
&\quad\leq C(t+1)^{-\frac{N-1}{2}\left(\frac{1}{r}-\frac{1}{q}\right)-\frac{1}{2}\left(\frac{1}{2}-\frac{1}{q}\right)
-k-\frac{|\alpha^{\prime}|+\ell}{2}}\|d\|_{L_{r}(\mathbf{R}^{N-1})}\quad(n=2,3),
\end{align*}
provided that $k+\ell+|\al^\pr|\neq0$.
In addition, if $(q,r)\neq(2,2)$, then there hold for any $t>0$
\begin{align*}
\|K_{n}^{\pm,2}(t;A_0)f\|_{L_{q}(\mathbf{R}_{+}^{N})}
&\leq C(t+1)^{-\frac{N}{2}\left(\frac{1}{r}-\frac{1}{q}\right)}
\|f\|_{L_{r}(\mathbf{R}_{+}^{N})}\quad(n=3,4,5,6), \\
\|L_{n}^{\pm,2}(t;A_0)d\|_{L_{q}(\mathbf{R}_{+}^{N})}
&\leq C(t+1)^{-\frac{N-1}{2}\left(\frac{1}{r}-\frac{1}{q}\right)-\frac{1}{2}\left(\frac{1}{2}-\frac{1}{q}\right)
}\|d\|_{L_{r}(\mathbf{R}^{N-1})}\quad(n=2,3).
\end{align*}
\end{enumerate}
\end{coro}
\begin{proof}
We only show the inequalities for $K_{5}^{\pm,2}(t)$, $K_{6}^{\pm,2}(t)$, and $L_{3}^{\pm,2}(t)$.
The other inequalities can be proved by Lemma \ref{lem2:Gam2} directly.
By (\ref{op:low})
\begin{align*}
&\pa_{t}^{k}D_{x^{\prime}}^{\alpha^{\prime}}[K_{n}^{\pm,2}(t;A_0)f](x) \\
&\quad=\int_{0}^{\infty}
\CF_{\xi^{\prime}}^{-1}\left[\int_{\Gamma_{2}^{\pm}}e^{\lambda t}\varphi_{0}(\xi^{\prime})\lambda^{k}(i\xi^{\prime})^{\alpha^{\prime}}
k_{n}(\xi^\prime,\lambda)\CX_{n}(x_{N},y_{N})\,d\lambda\,\widehat{f}(\xi^{\prime},y_{N})\right](x^{\prime}), \\
&\pa_{t}^{k}D_{x^{\prime}}^{\alpha^{\prime}}[L_{3}^{\pm,2}(t;A_0)d](x)=\CF_{\xi^{\prime}}^{-1}
\left[\int_{\Gamma_{2}^{\pm}}e^{\lambda t}\varphi_{0}(\xi^{\prime})\lambda^{k}(i\xi^{\prime})^{\alpha^{\prime}}
\ell_{3}(\xi^\prime,\lambda)\CM(x_N)\,d\lambda\,\widehat{d}(\xi^{\prime})\right](x^{\prime})
\end{align*}
for $n=5,6$. Since by Lemma \ref{lem1:Gam2}
\begin{align*}
|\lambda^{k}(i\xi^{\prime})^{\alpha^{\prime}}k_{n}(\xi^\prime,\lambda)|&\leq C
\left\{\begin{aligned}
&(A\sqrt{1-u}+\sqrt{u}+A)^{-1}A|B|^{2k+|\alpha^{\prime}|}&&(n=5), \\
&A|B|^{2k+|\alpha^{\prime}|}&&(n=6),
\end{aligned}\right. \displaybreak[0] \\
|\lambda^{k}(i\xi^{\prime})^{\alpha^{\prime}}\ell_{3}(\xi',\lambda)|
&\leq (A\sqrt{1-u}+\sqrt{u}+A^{1/4})^{-3}A|B|^{2k+|\alpha^{\prime}|}
\end{align*}
for $\lambda\in\Gamma_2^\pm$ and $A\in (0,A_0)$ by choosing some $A_0\in(0,A_1)$,
%
%
%
%
we obtain by Lemma \ref{lem2:Gam2}
\begin{align}\label{0821_2}
\|\partial_{t}^{k}D_{x^{\prime}}^{\alpha^{\prime}}K_{n}^{\pm,2}(t)f\|_{L_{q}(\mathbf{R}_{+}^{N})}
&\leq C(t+1)^{-\frac{N}{2}\left(\frac{1}{r}-\frac{1}{q}\right)-k-\frac{|\alpha^{\prime}|}{2}}
\|f\|_{L_{r}(\mathbf{R}_{+}^{N})}\quad(n=5,6), \notag \\
\|\partial_{t}^{k}D_{x^{\prime}}^{\alpha^{\prime}}L_{3}^{\pm,2}(t)d\|_{L_{q}(\mathbf{R}_{+}^{N})}&\leq
C(t+1)^{-\frac{N-1}{2}\left(\frac{1}{r}-\frac{1}{q}\right)
-\frac{1}{2}\left(\frac{1}{2}-\frac{1}{q}\right)-k-\frac{|\alpha^{\prime}|}{2}}
\|d\|_{L_{r}(\mathbf{R}^{N-1})}
\end{align}
for any $t>0$, provided that $k+|\al^\pr|\neq0$.
In the case of $k+|\al^\pr|=0$, we have by Lemma \ref{lem2:Gam2}
\begin{align}\label{140923_5}
\|K_{n}^{\pm,2}(t)f\|_{L_{q}(\mathbf{R}_{+}^{N})}
&\leq C(t+1)^{-\frac{N}{2}\left(\frac{1}{r}-\frac{1}{q}\right)}
\|f\|_{L_{r}(\mathbf{R}_{+}^{N})}\quad(n=5,6), \notag \\
\|L_{3}^{\pm,2}(t)d\|_{L_{q}(\mathbf{R}_{+}^{N})}
&\leq
C(t+1)^{-\frac{N-1}{2}\left(\frac{1}{r}-\frac{1}{q}\right)
-\frac{1}{2}\left(\frac{1}{2}-\frac{1}{q}\right)}
\|d\|_{L_{r}(\mathbf{R}^{N-1})}
\end{align}
when $(q,r)\neq(2,2)$. On the other hand, by (\ref{deriv:M})
\begin{align*}
&\pa_{t}^{k}D_{x^{\prime}}^{\alpha^{\prime}}D_{N}^{\ell}[K_{5}^{\pm,2}(t)f](x)=(-1)^{\ell}\Bigg\{ \\
&\int_{0}^{\infty}\CF_{\xi^{\prime}}^{-1}\left[\int_{\Gamma_{2}^{\pm}}e^{\lambda t}
\varphi_{0}(\xi^{\prime})\lambda^{k}(i\xi^{\prime})^{\alpha^{\prime}}(B+A)^{\ell-1}k_{5}(\xi^\prime,\lambda)e^{-B(x_{N}+y_{N})}
d\lambda\,\widehat{f}(\xi^{\prime},y_{N})\right](x^{\prime})\,dy_{N} \\
&+\int_{0}^{\infty}\mathcal{F}_{\xi^{\prime}}^{-1}\left[\int_{\Gamma_{2}^{\pm}}e^{\lambda t}
\varphi_{0}(\xi^{\prime})\lambda^{k}(i\xi^{\prime})^{\alpha^{\prime}}A^{\ell}k_{5}(\xi^\prime,\lambda)\mathcal{M}(x_{N})e^{-By_{N}}
d\lambda\,\widehat{f}(\xi^{\prime},y_{N})\right](x^{\prime})\,dy_{N}\Bigg\}, \displaybreak[0] \\
&\pa_{t}^{k}D_{x^{\prime}}^{\alpha^{\prime}}D_{N}^{\ell}[K_{6}^{\pm,2}(t)f](x)=(-1)^{\ell}\Bigg\{ \\
&\int_{0}^{\infty}\mathcal{F}_{\xi^{\prime}}^{-1}\left[\int_{\Gamma_{2}^{\pm}}e^{\lambda t}
\varphi_{0}(\xi^{\prime})\lambda^{k}(i\xi^{\prime})^{\alpha^{\prime}}(B+A)^{\ell-1}k_{6}(\xi^\prime,\lambda)e^{-Bx_{N}}
\mathcal{M}(y_{N})d\lambda\,\widehat{f}(\xi^{\prime},y_{N})\right](x^{\prime})\,dy_{N} \\
&+\int_{0}^{\infty}\mathcal{F}_{\xi^{\prime}}^{-1}\left[\int_{\Gamma_{2}^{\pm}}e^{\lambda t}
\varphi_{0}(\xi^{\prime})\lambda^{k}(i\xi^{\prime})^{\alpha^{\prime}}A^{\ell}k_{6}(\xi^\prime,\lambda)\CM(x_{N})\CM(y_{N})
d\lambda\,\widehat{f}(\xi^{\prime},y_{N})\right](x^{\prime})\,dy_{N}\Bigg\}, \displaybreak[0] \\
&\pa_{t}^{k}D_{x^{\prime}}^{\alpha^{\prime}}D_{N}^{\ell}[L_{3}^{\pm,2}(t)d](x) \\
&=(-1)^{\ell}\Bigg\{\mathcal{F}_{\xi^{\prime}}^{-1}\left[\int_{\Gamma_{2}^{\pm}}e^{\lambda t}
\varphi_{0}(\xi^{\prime})\lambda^{k}(i\xi^{\prime})^{\alpha^{\prime}}
(B+A)^{\ell-1}\ell_{3}(\xi^\prime,\lambda)e^{-Bx_{N}}d\lambda\,\widehat{f}(\xi^{\prime},y_{N})\right](x^{\prime})\\
&+\mathcal{F}_{\xi^{\prime}}^{-1}\left[\int_{\Gamma_{2}^{\pm}}e^{\lambda t}
\varphi_{0}(\xi^{\prime})\lambda^{k}(i\xi^{\prime})^{\alpha^{\prime}}
A^{\ell}\ell_3(\xi^\prime,\lambda)\mathcal{M}(x_{N})d\lambda\,\widehat{f}(\xi^{\prime},y_{N})\right](x^{\prime})\Bigg\}
\end{align*}
for $\ell=1,2$. Since by Lemma \ref{lem1:Gam2}
\begin{align*}
|\lambda^{k}(i\xi^{\prime})^{\alpha^{\prime}}(B+A)^{\ell-1}k_{5}(\xi^\prime,\lambda)|&\leq
C(A\sqrt{1-u}+\sqrt{u}+A)^{-2}A|B|^{2k+|\alpha^{\prime}|+\ell} \\
|\lambda^{k}(i\xi^{\prime})^{\alpha^{\prime}}A^{\ell}k_{5}(\xi^\prime,\lambda)|
&\leq C(A\sqrt{1-u}+\sqrt{u}+A)^{-1}A|B|^{2k+|\alpha^{\prime}|+\ell}, \\
|\lambda^{k}(i\xi^{\prime})^{\alpha^{\prime}}(B+A)^{\ell-1}k_{6}(\xi^\prime,\lambda)|&\leq 
C(A\sqrt{1-u}+\sqrt{u}+A)^{-1}A|B|^{2k+|\alpha^{\prime}|+\ell}, \\
|\lambda^{k}(i\xi^{\prime})^{\alpha^{\prime}}A^{\ell}k_{6}(\xi^\prime,\lambda)| &\leq CA|B|^{2k+|\alpha^{\prime}|+\ell}
\end{align*}
for any $\lambda\in\Gamma_{2}^{\pm}$ and $A\in(0,A_0)$ by choosing suitable $A_0\in(0,A_1)$,
we have by Lemma \ref{lem2:Gam2}
\begin{equation}\label{0823_1}
\|\partial_{t}^{k}D_{x^{\prime}}^{\alpha^{\prime}}D_{N}^{\ell}K_{n}^{\pm,2}(t)f\|_{L_{q}(\mathbf{R}_{+}^{N})}\leq
C(t+1)^{-\frac{N}{2}\left(\frac{1}{r}-\frac{1}{q}\right)-k-\frac{|\alpha^{\prime}|+\ell}{2}}\|f\|_{L_{r}(\mathbf{R}_{+}^{N})}\quad(n=5,6)
\end{equation}
for $\ell=1,2$. In addition,
\begin{align*}
&|\lambda^{k}(i\xi^{\prime})^{\alpha^{\prime}}(B+A)^{\ell-1}\ell_{3}(\xi^\prime,\lambda)|\leq
C(A\sqrt{1-u}+\sqrt{u}+A^{1/4})^{-4}A|B|^{2k+|\alpha^{\prime}|+\ell} \\
&|\lambda^{k}(i\xi^{\prime})^{\alpha^{\prime}}A^{\ell}\ell_{3}(\xi^\prime,\lambda)| \leq
C(A\sqrt{1-u}+\sqrt{u}+A^{1/4})^{-3}A|B|^{2k+|\alpha^{\prime}|+\ell}
\end{align*}
for any $\lambda\in\Gamma_{2}^{\pm}$ and $A\in (0,A_0)$, and therefore by Lemma \ref{lem2:Gam2}
\begin{equation*}
\|\partial_{t}^{k}D_{x^{\prime}}^{\alpha^{\prime}}D_{N}^{\ell}L_{3}^{\pm,2}(t)d\|_{L_{q}(\mathbf{R}_{+}^{N})}\leq
C(t+1)^{-\frac{N-1}{2}\left(\frac{1}{r}-\frac{1}{q}\right)-\frac{1}{2}\left(\frac{1}{2}-\frac{1}{q}\right)-k-\frac{|\alpha^{\prime}|+\ell}{2}}
\|d\|_{L_{r}(\mathbf{R}^{N-1})}
\end{equation*}
for $\ell=1,2$, which, combined with (\ref{0821_2}), \eqref{140923_5}, and (\ref{0823_1}), furnishes the required estimates
for $K_{5}^{\pm,2}(t)$, $K_{6}^{\pm,2}(t)$, and $L_{3}^{\pm,2}(t)$.
This completes the proof of Corollary \ref{cor:Gam2}.
\end{proof}
By Lemma \ref{lem1:Gam2} 
there exist a positive number $A_1\in(0,1)$
and a positive constant $C$ such that for $j,k=1,\dots,N$,
$\lambda\in\Gamma_2^\pm$, and $A\in(0,A_1)$ we have
\begin{align*}
&\left|\frac{\CV_{jk}^{BB}(\xi^\prime,\lambda)}{L(A,B)}\right|
\leq \frac{CA}{(A\sqrt{1-u}+\sqrt{u}+A)^{2}},
&\left|\frac{\CV_{jk}^{B\CM}(\xi^\prime,\lambda)}{L(A,B)}\right|
&\leq \frac{CA|B|}{(A\sqrt{1-u}+\sqrt{u}+A)^{2}}, \\
&\left|\frac{\CV_{jk}^{\CM B}(\xi^\prime,\lambda)}{L(A,B)}\right|
\leq \frac{CA|B|}{(A\sqrt{1-u}+\sqrt{u}+A)^{2}},
&\left|\frac{\CV_{jk}^{\CM\CM}(\xi^\prime,\lambda)}{L(A,B)}\right|
&\leq \frac{CA|B|^{2}}{(A\sqrt{1-u}+\sqrt{u}+A)^{2}}, \\
&\left|\frac{\CP_{j}^{AA}(\xi^\prime,\lambda)}{L(A,B)}\right|
\leq \frac{CA}{(A\sqrt{1-u}+\sqrt{u}+A^{1/4})^4},
&\left|\frac{\CP_j^{A\CM}(\xi^\prime,\lambda)}{L(A,B)}\right|
&\leq \frac{CA^2}{(A\sqrt{1-u}+\sqrt{u}+A^{1/4})^4},
\end{align*}
and furthermore,
\begin{align*}
|A/L(A,B)|&\leq C(A\sqrt{1-u}+\sqrt{u}+A^{1/4})^{-4}A, \\
|\{A(B^2+A^2)\}/\{(B+A)L_(A,B)\}|&\leq C(A\sqrt{1-u}+\sqrt{u}+A^{1/4})^{-3}A, \\
|D(A,B)/\{(B+A)L(A,B)\}|&\leq C(A\sqrt{1-u}+\sqrt{u}+A^{1/4})^{-4}|B|^2.
\end{align*}
Therefore, remembering  (\ref{Fsol_eq:2})-(\ref{hFsol_eq:2}), 
and (\ref{low:decomp}) with $\sigma=2$, and using  
Corollary \ref{cor:Gam2}, we have Theorem \ref{thm:Gam2}.
This completes the proof of Theorem \ref{thm:Gam2}.

\subsection{Analysis on $\Gamma_{3}^{\pm}$}
Our aim here is to show the following theorem for the operators defined 
in (\ref{low:decomp}) with $\si=3$.
\begin{theo}\label{thm:Gam3}
Let $1\leq r\leq 2\leq q\leq\infty$, $(\al^\pr,\al)\in\BN_0^{N-1}\times\BN_0^N$, and $F=(f,d)\in L_r(\uhs)^N\times L_r(\tws)$.
Then there exist positive constants $\de_0$, $A_0$, and $C$ such that for any $t\geq 1$
\begin{align*}
&\|(\pa_t S_0^{f,3}(t;A_0)F,\nabla\Pi_0^{f,3}(t;A_0)F)\|_{L_q(\uhs)} \\
&\quad+\|(D_{x^\pr}^{\al^\pr}S_{0}^{f,3}(t;A_0)F,
D_x^\al\pa_t \CE(T_0^{f,3}(t;A_0)F),D_x^\al\nabla\CE(T_0^{f,3}(t;A_0)F))\|_{W_q^2(\uhs)} \\
&\leq C e^{-\delta_0 t}\|f\|_{L_r(\uhs)}, \\
&\|(\pa_t S_0^{d,3}(t;A_0)F,\nabla\Pi_0^{d,3}(t;A_0)F)\|_{L_q(\uhs)} \\
&\quad+\|(D_{x^\pr}^{\al^\pr} S_{0}^{d,3}(t;A_0)F,D_x^\al\nabla\CE(T_0^{d,3}(t;A_0)F))\|_{W_q^2(\uhs)}
\leq C  e^{-\delta_0 t}\|d\|_{L_r(\tws)}.
\end{align*}
\end{theo}
In order to show Theorem \ref{thm:Gam3}, we start with the following lemma.
\begin{lemm}\label{lem:Gam3}
Let $1\leq r\leq 2\leq q\leq\infty$, and let $f\in L_{r}(\uhs)^{N}$ and $d\in L_{r}(\mathbf{R}^{N-1})$.
We use the operators defined in $(\ref{op:low})$ with the forms:
\begin{equation*}
k_n(\xi^\prime,\lambda)=\kappa_n(\xi^\prime,\lambda)/L(A,B),\quad
\ell_n(\xi^\prime,\lambda)=m_n(\xi^\prime,\lambda)/L(A,B).
\end{equation*}
\begin{enumerate}[$(1)$]
\item
Let $s\geq0$ and suppose that there exist positive constants $A_1\in(0,1)$ and $C$ such that for any $\lambda\in \Gamma_{3}^{\pm}$
and $A\in(0,A_1)$
\begin{equation*}
|\kappa_{n}(\xi^\prime,\lambda)|\leq C(|\lambda|^{1/2}+A)^2 A^{1+s}\enskip(n=1,2,4,5,6),
\quad|\kappa_{3}(\xi^\prime,\lambda)|\leq C(|\lambda|^{1/2}+A)^2A^s.
\end{equation*}
Then there exist positive constants $\de_{0}$, $A_0\in(0,A_1)$, and $C$ such that
for any $t\geq1$
\begin{equation*}
\|K_{n}^{\pm,3}(t;A_0)f\|_{L_{q}(\mathbf{R}_{+}^{N})}
\leq C e^{-\delta_{0}t} \|f\|_{L_{r}(\mathbf{R}_{+}^{N})}\quad(n=1,\dots,6).
\end{equation*}
\item
Let $s\geq0$ and suppose that there exist positive constants $A_1\in(0,1)$ and $C$ such that for any $\lambda\in\Gamma_{3}^{\pm}$
and $A\in(0,A_1)$
\begin{align*}
|m_{1}(\xi^\prime,\lambda)|&\leq C(|\lambda|^{1/2}+A)^2 A^{1+s}, \quad
|m_{2}(\xi^\prime,\lambda)|\leq C(|\lambda|^{1/2}+A)^2A^s,\\
|m_3(\xi^\prime,\lambda)|&\leq C|\lambda|^{1/2}(|\lambda|^{1/2}+A)^2 A^{1+s}.
\end{align*}
Then there exist positive constants $A_0\in(0,A_1)$, $\delta_{0}$, and $C$ such that for any $t\geq 1$
\begin{equation*}
\|L_{n}^{\pm,3}(t;A_0)d\|_{L_{q}(\mathbf{R}_{+}^{N})}
\leq C e^{-\delta_{0}t}\|d\|_{L_{r}(\mathbf{R}^{N-1})}\quad(n=1,2,3).
\end{equation*}
\end{enumerate}
\end{lemm}
\begin{proof}
We use the abbreviations: $\|\cdot\|_{2}=\|\cdot\|_{L_{2}(\mathbf{R}^{N-1})}$,
$\widehat{f}(y_{N})=\widehat{f}(\xi^{\prime},y_{N})$, and $\wt=t+1$ for $t>0$ in this proof,
and consider only the estimates on $\Gamma_3^{+}$, 
because the estimates on $\Gamma_3^-$ can be shown similarly. \\
(1) First, we show the inequality for $K_{1}^{+,3}(t)$.
Noting that $\lambda=-\gamma_{0}+ i\widetilde{\gamma}_{0}+ue^{i(\pi-\ep_0)}$
for $u\in[0,\infty)$ on $\Gamma_{3}^{+}$, we have, by (\ref{op:low}),
\begin{multline*}
[K_{1}^{+,3}(t)f](x)=
\int_{0}^{\infty}\CF_{\xi^{\prime}}^{-1}\left[\int_{0}^{\infty}e^{\{-\gamma_{0}+ i\widetilde{\gamma}_{0}+ue^{i(\pi-\ep_0)}\}t} \right. \\
\left.\times\varphi_{0}(\xi^{\prime})\frac{\kappa_{1}(\xi^\prime,\lambda)}{L(A,B)}e^{-A(x_{N}+y_{N})}e^{i(\pi-\varepsilon_0)}\,du
\widehat{f}(y_{N})\right](x^{\prime})\,dy_{N}.
\end{multline*}
Since $e^{-(\gamma_{0}/2)t}e^{A^{2}\wt}\leq Ce^{-A^{2}\wt}$ for any $A\in(0,A_0)$ by choosing some $A_0\in(0,A_1)$,
we obtain by Lemma \ref{lem:symbol} (3),
$L_{q}\text{-}L_{r}$ estimates of the $(N-1)$ dimensional heat kernel, and Parseval's theorem
\begin{align*}
&\|[K_{1}^{+,3}(t;A_0)f](\cdot\,,x_{N})\|_{L_{q}(\tws)} \\
&\quad\leq
C\wt^{-\frac{N-1}{2}\left(\frac{1}{2}-\frac{1}{q}\right)}\int_{0}^{\infty}
\left\|\int_{0}^{\infty}\varphi_{0}(\xi^{\prime})e^{A^{2}\wt}
e^{-(\gamma_{0}+u\cos{\varepsilon_0})t}\frac{A^{1+s}}{|\lambda|}e^{-A(x_{N}+y_{N})}\,du\widehat{f}(y_{N})\right\|_{2}\,dy_{N} \\
&\quad\leq
C\wt^{-\frac{N-1}{2}\left(\frac{1}{2}-\frac{1}{q}\right)}e^{-(\gamma_{0}/2)t}\int_{0}^{\infty}
\left\|\int_{0}^{\infty}\frac{e^{-u(\cos{\varepsilon_0})t}}{|\lambda|}\,du\,e^{-A^2\wt}Ae^{-A(x_{N}+y_{N})}
\,\widehat{f}(y_{N})\right\|_{2}\,dy_{N} \\
&\quad\leq
C\wt^{-\frac{N-1}{2}\left(\frac{1}{2}-\frac{1}{q}\right)}e^{-(\gamma_{0}/2)t}
\int_{0}^{\infty}\left\|e^{-A^2\widetilde{t}}Ae^{-A(x_N+y_N)}\wh{f}(y_N)\right\|_2dy_N
\end{align*}
for any $t\geq1$ with some positive constant $C$, where we note that $|\la|\geq\ga_\infty$ on $\Ga_3^+$
and $A^s\leq C$ on $\text{supp\,}\ph_0$.
We thus obtain the required inequality of $K_{1}^{+,3}(t;A_0)$ by Lemma \ref{lemm:fund3}.
Analogously, we can show the case of $n=2,4,5,6$ by
using the fact that
\begin{equation}\label{140923_6}
|e^{-Ba}|\leq C e^{-Ca},\quad
|\CM(a)|\leq C|\la|^{-1}e^{-Ca}\leq Ce^{-Ca}
\end{equation}
for any $a>0$ and $\la\in\Ga_3^+$ with some positive constant $C$ by Lemma \ref{lem:symbol} (1) and \eqref{deriv:M}.

We finally show the inequality for $K_{3}^{+,3}(t;A_0)$.
By H\"{o}lder's inequality and \eqref{140923_6},
we easily have for $r'=r/(r-1)$
\begin{align*}
&\|[K_{3}^{+,3}(t;A_0)f](\cdot,x_N)\|_{L_{q}(\tws)} \\
&\leq C\wt^{-\frac{N-1}{2}\left(\frac{1}{2}-\frac{1}{q}\right)}e^{-(\ga_0/2)t}\int_{0}^{\infty}
\left\|\int_{0}^{\infty}\frac{e^{-u(\cos{\varepsilon_0})t}}{|\lambda|}e^{-C|\lambda|^{\frac{1}{2}}(x_{N}+y_{N})}
\,du\,e^{-A^{2}\wt}\widehat{f}(y_{n})\right\|_{2}\,dy_{N} \\
&\leq Ce^{-(\ga_0/2)t}
\int_{0}^{\infty}\frac{e^{-u(\cos{\varepsilon_0})t}e^{-C|\lambda|^{\frac{1}{2}}x_{N}}}{|\lambda|}
\left(\int_{0}^{\infty}e^{-C r^{\prime}|\lambda|^{\frac{1}{2}}y_{N}}\,dy_{N}\right)^{1/{r'}}\,du
\|f\|_{L_{r}(\uhs)},
\end{align*}
%
%
%
%
Therefore, we see that
\begin{align*}
\|K_{3}^{+,3}(t;A_0)f\|_{L_{q}(\uhs)}
&\leq
C
e^{-(\gamma_{0}/2)t}
\int_0^\infty
\frac{e^{-u(\cos{\varepsilon_0})t}}{|\lambda|^{1+1/(2q)+1/(2r')}}\,du\,\|f\|_{L_r(\uhs)}\\
&\leq
Ce^{-(\gamma_{0}/2)t}\|f\|_{L_r(\uhs)}
\end{align*}
for any $t\geq 1$ with some positive constant $C$.
%
%
%
%
%
%
%
\\
(2)
Employing an argumentation similar to \thetag{1}
and using \eqref{140923_6} for $L_3^{+,3}(t;A_0)$, 
we can prove \thetag{2}, so that we may omit the detailed 
proof of \thetag2.  This completes the proof of Lemma \ref{lem:Gam3}.
\end{proof}
We see that by Lemma \ref{lem:symbol}
there exist positive constants $A_1\in(0,1)$ and $C$ such that for any $A\in(0,A_1)$ and $\lambda\in\Gamma_3^\pm$
we have
\begin{equation*}
\begin{aligned}
|\CV_{jk}^{BB}(\xi^\prime,\lambda)|&\leq C, &|\CV_{jk}^{B\CM}(\xi^\prime,\lambda)|&\leq CA, &|\CV_{jk}^{\CM B}(\xi^\prime,\lambda)|&\leq CA, \\
|\CV_{jk}^{\CM\CM}(\xi^\prime,\lambda)|&\leq CA, &|\CP_{j}^{AA}(\xi^\prime,\lambda)|&\leq CA, &|\CP_{j}^{A\CM}(\xi^\prime,\lambda)|&\leq CA
\end{aligned}
\end{equation*}
for $j,k=1,\dots,N$.
Therefore, remembering  (\ref{Fsol_eq:2})-(\ref{hFsol_eq:2}), 
and (\ref{low:decomp}) with $\sigma=3$, and using  
Lemma \ref{lem:Gam3}, we have Theorem \ref{thm:Gam3}.
This completes the proof of Theorem \ref{thm:Gam3}.

We finally consider the term $\partial_{t}\CE(T_0^d(t;A_0)F)$ given by
\begin{align*}
&\partial_{t}\CE({T}_{0}^{d}(t;A_0)F)=\mathcal{F}_{\xi^{\prime}}^{-1}\left[\frac{1}{2\pi i}\int_{\Gamma}e^{\lambda t}
\frac{\varphi_{0}(\xi^{\prime})\lambda D(A,B)}{(B+A)L(A,B)}\,d\lambda\,e^{-Ax_{N}}\widehat{d}(\xi^{\prime})\right](x^{\prime}) \\
&\quad=\mathcal{F}_{\xi^{\prime}}^{-1}\left[\frac{1}{2\pi i}\int_{\Gamma}e^{\lambda t}\,d\lambda
\,\varphi_{0}(\xi^{\prime})e^{-Ax_{N}}\widehat{d}(\xi^{\prime})\right](x^{\prime}) \\
&\quad-\mathcal{F}_{\xi^{\prime}}^{-1}\left[\frac{1}{2\pi i}\int_{\Gamma}e^{\lambda t}
\frac{\varphi_{0}(\xi^{\prime})A(c_{g}+c_{\sigma}A^{2})}{L(A,B)}\,d\lambda\,e^{-Ax_{N}}\widehat{d}(\xi^{\prime})\right](x^{\prime}),
\end{align*}
where we have used the relations: $D(A,B)=(B-A)^{-1}\{L(A,B)-A(c_{g}+c_{\sigma}A^{2})\}$ and $B^{2}-A^{2}=\lambda$.
Note that the first term vanishes by Cauchy's integral theorem,
so that it suffices to consider the second term only. Set
\begin{equation*}
I_{\sigma}^{\pm}(t;A_0)=-\mathcal{F}_{\xi^{\prime}}^{-1}\left[\frac{1}{2\pi i}\int_{\Gamma_{\sigma}^{\pm}}e^{\lambda t}
\frac{\varphi_{0}(\xi^{\prime})A(c_{g}+c_{\sigma}A^{2})}{L(A,B)}\,d\lambda\,e^{-Ax_{N}}\widehat{d}(\xi^{\prime})
\right](x^{\prime})\quad(\sigma=0,1,2,3).
\end{equation*}
Since by Lemma \ref{lem1:Gam2} there exist positive constants $A_1\in(0,1)$ and $C$ such that
for any $\la\in\Ga_2^\pm$ and $A\in(0,A_1)$
\begin{equation*}
|A(c_{g}+c_{\sigma}A^{2})/L(A,B)|\leq C(A\sqrt{1-u}+\sqrt{u}+A^{1/4})^{-4}A,
\end{equation*}
by Lemma \ref{lem2:Gam2} we have for $t>0$, $\alpha\in\BN_0^N$ with $|\al|\neq0$, and $1\leq r\leq 2\leq q\leq \infty$
\begin{equation*}
\|D_{x}^{\alpha}I_{2}^{\pm}(t;A_0)\|_{L_{q}(\mathbf{R}_{+}^{N})}\leq C(t+1)^{-\frac{N-1}{2}\left(\frac{1}{r}-\frac{1}{q}\right)
-\frac{1}{2}\left(\frac{1}{2}-\frac{1}{q}\right)-\frac{|\alpha|}{2}}\|d\|_{L_{r}(\mathbf{R}^{N-1})}
\end{equation*}
with some positive constant $C$. If $(q,r)\neq(2,2)$, then we also have
\begin{equation*}
\|I_{2}^{\pm}(t;A_0)\|_{L_{q}(\mathbf{R}_{+}^{N})}\leq C(t+1)^{-\frac{N-1}{2}\left(\frac{1}{r}-\frac{1}{q}\right)
-\frac{1}{2}\left(\frac{1}{2}-\frac{1}{q}\right)}\|d\|_{L_{r}(\mathbf{R}^{N-1})}.
\end{equation*}
In addition, by Lemma \ref{lem:Gam0}, \ref{lem2:Gam1}, and \ref{lem:Gam3}, we have  
\begin{align*}
\|D_{x}^{\alpha}I_{n}^{\pm}(t;A_0)\|_{L_{q}(\mathbf{R}_{+}^{N})}&\leq C(t+1)^{-\frac{N-1}{2}\left(\frac{1}{r}-\frac{1}{q}\right)
-\frac{1}{2}\left(\frac{1}{2}-\frac{1}{q}\right)-\frac{|\alpha|}{2}}\|d\|_{L_{r}(\mathbf{R}^{N-1})}\quad(n=0,1), \\
\|D_{x}^{\alpha}I_{3}^{\pm}(t;A_0)\|_{L_{q}(\mathbf{R}_{+}^{N})}&\leq C e^{-\delta_{0}t}\|d\|_{L_{r}(\mathbf{R}^{N-1})}
\end{align*}
for any $t\geq 1$, $\alpha\in\BN_0^N$, and $1\leq r\leq 2\leq q\leq\infty$
with some positive constant $C$.
Thus, we have
\begin{align*}
&\|D_x^\al\partial_{t}\CE({T}_{0}^{d}(t;A_0)F)\|_{L_q(\uhs)} \\
&\quad\leq 
C(t+1)^{-\frac{N-1}{2}\left(\frac{1}{r}-\frac{1}{q}\right)
-\frac{1}{2}\left(\frac{1}{2}-\frac{1}{q}\right)-\frac{|\alpha|}{2}}\|d\|_{L_{r}(\mathbf{R}^{N-1})}
\quad(1\leq r\leq 2\leq q\leq\infty,\ |\al|\neq0), \\
&\|\partial_{t}\CE({T}_{0}^{d}(t;A_0)F)\|_{L_q(\uhs)} \\
&\quad\leq 
C(t+1)^{-\frac{N-1}{2}\left(\frac{1}{r}-\frac{1}{q}\right)
-\frac{1}{2}\left(\frac{1}{2}-\frac{1}{q}\right)}\|d\|_{L_{r}(\mathbf{R}^{N-1})}
\quad(\text{$1\leq r\leq 2\leq q\leq\infty$ and $(q,r)\neq(2,2)$})
\end{align*}
for any $t\geq 1$ with some positive constant $C$,
which, combined with Theorem \ref{thm:Gam0}, \ref{thm:Gam1}, \ref{thm:Gam2},
and \ref{thm:Gam3}, completes the proof of \eqref{140917_1} in Theorem \ref{theo:main} (2),
because
\begin{align*}
&S_0(t)F=\sum_{g\in\{f,d\}}\sum_{\si=0}^3 S_{0}^{g,\si}(t;A_0)F,\quad
\Pi_0(t)F=\sum_{g\in\{f,d\}}\sum_{\si=0}^3\Pi_{0}^{g,\si}(t;A_0)F, \\
&T_0(t)F=\sum_{g\in\{f,d\}}\sum_{\si=0}^3T_{0}^{g,\si}(t;A_0)F.
\end{align*}
%
%

\section{Analysis of high frequency part}\label{sec:high}
In this section, we show the estimate \eqref{140917_2} in Theorem \ref{theo:main} (2).
If we consider the Lopatinskii determinant $L(A,B)$ defined by (\ref{0921_1}) as a polynomial
with respect to $B$, it has the following four roots:
\begin{equation}\label{expan:high}
B_{j}=a_{j}A+\frac{c_{\sigma}}{4(1-a_{j}-a_{j}^{3})}
+\frac{(1+3a_{j}^{2})c_{\sigma}^{2}}{32(1-a_{j}-a_{j}^{3})^{3}}\frac{1}{A}
+O\left(\frac{1}{A^{2}}\right)
\quad\text{as $A\to\infty$},
\end{equation}
where $a_{j}$ $(j=1,\dots,4)$ are the solutions to the equation: $x^{4}+2x^{2}-4x+1=0$.
We have the following informations about $a_j$:
$a_1$ and $a_2$ are real numbers
such that $a_1=1$ and $0<a_2<1/2$,
and $a_3$ and $a_4$ are complex numbers satisfying ${\rm Re}\,a_j<0$ for $j=3,4$. 
%
%
%
%
We define $\la_j$ by $\lambda_j=B_j^2-A^2$ for $j=1,2$, and then
\begin{equation}\label{expan:high2}
\lambda_{1}=-\frac{c_{\sigma}}{2}A-\frac{3}{16}c_{\sigma}^{2}+O(\frac{1}{A}),
\enskip\lambda_{2}=-(1-a_{2}^{2})A^{2}+\frac{a_{2}c_{\sigma}}{2(1-a_{2}-a_{2}^{3})}A+O(1)
\quad\text{as $A\to\infty$}.
\end{equation}
Let $L_{0}=\{\lambda\in \mathbf{C}\ |\ L(A,B)=0,\ {\rm Re}B\geq0,\ A\in{\rm supp}\,\varphi_{\infty}\}$,
where $\varphi_\infty$ is defined in (\ref{A0}),
and then we see,  by the expansion formulas (\ref{expan:low2}), (\ref{expan:high2}), 
and Lemma \ref{lem:anal}, that  
 there exist positive numbers 
$0<\varepsilon_{\infty}<\pi/2$ and $\lambda_{\infty}>0$ such that
 $L_{0}\subset\Sigma_{\varepsilon_{\infty}}
\cap\{z\in\mathbf{C}\ |\ {\rm Re}z<-\lambda_{\infty}\}$.
Set $\gamma_{\infty}=\min\{\lambda_{\infty},4^{-1}\times(A_{0}/6)^{2}\}$ for $A_{0}$ defined in (\ref{A0}),
and set, for \eqref{sol:decomp} and $g\in\{f,d\}$,
\begin{equation*}
S_\infty^g(t)=S_\infty^g(t;A_0),\quad
\Pi_\infty^g(t)=\Pi_\infty^g(t;A_0), \quad
T_\infty^g(t)=T_\infty^g(t;A_0).
\end{equation*}
In order to estimate each term above, we use the integral paths:
\begin{align*}
\Gamma_{4}^{\pm}
&=\{\lambda\in\BC\mid\lambda=-\gamma_{\infty}\pm i u,\ u:0\to\widetilde{\gamma}_\infty\}, \\
\Gamma_{5}^{\pm}&=\{\lambda\in\BC\mid\lambda=-\gamma_{\infty}\pm i\widetilde{\gamma}_\infty
+u e^{\pm i(\pi-\varepsilon_{\infty})},\ u:0\to\infty \},
\end{align*}
where $\widetilde{\gamma}_\infty=(\tan\ep_\infty)(\widetilde{\lambda}_{0}(\ep_\infty)+\gamma_\infty)$
and $\widetilde{\lambda}_0(\ep_\infty)$ is the same constant as in \eqref{Gamma} with $\ep=\ep_\infty$.
Furthermore, for $g\in\{f,d\}$, setting $v_{\infty}^{g}(x,\lambda)=(v_{1,\infty}^{g}(x,\lambda),\dots,v_{N,\infty}^{g}(x,\lambda))^{T}$ and
\begin{align*}
&v_{j,\infty}^{g}(x,\lambda)=\CF_{\xi^{\prime}}^{-1}[\varphi_{\infty}(\xi^{\prime})\widehat{v}_{j}^{g}(\xi^{\prime},x_{N},\lambda)](x^{\prime})
\quad(j=1,\dots,N), \displaybreak[1] \\
&\pi_{\infty}^{g}(x,\lambda)=\CF_{\xi^{\prime}}^{-1}[\varphi_{\infty}(\xi^{\prime})\widehat{\pi}^{g}(\xi^{\prime},x_{N},\lambda)](x^{\prime}), \\
&h_{A,\infty}^{g}(x,\lambda)=\CF_{\xi^{\prime}}^{-1}[\varphi_{\infty}(\xi^{\prime})e^{-Ax_N}\widehat{h}^{g}(\xi^{\prime},\lambda)](x^{\prime})
\end{align*}
by (\ref{Fsol_eq:2})-(\ref{A0}), we have, by Cauchy's integral theorem, the following decompositions: 
\begin{equation*}
S_{\infty}^{g}(t)F=\sum_{\sigma=4}^{5}S_{\infty}^{g,\sigma}(t)F,\enskip
\Pi_{\infty}^{g}(t)F=\sum_{\sigma=4}^{5}\Pi_{\infty}^{g,\sigma}(t)F,\enskip
\CE(T_{\infty}^{g}(t)F)=\sum_{\sigma=4}^{5}\CE({T}_{\infty}^{g,\sigma}(t)F),
\end{equation*}
where the right-hand sides are given by
\begin{align}\label{high:decomp}
&S_{\infty}^{g,\sigma}(t)F=\frac{1}{2\pi i}\int_{\Gamma_{\sigma}^{+}\cup\Gamma_{\sigma}^{-}}e^{\lambda t}v_{\infty}^{g}(x,\lambda)\,d\lambda,
\quad\Pi_{\infty}^{g,\sigma}(t)F=\frac{1}{2\pi i}\int_{\Gamma_{\sigma}^{+}\cup\Gamma_{\sigma}^{-}}
e^{\lambda t}\pi_{\infty}^{g}(x,\lambda)\,d\lambda, \notag \\
&\CE({T}_{\infty}^{g,\sigma}(t)F)=\frac{1}{2\pi i}\int_{\Gamma_{\sigma}^{+}\cup\Gamma_{\sigma}^{-}}e^{\lambda t}
h_{A,\infty}^{g}(x,\lambda)\,d\lambda.
\end{align}
By the relation $1=B^{2}/B^{2}=(\lambda+A^{2})/B^{2}$,
we write $v_{\infty}^{f}$, $\pi_{\infty}^{f}$, and $h_{A,\infty}^{f}$ as follows:
For $j=1,\dots,N$, $\wh{f}_j(y_N)=\wh{f}_j(\xi',y_N)$, and $\ph_\infty=\ph_\infty(\xi')$,
\begin{align}\label{1011_3}
&v_{j,\infty}^{f}(x,\la)
=\sum_{k=1}^{N}\int_{0}^{\infty}\CF_{\xi^{\prime}}^{-1}\left[\varphi_{\infty}
\frac{\CV_{jk}^{BB}(\xi^\prime,\lambda)(c_{g}+c_{\sigma}A^{2})}{AL(A,B)}Ae^{-B(x_{N}+y_{N})}
\widehat{f}_{k}(y_{N})\right](x^{\prime})\,dy_{N} \displaybreak[0] \notag \\
&+\sum_{k=1}^{N}\int_{0}^{\infty}\CF_{\xi^{\prime}}^{-1}\left[\varphi_{\infty}
\frac{\lambda|\lambda|^{-\frac{1}{2}}\CV_{jk}^{B\CM}(\xi^\prime,\lambda)(c_{g}+c_{\sigma}A^{2})}{AB^{2}L(A,B)}
A|\lambda|^{\frac{1}{2}}e^{-Bx_{N}}\CM(y_{N})\widehat{f}_{k}(y_{N})\right](x^{\prime})\,dy_{N} \notag \\
&+\sum_{k=1}^{N}\int_{0}^{\infty}\CF_{\xi^{\prime}}^{-1}\left[\varphi_{\infty}
\frac{\CV_{jk}^{B\CM}(\xi^\prime,\lambda)(c_{g}+c_{\sigma}A^{2})}{B^{2}L(A,B)}A^{2}e^{-Bx_{N}}\CM(y_{N})
\widehat{f}_{k}(y_{N})\right](x^{\prime})\,dy_{N} \notag \\
&+\sum_{k=1}^{N}\int_{0}^{\infty}\CF_{\xi^{\prime}}^{-1}\left[\varphi_{\infty}
\frac{\lambda|\lambda|^{-\frac{1}{2}}\CV_{jk}^{\CM B}(\xi^\prime,\lambda)(c_{g}+c_{\sigma}A^{2})}{AB^{2}L(A,B)}
A|\lambda|^{\frac{1}{2}}\CM(x_{N})e^{-By_{N}}\widehat{f}_{k}(y_{N})\right](x^{\prime})\,dy_{N}\displaybreak[0]\notag\\
&+\sum_{k=1}^{N}\int_{0}^{\infty}\CF_{\xi^{\prime}}^{-1}\left[\varphi_{\infty}
\frac{\CV_{jk}^{\CM B}(\xi^\prime,\lambda)(c_{g}+c_{\sigma}A^{2})}{B^{2}L(A,B)}A^{2}\CM(x_{N})e^{-By_{N}}
\widehat{f}_{k}(y_{N})\right](x^{\prime})\,dy_{N}\displaybreak[0]\notag\\
&+\sum_{k=1}^{N}\int_{0}^{\infty}\CF_{\xi^{\prime}}^{-1}\left[\varphi_{\infty}
\frac{\CV_{jk}^{\CM\CM}(\xi^\prime,\lambda)(c_{g}+c_{\sigma}A^{2})}{AB^{2}L(A,B)}A\lambda\CM(x_{N})\CM(y_{N})
\widehat{f}_{k}(y_{N})\right](x^{\prime})\,dy_{N}\displaybreak[0]\notag\\
&+\sum_{k=1}^{N}\int_{0}^{\infty}\CF_{\xi^{\prime}}^{-1}\left[\varphi_{\infty}
\frac{\CV_{jk}^{\CM\CM}(\xi^\prime,\lambda)(c_{g}+c_{\sigma}A^{2})}{AB^{2}L(A,B)}A^{3}\CM(x_{N})\CM(y_{N})
\widehat{f}_{k}(y_{N})\right](x^{\prime})\,dy_{N},\displaybreak[0]\notag\\
&\pi_{\infty}^{f}(x,\la)
=\sum_{k=1}^{N}\int_{0}^{\infty}\CF_{\xi^{\prime}}^{-1}\left[\varphi_{\infty}
\frac{\CP_k^{AA}(\xi^\prime,\lambda)(c_{g}+c_{\sigma}A^{2})}{AL(A,B)}Ae^{-A(x_{N}+y_{N})}
\widehat{f}_{k}(y_{N})\right](x^{\prime})\,dy_{N}\displaybreak[0]\notag\\
&+\sum_{k=1}^{N}\int_{0}^{\infty}\CF_{\xi^{\prime}}^{-1}
\left[\varphi_{\infty}\frac{\CP_k^{A\CM}(\xi^\prime,\lambda)(c_{g}+c_{\sigma}A^{2})}
{A^{2}L(A,B)}A^{2}e^{-Ax_{N}}\CM(y_{N})\widehat{f}_{k}(y_{N})\right](x^{\prime})\,dy_{N},\displaybreak[0]\notag\\
&h_{A,\infty}^{f}(x,\la)
=-\sum_{k=1}^{N-1}\int_{0}^{\infty}\CF_{\xi^{\prime}}^{-1}\left[\varphi_{\infty}
\frac{i\xi_k(B-A)}{A(B+A)L(A,B)}Ae^{-A(x_{N}+y_{N})}\widehat{f}_{k}(y_{N})\right](x^{\prime})\,dy_{N}\displaybreak[0]\notag\\
&-\int_{0}^{\infty}\CF_{\xi^{\prime}}^{-1}\left[\varphi_{\infty}
\frac{1}{L(A,B)}Ae^{-A(x_{N}+y_{N})}\widehat{f}_{N}(y_{N})\right](x^{\prime})\,dy_{N}\displaybreak[0]\notag\\
&+\sum_{k=1}^{N-1}\int_{0}^{\infty}\CF_{\xi^{\prime}}^{-1}\left[\varphi_{\infty}
\frac{2i\xi_k B}{A(B+A)L(A,B)}A^{2}e^{-Ax_{N}}\CM(y_{N})\widehat{f}_{k}(y_{N})\right](x^{\prime})\,dy_{N}\notag\\
&+\int_{0}^{\infty}\CF_{\xi^{\prime}}^{-1}\left[\varphi_{\infty}
\frac{2A}{(B+A)L(A,B)}A^{2}e^{-Ax_{N}}\CM(y_{N})\widehat{f}_{N}(y_{N})\right](x^{\prime})\,dy_{N}.
\end{align}
Moreover, using the relations:
\begin{align}\label{vole}
	e^{-Bx_{N}}\widehat{g}(0)=&\int_{0}^{\infty}Be^{-B(x_{N}+y_{N})}\widehat{g}(y_{N})\,dy_{N}
			-\int_{0}^{\infty}e^{-B(x_{N}+y_{N})}\widehat{D_{N}g}(y_{N})\,dy_{N}, \notag \\
	\mathcal{M}(x_{N})\widehat{g}(0)=&\int_{0}^{\infty}\left(e^{-B(x_{N}+y_{N})}+A\mathcal{M}(x_{N}+y_{N})\right)
			\widehat{g}(y_{N})\,dy_{N} \notag \\
		&+\int_{0}^{\infty}\mathcal{M}(x_{N}+y_{N})\widehat{D_{N}g}(y_{N})\,dy_{N},
\end{align}
where $\widehat{g}(y_N)=\widehat{g}(\xi^\prime,y_N)$, and 
 using the identity: $1=A^{2}/A^{2}=-\sum_{k=1}^{N-1}(i\xi_{k})^{2}/A^{2}$, 
we write  $v_{\infty}^{d}$, $\pi_{\infty}^{d}$, and $h_{A,\infty}^{d}$
as follows: For $j=1,\dots,N-1$,
%
%
%
%
%
\begin{align}\label{1011_4}
&v_{j,\infty}^{d}(x,\la)
=-\int_{0}^{\infty}\CF_{\xi^{\prime}}^{-1}\left[\varphi_{\infty}
\frac{i\xi_{j}(c_{g}+c_{\sigma}A^{2})}{A^{2}L(A,B)}
Ae^{-B(x_{N}+y_{N})}\widehat{\Delta^{\prime}d}(y_{N})\right](x^{\prime})\,dy_{N} \notag \displaybreak[0]\\
&+\sum_{k=1}^{N-1}\int_{0}^{\infty}\mathcal{F}_{\xi^{\prime}}^{-1}\left[\varphi_{\infty}
\frac{\xi_{j}\xi_{k}(B-A)(c_{g}+c_{\sigma}A^{2})}{A^{3}(B+A)L(A,B)}
Ae^{-B(x_{N}+y_{N})}\widehat{D_{k}D_{N}d}(y_{N})\right](x^{\prime})\,dy_{N}\notag\displaybreak[0]\\
&-\int_{0}^{\infty}\mathcal{F}_{\xi^{\prime}}^{-1}\left[\varphi_{\infty}
\frac{i\xi_{j}(B^{2}+A^{2})(c_{g}+c_{\sigma}A^{2})}{A^{3}(B+A)L(A,B)}
A^{2}\CM(x_{N}+y_{N})\widehat{\Delta^{\prime}d}(y_{N})\right](x^{\prime})\,dy_{N}\notag\displaybreak[0]\\
&-\sum_{k=1}^{N-1}\int_{0}^{\infty}\mathcal{F}_{\xi^{\prime}}^{-1}\left[\varphi_{\infty}
\frac{\xi_{j}\xi_{k}(B^{2}+A^{2})(c_{g}+c_{\sigma}A^{2})}{A^{4}(B+A)L(A,B)}
A^{2}\CM(x_{N}+y_{N})\widehat{D_{k}D_{N}d}(y_{N})\right](x^{\prime})\,dy_{N},\notag\displaybreak[0]\\
%
%
&v_{N,\infty}^{d}(x,\la)
=-\int_{0}^{\infty}\mathcal{F}_{\xi^{\prime}}^{-1}\left[\varphi_{\infty}
\frac{(B-A)(c_{g}+c_{\sigma}A^{2})}{A(B+A)L(A,B)}Ae^{-B(x_{N}+y_{N})}
\widehat{\Delta^{\prime}d}(y_{N})\right](x^{\prime})\,dy_{N}\notag\displaybreak[0]\\
&+\sum_{k=1}^{N-1}\int_{0}^{\infty}\mathcal{F}_{\xi^{\prime}}^{-1}\left[\varphi_{\infty}
\frac{i\xi_{k}(c_{g}+c_{\sigma}A^{2})}{A^{2}L(A,B)}Ae^{-B(x_{N}+y_{N})}
\widehat{D_{k}D_{N}d}(y_{N})\right](x^{\prime})\,dy_{N}\notag\displaybreak[0]\\
&+\int_{0}^{\infty}\mathcal{F}_{\xi^{\prime}}^{-1}\left[\varphi_{\infty}
\frac{(B^{2}+A^{2})(c_{g}+c_{\sigma}A^{2})}{A^{2}(B+A)L(A,B)}A^{2}\CM(x_{N}+y_{N})
\widehat{\Delta^{\prime}d}(y_{N})\right](x^{\prime})\,dy_{N}\notag\displaybreak[0]\\
&-\sum_{k=1}^{N-1}\int_{0}^{\infty}\mathcal{F}_{\xi^{\prime}}^{-1}\left[\varphi_{\infty}
\frac{i\xi_{k}(B^{2}+A^{2})(c_{g}+c_{\sigma}A^{2})}{A^{3}(B+A)L(A,B)}A^{2}\CM(x_{N}+y_{N})
\widehat{D_{k}D_{N}d}(y_{N})\right](x^{\prime})\,dy_{N},\notag\displaybreak[0]\\
%
%
&\pi_{\infty}^{d}(x,\la)
=-\int_{0}^{\infty}\mathcal{F}_{\xi^{\prime}}^{-1}\left[\varphi_{\infty}
\frac{(B^{2}+A^{2})(c_{g}+c_{\sigma}A^{2})}{A^{2}L(A,B)}Ae^{-A(x_{N}+y_{N})}
\widehat{\Delta^{\prime}d}(y_{N})\right](x^{\prime})\,dy_{N}\displaybreak[0]\notag\\
&+\sum_{k=1}^{N-1}\int_{0}^{\infty}\mathcal{F}_{\xi^{\prime}}^{-1}\left[\varphi_{\infty}
\frac{i\xi_{k}(B^{2}+A^{2})(c_{g}+c_{\sigma}A^{2})}{A^{3}L(A,B)}Ae^{-A(x_{N}+y_{N})}
\widehat{D_{k}D_{N}d}(y_{N})\right](x^{\prime})\,dy_{N},\notag\displaybreak[0]\\
%
%
&h_{A,\infty}^{d}(x,\la)
=-\int_{0}^{\infty}\mathcal{F}_{\xi^{\prime}}^{-1}\left[\varphi_{\infty}
\frac{D(A,B)}{A^{2}(B+A)L(A,B)}Ae^{-A(x_{N}+y_{N})}\widehat{\Delta^{\prime}d}(y_{N})
\right](x^{\prime})\,dy_{N}\notag\displaybreak[0]\\
&+\sum_{k=1}^{N-1}\int_{0}^{\infty}\mathcal{F}_{\xi^{\prime}}^{-1}\left[\varphi_{\infty}
\frac{i\xi_{k}D(A,B)}{A^{3}(B+A)L(A,B)}Ae^{-A(x_{N}+y_{N})}\widehat{D_{k}D_{N}d}(y_{N})\right](x^{\prime})\,dy_{N}.
%
\end{align}
\begin{rema}
We extend $d\in W_{p}^{2-1/p}(\tws)$ to a function $\widetilde{d}$,
which is defined on $\uhs$ and satisfies 
$\|\widetilde{d}\|_{W_{p}^{2}(\uhs)}\leq C\|d\|_{W_{p}^{2-1/p}(\tws)}$
for a positive constant $C$ independent of $d$ and $\wit{d}$.
For simplicity, such a $\widetilde{d}$ is denoted by $d$ again in the present section.
\end{rema}
To estimates all the terms given in (\ref{1011_3}) and (\ref{1011_4}), we introduce the following operators:
\begin{align}\label{K}
[K_{1}(\lambda)f](x)&=\int_{0}^{\infty}\CF_{\xi^{\prime}}^{-1}\left[\varphi_\infty(\xi^\prime)k_1(\xi^{\prime},\lambda)Ae^{-A(x_{N}+y_{N})}
\widehat{f}(\xi^{\prime},y_{N})\right](x^{\prime})\,dy_{N}, \displaybreak[0] \notag \\
[K_{2}(\lambda)f](x)&=\int_{0}^{\infty}\CF_{\xi^{\prime}}^{-1}\left[\varphi_\infty(\xi^\prime)k_2(\xi^{\prime},\lambda)A^{2}e^{-Ax_{N}}\CM(y_{N})
\widehat{f}(\xi^{\prime},y_{N})\right](x^{\prime})\,dy_{N}, \displaybreak[0] \notag \\
[K_{3}(\lambda)f](x)&=\int_{0}^{\infty}\CF_{\xi^{\prime}}^{-1}\left[\varphi_\infty(\xi^\prime)k_3(\xi^{\prime},\lambda)Ae^{-B(x_{N}+y_{N})}
\widehat{f}(\xi^{\prime},y_{N})\right](x^{\prime})\,dy_{N}, \notag \displaybreak[0] \\
[K_{4}(\lambda)f](x)&=\int_{0}^{\infty}\CF_{\xi^{\prime}}^{-1}\left[\varphi_\infty(\xi^\prime)k_4(\xi^{\prime},\lambda)A^{2}e^{-Bx_{N}}\CM(y_{N})
\widehat{f}(\xi^{\prime},y_{N})\right](x^{\prime})\,dy_{N}, \displaybreak[0] \notag \\
[K_{5}(\lambda)f](x)&=\int_{0}^{\infty}\CF_{\xi^{\prime}}^{-1}
\left[\varphi_\infty(\xi^\prime)k_5(\xi^{\prime},\lambda)A|\lambda|^{\frac{1}{2}}e^{-Bx_{N}}\mathcal{M}(y_{N})
\widehat{f}(\xi^{\prime},y_{N})\right](x^{\prime})\,dy_{N}, \displaybreak[0] \notag \\
[K_{6}(\lambda)f](x)&=\int_{0}^{\infty}\CF_{\xi^{\prime}}^{-1}\left[\varphi_\infty(\xi^\prime)k_6(\xi^{\prime},\lambda)A^{2}\CM(x_{N})e^{-By_{N}}
\widehat{f}(\xi^{\prime},y_{N})\right](x^{\prime})\,dy_{N}, \displaybreak[0] \notag \\
[K_{7}(\lambda)f](x)&=\int_{0}^{\infty}\CF_{\xi^{\prime}}^{-1}
\left[\varphi_\infty(\xi^\prime)k_7(\xi^{\prime},\lambda)A|\lambda|^{\frac{1}{2}}\CM(x_{N})e^{-By_{N}}
\widehat{f}(\xi^{\prime},y_{N})\right](x^{\prime})\,dy_{N}, \displaybreak[0] \notag \\
[K_{8}(\lambda)f](x)&=\int_{0}^{\infty}\CF_{\xi^{\prime}}^{-1}\left[\varphi_\infty(\xi^\prime)k_8(\xi^{\prime},\lambda)A^{2}\CM(x_{N}+y_{N})
\widehat{f}(\xi^{\prime},y_{N})\right](x^{\prime})\,dy_{N}, \displaybreak[0] \notag \\
[K_{9}(\lambda)f](x)&=\int_{0}^{\infty}\CF_{\xi^{\prime}}^{-1}
\left[\varphi_\infty(\xi^\prime)k_9(\xi^{\prime},\lambda)A^{3}\mathcal{M}(x_{N})\mathcal{M}(y_{N})
\widehat{f}(\xi^{\prime},y_{N})\right](x^{\prime})\,dy_{N}, \displaybreak[0] \notag \\
[K_{10}(\lambda)f](x)&=\int_{0}^{\infty}\CF_{\xi^{\prime}}^{-1}
\left[\varphi_\infty(\xi^\prime)k_{10}(\xi^{\prime},\lambda)A\lambda \mathcal{M}(x_{N})\mathcal{M}(y_{N})
\widehat{f}(\xi^{\prime},y_{N})\right](x^{\prime})\,dy_{N}.
\end{align}
%
%
We know the following proposition (cf. \cite[Lemma 5.4]{S-S:model}).
\begin{prop}\label{prop:S-S}
Let $1<p<\infty$, $0<\ep<\pi/2$, and $f\in L_p(\uhs)$,
and let $\La_\ep$ be a subset of $\Si_\ep$.
Suppose that for every $\al^\pr\in\BN_0^{N-1}$ there exists a positive constant $C=C(\al')$ such that
for any $\la\in\La_\ep$ and $\xi^\pr\in\tws\setminus\{0\}$
\begin{equation*}
|D_{\xi^\pr}^{\al^\pr}\{\ph_\infty(\xi^\pr)k_n(\xi^\pr,\la)\}|\leq CA^{-|\al^\pr|}
\quad(n=1,\dots,10).
\end{equation*}
Then there exists a positive constant $C$ such that for any $\la\in\La_\ep$
\begin{equation*}
\|K_n(\la)f\|_{L_p(\uhs)}\leq C\|f\|_{L_p(\uhs)}\quad(n=1,\dots,10).
\end{equation*}
\end{prop}

\subsection{Analysis on $\Gamma_{4}^{\pm}$}\label{subs:high1}
%
%
%
%
%
%
We first show the following lemma concerning estimates of the symbols defined in (\ref{0921_1})
%
%
\begin{lemm}\label{lem1:Gam4}
\begin{enumerate}[$(1)$]
\item
There exists a positive constant $A_{\infty}$ such that for any $A\geq A_{\infty}$ and $\lambda\in\Gamma_{4}^{\pm}$
\begin{equation*}
2^{-1}A\leq{\rm Re}B\leq|B|\leq 2A,\enskip
|D(A,B)|\geq A^{3},\enskip
|L(A,B)|\geq (c_{\sigma}/16)(8^{-1}A)^{3}.
\end{equation*}
\item
There exist  positive constants $C_1, C_2$, and $C$
such that for any $A\in[A_{0}/6,2A_{\infty}]$ and $\lambda\in\Gamma_{4}^{\pm}$,
\begin{equation*}
C_1A\leq {\rm Re}B\leq|B|\leq C_2A,\enskip
|D(A,B)|\geq CA^{3},\enskip
|L(A,B)|\geq CA^{3}.
\end{equation*}
where $A_{\infty}$ and $A_{0}$ are the same constants as in \thetag1 and in $(\ref{A0})$, respectively.
\item
Let $\alpha^{\prime}\in\mathbf{N}_{0}^{N-1}$, $s\in\mathbf{R}$, and $a>0$.
Then there exist constants  $C> 0$ and $b_{\infty}\geq1$, independent of $a$, such that
for any $\lambda\in \Gamma_{4}^{\pm}$ and $A\geq A_0/6$ with $A_0$ defined as in \eqref{A0}
\begin{equation*}
\begin{aligned}
&|D_{\xi^{\prime}}^{\alpha^{\prime}}B^{s}|\leq CA^{s-|\alpha^{\prime}|},
\quad|D_{\xi^{\prime}}^{\alpha^{\prime}}D(A,B)^{s}|\leq CA^{3s-|\alpha^{\prime}|},
\quad|D_{\xi^{\prime}}^{\alpha^{\prime}}e^{-Ba}|\leq CA^{-|\alpha^{\prime}|}e^{-b_{\infty}^{-1}Aa}, \\
&|D_{\xi^{\prime}}^{\alpha^{\prime}}L(A,B)^{-1}|\leq CA^{-3-|\alpha^{\prime}|},
\quad|D_{\xi^{\prime}}^{\alpha^{\prime}}\CM(a)|\leq CA^{-1-|\alpha^{\prime}|}e^{-b_{\infty}^{-1}Aa}.
\end{aligned}
\end{equation*}
\end{enumerate}
\end{lemm}
\begin{proof}
(1) We first consider the estimates of $B$.
For $\lambda\in\Gamma_{4}^{\pm}$, set $\sigma=\lambda+A^{2}=-\gamma_{\infty}+A^{2}\pm iu$ $(u\in[0,\widetilde{\gamma}_\infty])$
and $\theta=\arg{\sigma}$. Then we have
\begin{equation*}
{\rm Re}B=|\sigma|^{\frac{1}{2}}\cos{\frac{\theta}{2}}=\frac{|\sigma|^{\frac{1}{2}}}{\sqrt{2}}(1+\cos{\theta})^{\frac{1}{2}}
=\frac{1}{\sqrt{2}}(|\sigma|+A^{2}-\gamma_{\infty})^{\frac{1}{2}},
\end{equation*}
so that for any $A\geq A_{\infty}$
\begin{equation*}
{\rm Re}B\geq\frac{1}{\sqrt{2}}\left(2A^{2}-2\gamma_{\infty}-\widetilde{\gamma}_{\infty}\right)^{\frac{1}{2}}\geq\frac{A}{\sqrt{2}},
\end{equation*}
provided that $A_{\infty}$ satisifes
 $A_{\infty}^2\geq2\gamma_{\infty}+\widetilde{\gamma}_{\infty}$.
On the other hand, it is clear that $|B|\leq 2A$.

Next, we show the inequality for $D(A,B)$. Since
\begin{equation*}
D(A,B)=B(B^{2}+3A^{2})+A(B^{2}-A^{2})=B(\lambda+4A^{2})+\lambda A=4A^{2}B+(B+A)(-\gamma_{\infty}\pm iu),
\end{equation*}
we see, by the inequality of $B$ obtained above, that 
\begin{align*}
|D(A,B)|&\geq 4A^{2}|B|-|B+A||-\gamma_{\infty}\pm iu|\geq 4A^{2}({\rm Re}B)-(|B|+A)(\gamma_{\infty}+\widetilde{\gamma}_{\infty}) \\
&\geq 2A^{3}-3(\gamma_{\infty}+\widetilde{\gamma}_{\infty})A\geq A^{3}
\end{align*}
for any $A\geq A_{\infty}$,
provided that $A_{\infty}$ satisfies $A_{\infty}^{2}\geq 3(\gamma_{\infty}+\widetilde{\gamma}_{\infty})$.

Finally, we show the inequality for $L(A,B)$. Since
\begin{equation*}
B_{1}^{2}-B^{2}=-\frac{c_{\sigma}}{2}A-\frac{3}{16}c_{\sigma}^{2}-(-\gamma_{\infty}\pm iu)+O(\frac{1}{A})
\quad\text{as $A\to\infty$},
\end{equation*}
there exist positive constants $A_\infty$ 
and $C$ such that for any $A\geq A_\infty$ and $\lambda\in\Gamma_4^{\pm}$
we have 
$|B_{1}^{2}-B^{2}|\geq(c_{\sigma}/4)A$, which, combined 
with the inequality of $B$ obtained above 
and (\ref{expan:high}), furnishes that
\begin{equation*}
|B_{1}-B|\geq\frac{|B_{1}^{2}-B^{2}|}{|B_{1}+B|}\geq\frac{(c_{\sigma}/4)A}{4A}
\geq\frac{c_{\sigma}}{16}\quad(A\geq A_\infty\ {\rm and}\ \lambda\in\Gamma_4^\pm).
\end{equation*}
On the other hand, we have
\begin{equation*}
B_{2}^{2}-B^{2}=-(1-a_{2}^{2})A^{2}+O(A)
\quad\text{as $A\to\infty$},
\end{equation*}
so that there exists a positive number $A_\infty$ such that for any $A\geq A_\infty$ and $\lambda\in\Gamma_4^{\pm}$ we have
 $|B_{2}^{2}-B^{2}|\geq (A^{2}/2)$, from which it follows that
\begin{equation*}
|B_{2}-B|=\frac{|B_{2}^{2}-B^{2}|}{|B_{2}+B|}\geq\frac{(A^{2}/2)}{4A}=\frac{A}{8}.
\end{equation*}
Since $|B-B_{2}|\leq|B-B_{j}|$ $(j=3,4)$, we thus obtain
\begin{equation*}
|L(A,B)|\geq (c_{\sigma}/16)(8^{-1}A)^{3}\quad(A\geq A_\infty\ {\rm and}\ \lambda\in\Gamma_4^\pm).
\end{equation*}
(2)	It is sufficient to show the existence of
 positive constants $C_1$, $C_2$, and $C$ 
such that for any $A\in[A_0/6,2A_\infty]$ and $\lambda\in \Gamma_{4}^{\pm}$
\begin{equation*}
C_1\leq{\rm Re}B\leq|B|\leq C_2,\enskip
|D(A,B)|\geq C,\enskip
|L(A,B)|\geq C.
\end{equation*}
It is obvious that the inequalities for $B$ holds, so that 
we here consider $D(A,B)$ and $L(A,B)$ only.

First, we show the inequality for $D(A,B)$. Set
\begin{equation*}
\widetilde{A}=\frac{A}{2},\quad
\widetilde{\lambda}=-\gamma_{\infty}+3\widetilde{A}^{2}\pm iu\quad
\text{for $u\in[0,\widetilde{\gamma}_\infty]$},
\end{equation*}
and note that $B=(\widetilde{\lambda}+\widetilde{A}^{2})^{1/2}$.
%
%
%
%
We then see that 
\begin{equation*}
\{B/\widetilde{A}\in\mathbf{C}\ |\ \lambda\in\Gamma_{4}^{\pm}\ {\rm and}\ A\in[A_{0}/6,2A_{\infty}]\}
\subset\{z\in\mathbf{C}\ |\ 1\leq {\rm Re}z\}.
\end{equation*}
In fact, setting $\sigma=1-(\gamma_{\infty}/A^2)\pm i(u/A^2)$ and $\theta=\arg{\sigma}$, we have
\begin{align*}
{\rm Re}\frac{B}{\widetilde{A}}&=2|\sigma|^{1/2}\cos{\frac{\theta}{2}}=2|\sigma|^{1/2}
\left(\frac{1+\cos\theta}{2}\right)^{1/2}=\sqrt{2}(|\sigma|+{\rm Re}\sigma)^{1/2}\geq2({\rm Re}\sigma)^{1/2} \\
&=2\left(1-\frac{\gamma_{\infty}}{A^2}\right)^{1/2}\geq 2\left(1-\frac{4^{-1}\times(A_{0}/6)^{2}}{(A_{0}/6)^2}\right)^{1/2}
=\sqrt{3},
\end{align*}
which, combined  with Lemma \ref{lem:poly} and
the formula:
\begin{align*}
D(A,B)=B^{3}+2\widetilde{A}B^{2}+12\widetilde{A}^{2}B-8\widetilde{A}^{3}=\widetilde{A}^{3}
\left\{\left(\frac{B}{\widetilde{A}}\right)^{3}+2\left(\frac{B}{\widetilde{A}}\right)^{2}
+12\left(\frac{B}{\widetilde{A}}\right)-8\right\},
\end{align*}
furnishes the existence of a positive constant $C$ such that
for any $A\in[A_{0}/6,2A_{\infty}]$ and $\lambda\in\Gamma_{4}^{\pm}$ 
we have  $|D(A,B)|\geq C$. The inequality for $L(A,B)$ follows clearly 
from the definition of the integral path $\Gamma_{4}^{\pm}$.\\
(3)
We see, by Lemma \ref{lem1:Gam4} (1) and (2), that  there exist positive constants $C_1, C_2$, and $C$ such that 
for any $\lambda\in\Gamma_{4}^{\pm}$ and $A\geq A_0/6$
\begin{equation}\label{140923_15}
C_1A\leq{\rm Re}B\leq|B|\leq C_2A,\enskip
|D(A,B)|\geq CA^{3},\enskip
|L(A,B)|\geq CA^{3}.
\end{equation}
We thus obtain the required inequalities by using Leibniz's rule and Bell's formula, and noting
\begin{align*}
|D_{\xi^{\prime}}^{\alpha^{\prime}}D(A,B)|
&=|D_{\xi^{\prime}}^{\alpha^{\prime}}(B^{3}+AB^{2}+3A^{2}B-A^{3})|\leq CA^{3}, 
\\
|D_{\xi^{\prime}}^{\alpha^{\prime}}L(A,B)|
&=\left|D_{\xi^{\prime}}^{\alpha^{\prime}}
\left(\frac{\lambda}{B+A}D(A,B)+A(c_{g}
+c_{\sigma}A^{2})\right)\right|\leq CA^{3}
\end{align*}
for any $\alpha^{\prime}\in\mathbf{N}_{0}^{N-1}$, $\lambda\in\Gamma_{4}^{\pm}$, and $A\geq A_0/6$ by \eqref{140923_15}
(cf. \cite[Lemma 5.2, Lemma 5.3, Lemma 7.2]{S-S:model}).
\end{proof}
Now, we have a multiplier theorems on $\Gamma_{4}^{\pm}$.
\begin{lemm}\label{lem2:Gam4}
Let $1<p<\infty$, $n=1,\dots,10$, and $f\in L_{p}(\mathbf{R}_{+}^{N})$.
We use the symbols defined in $(\ref{K})$ and assume that
for any $\alpha^{\prime}\in\mathbf{N}_{0}^{N-1}$ there exists a positive constant $C=C(\alpha^{\prime})$ such that
%
%
$|D_{\xi^{\prime}}^{\alpha^{\prime}}k_n(\xi^{\prime},\lambda)|\leq CA^{-|\alpha^{\prime}|}$
%
%
for any $\lambda\in\Gamma_4^\pm$ and $A\geq A_0/6$ with $A_0$ defined as in \eqref{A0}.
Then there exists a positive constant  $C$ such that for any $\lambda\in\Gamma_4^\pm$
\begin{equation*}
\|K_{n}(\lambda)f\|_{L_{p}(\mathbf{R}_{+}^{N})}\leq C\|f\|_{L_{p}(\mathbf{R}_{+}^{N})}
\quad(n=1,\dots,10).
\end{equation*}
\end{lemm}
\begin{proof}
Employing the similar argumentation to 
the proof of \cite[Lemma 5.4]{S-S:model} and using Lemma \ref{lem1:Gam4},
we can prove the lemma.  
\end{proof}
%
%
%
%
%
%
%
%
%
%
%
%
%
%
%
%
%
%
%
%
%
%
By (\ref{symbol}), (\ref{1011_3}), (\ref{1011_4}), Lemma \ref{lem1:Gam4}, and Lemma \ref{lem2:Gam4},
we have the following lemma.
\begin{lemm}\label{lem3:Gam4}
Let $1<p<\infty$, $f\in L_p(\uhs)^N$, and $d\in W_p^2(\uhs)$.
Then there exists
a positive constant $C$ such that for any $\la\in\Ga_4^\pm$
\begin{align*}
&\|v_\infty^{f}\|_{W_p^2(\uhs)}+\|\pi_\infty^{f}\|_{W_p^1(\uhs)}
+\|h_{A,\infty}^{f}\|_{W_p^3(\uhs)}
\leq C\|f\|_{L_p(\uhs)}, \\
&\|v_\infty^{d}\|_{W_p^2(\uhs)}+\|\pi_\infty^{d}\|_{W_p^1(\uhs)}+\|h_{A,\infty}^{d}\|_{W_p^3(\uhs)}
\leq C\| d\|_{W_p^2(\uhs)}.
\end{align*}
\end{lemm}
Applying Lemma \ref{lem3:Gam4} to the terms in (\ref{high:decomp}), we have 
\begin{align}\label{est:Gam4}
&\|(\pa_t S_\infty^{f,4}(t)F,\nabla\Pi_\infty^{f,4}(t)F)\|_{L_p(\uhs)} \notag\\
&\qquad+\|(S_\infty^{f,4}(t)F,\pa_t\CE({T}_\infty^{f,4}(t)F),\nabla\CE({T}_\infty^{f,4}(t)F))\|_{W_p^2(\uhs)}
\leq Ce^{-\ga_\infty t}\|f\|_{L_p(\uhs)},\displaybreak[0] \notag\\
&\|(\pa_t S_\infty^{d,4}(t)F,\nabla\Pi_\infty^{d,4}(t)F)\|_{L_p(\uhs)} \notag\\
&\qquad+\|(S_\infty^{d,4}(t)F,\pa_t\CE({T}_\infty^{d,4}(t)F),\nabla\CE({T}_\infty^{d,4}(t)F))\|_{W_p^2(\uhs)}
\leq Ce^{-\ga_\infty t}\| d\|_{W_p^2(\uhs)}
%
\end{align}
for any $t>0$ with some positive constant $C$.

%
%
%
%

\subsection{Analysis on $\Gamma_{5}^{\pm}$}\label{subs:high2}
By Lemma \ref{lem:symbol}, (\ref{symbol}), (\ref{1011_3}), (\ref{1011_4}), and Proposition \ref{prop:S-S}, 
we easily see that the following lemma holds.
\begin{lemm}\label{lem:Gam5}
Let $1<p<\infty$, $f\in L_p(\uhs)^N$, and $d\in W_p^2(\uhs)$.
Then there exists a positive constant $C$ such that for any $\la\in\Ga_5^\pm$
\begin{align*}
&\|(\la^{3/2}v_\infty^f,\la\nabla v_\infty^f,\nabla^2 v_\infty^f,\nabla\pi_\infty^f)\|_{L_p(\uhs)}
\leq C\|f\|_{L_p(\uhs)}, \\
&\|(\la^2 h_{A,\infty}^f,\la^{3/2}\nabla h_{A,\infty}^f,\la\nabla^2 h_{A,\infty}^f,\nabla^3 h_{A,\infty}^f)\|_{L_p(\uhs)}
\leq C\|f\|_{L_p(\uhs)},\\
&
\|(\la^{3/2} v_\infty^d,\lambda \nabla v_\infty^d,\nabla^2v_\infty^d,\nabla\pi_\infty^d)\|_{L_p(\uhs)}
\leq C\| d\|_{W_p^2(\uhs)},\\
&\|\la h_{A,\infty}^d\|_{W_p^2(\uhs)}
+\|h_{A,\infty}^d\|_{W_p^3(\uhs)}
\leq C\|d\|_{W_p^2(\uhs)}.
\end{align*}
\end{lemm}
%
%
%
%
%
%
Applying Lemma \ref{lem:Gam5} to the terms in  (\ref{high:decomp}), we have for $t\geq1$
\begin{align}\label{est:Gam5_1}
&\|(\pa_t S_\infty^{f,5}(t)F,\nabla \Pi_\infty^{f,5}(t)F)\|_{L_p(\uhs)} \notag \\
&\quad+\|(S_\infty^{f,5}(t)F,\pa_t\CE(T_\infty^{5,f}(t)F),\nabla\CE(T_\infty^{5,f}(t)F))\|_{W_p^2(\uhs)}
\leq Ce^{-\ga_\infty t}\|f\|_{L_p(\uhs)}, \notag \\
&\|(\pa_t S_\infty^{d,5}(t)F,\nabla \Pi_\infty^{d,5}(t)F)\|_{L_p(\uhs)} \notag \\
&\quad+\|(S_\infty^{d,5}(t)F,\pa_t\CE(T_\infty^{5,d}(t)F),\nabla\CE(T_\infty^{5,d}(t)F))\|_{W_p^2(\uhs)}
\leq Ce^{-\ga_\infty t}\|d\|_{W_p^2(\uhs)}
\end{align}
with some positive constant $C$.

Summing up (\ref{est:Gam4}) and (\ref{est:Gam5_1}), 
we have obtained the estimate \eqref{140917_2} in Theorem \ref{theo:main} (2),
since
%
%
%
%
\begin{align*}
&S_\infty(t)F
=\sum_{g\in\{f,d\}}\sum_{\si=4}^5 S_\infty^{g,\si}(t)F,\quad
\Pi_\infty(t)F
=\sum_{g\in\{f,d\}}\sum_{\si=4}^5 \Pi_\infty^{g,\si}(t)F, \\
&T_\infty(t)F
=\sum_{g\in\{f,d\}}\sum_{\si=4}^5 T_\infty^{g,\si}(t)F.
\end{align*}
%
%
%
%
%
%
%
%


\end{document}